\documentclass[11pt,preprint]{imsart}

\RequirePackage[numbers]{natbib}
\RequirePackage[colorlinks,citecolor=blue,urlcolor=blue]{hyperref}
\usepackage{amsmath,amsmath,amssymb,amsfonts}
\usepackage{amsthm}
\usepackage[american]{babel}
\usepackage{graphicx}
\usepackage{flafter}
\usepackage[section]{placeins}
\usepackage{float}
\usepackage{makecell}
\usepackage{comment}
\usepackage[mathscr]{eucal}
\usepackage{bbm}
\usepackage{todonotes}
\usepackage{enumerate}
\presetkeys{todonotes}{color=red!30}{linecolor=black!50}
\startlocaldefs

\def\eps{{\varepsilon}}

\def\Bbb E{\mathbb{E}}
\def\Bbb R{\mathbb{R}}
\newtheorem{corollary}{Corollary}[section]
\usepackage{mathtools}
\makeatletter \@addtoreset{equation}{section}

\makeatother

\newtheorem{lemma}{Lemma}[section]
\newtheorem{theorem}{Theorem}[section]
\newtheorem{proposition}{Proposition}[section]

\newtheorem{remark}{Remark}[section]

\font\tencmmib=cmmib10 \skewchar\tencmmib '60
\newfam\cmmibfam
\textfont\cmmibfam=\tencmmib

\font\tenmsb=msbm10 
\def\Bbb#1{\hbox{\tenmsb#1}}

\def\lessim{\ \lower4pt\hbox{$
\buildrel{\displaystyle <}\over\sim$}\ }
\def\gessim{\ \lower4pt\hbox{$\buildrel{\displaystyle >}
\over\sim$}\ }

\def\eps{\varepsilon}

\def\go0{\to 0}

\def\leftitem#1{\item{\hbox to\parindent{\enspace#1\hfill}}}

\def\sg{\sigma}

\def\sg2{\sigma^2}

\def\__{_{\infty}}

\numberwithin{equation}{section} 

\newcommand{\1}{{\rm 1}\kern-0.24em{\rm I}}

\endlocaldefs

\begin{document}

\begin{frontmatter}
\title{Estimation of trace functionals and spectral measures of covariance operators in Gaussian models}
\runtitle{Functional estimation}

\begin{aug}
\author{\fnms{Vladimir} \snm{Koltchinskii}\thanksref{t1}\ead[label=e1]{vlad@math.gatech.edu}} 
\thankstext{t1}{Supported in part by NSF grant DMS-2113121}
\runauthor{Koltchinskii}


\address{School of Mathematics\\
Georgia Institute of Technology\\
Atlanta, GA 30332-0160, USA\\
\printead{e1}\\
}
\end{aug}
\vspace{0.2cm}
{\small \today}
\vspace{0.2cm}

\begin{abstract}
Let $f:{\mathbb R}_+\mapsto {\mathbb R}$ be a smooth function with $f(0)=0.$ 
A problem of estimation of a functional $\tau_f(\Sigma):= {\rm tr}(f(\Sigma))$ of unknown covariance operator $\Sigma$ in a separable Hilbert space 
${\mathbb H}$ based on i.i.d. mean zero Gaussian observations $X_1,\dots, X_n$ with values in ${\mathbb H}$ and covariance operator $\Sigma$ is studied. Functionals $\tau_f(\Sigma)$ for a sufficiently large class of test functions $f$ define the spectral measure $\mu_{\Sigma}$ of $\Sigma$ by the following relationship: $\tau_f(\Sigma)=\int_{{\mathbb R}_+}fd\mu_{\Sigma}.$ 
Let $\hat \Sigma_n$ be the sample covariance operator based on observations $X_1,\dots, X_n.$
Estimators 
\begin{align*}
T_{f,m}(X_1,\dots, X_n):= 
\sum_{j=1}^m C_j \tau_f(\hat \Sigma_{n_j})
\end{align*}
based on linear aggregation of several plug-in estimators $\tau_f(\hat \Sigma_{n_j}),$
where the sample sizes $n/c\leq  n_1<\dots<n_m\leq n$ and coefficients $C_1,\dots, C_n$ are chosen to reduce the bias, are considered. Estimator $T_{f,m}(X_1,\dots, X_n)$ could be also represented as an integral $\int_{{\mathbb R}_+}fd\hat \mu_{n,m},$
where $\hat \mu_{n,m} := \sum_{j=1}^m C_j \mu_{\hat \Sigma_{n_j}}$ is a signed measure in ${\mathbb R}_+$ providing an estimator 
of the spectral measure $\mu_{\Sigma}.$ 
The complexity of the problem is characterized by the effective rank 
${\bf r}(\Sigma):= \frac{{\rm tr}(\Sigma)}{\|\Sigma\|}$ of covariance operator $\Sigma.$
It is shown that, if $f\in C^{m+1}({\mathbb R}_+)$ for some $m\geq 2,$ $\|f''\|_{L_{\infty}}\lesssim 1,$ 
$\|f^{(m+1)}\|_{L_{\infty}}\lesssim 1,$ $\|\Sigma\|\lessim 1$ and ${\bf r}(\Sigma)\lesssim n,$
then 
\begin{align*}
&
\|\hat T_{f,m}(X_1,\dots, X_n)-\tau_f(\Sigma)\|_{L_2} 
\lesssim_m
\frac{\|\Sigma f'(\Sigma)\|_2}{\sqrt{n}}
+
\frac{{\bf r}(\Sigma)}{n}+
{\bf r}(\Sigma)\Bigl(\sqrt{\frac{{\bf r}(\Sigma)}{n}}\Bigr)^{m+1}.
\end{align*}
Similar bounds have been proved for the $L_{p}$-errors and some other Orlicz norm errors of estimator $\hat T_{f,m}(X_1,\dots, X_n).$
The optimality of these error rates is discussed. A symmetrized (jackknife) version $\check T_{f,m}(X_1,\dots, X_n)=\int_{{\mathbb R}_+}fd\check \mu_{n,m}$ of estimator $\hat T_{f,m}(X_1,\dots, X_n)$ is also considered and, for this estimator, normal approximation bounds and asymptotic efficiency have been proved.
Finally, bounds on the sup-norms of stochastic processes $n^{1/2}\int_{{\mathbb R}_+}fd(\hat \mu_{n,m}-\mu_{\Sigma}), f\in {\mathcal F}$
and $n^{1/2}\int_{{\mathbb R}_+}fd(\check \mu_{n,m}-\mu_{\Sigma}), f\in {\mathcal F}$ for classes ${\mathcal F}$ of smooth functions $f$
as well as the results on Gaussian approximation of the second process are also discussed. 
\end{abstract}



\end{frontmatter}

\section{Introduction}
\label{Intro}

Let $X$ be a mean zero Gaussian random variable in a separable Hilbert space ${\mathbb H}$ with covariance operator $\Sigma: {\mathbb H}\mapsto {\mathbb H},$
defined as 
$\Sigma u := {\mathbb E}\langle X,u\rangle X, u\in {\mathbb H}.$ In this paper, we study the problem of estimation of functionals of the form $\tau_f(\Sigma):= {\rm tr}(f(\Sigma))$
for smooth functions $f:{\mathbb R}_+=[0,+\infty)\mapsto {\mathbb R}$ based on i.i.d. observations $X_1,\dots, X_n$ of r.v. $X\sim N(0,\Sigma).$\footnote{Here and in what follows ${\rm tr}(A)$ denotes the trace of operator $A$ in ${\mathbb H}.$} 
In what follows,  $\tau_f(\Sigma)$ will be called the {\em trace functional} generated by a test function $f.$
Note that covariance operator $\Sigma$ is a self-adjoint positively semidefinite compact operator in ${\mathbb H}$ 
(moreover, a trace class operator). Let  $\lambda_1\geq \lambda_2\geq \dots \geq 0$ be the eigenvalues of $\Sigma$
(repeated with their multiplicities). Let us define the spectral measure of $\Sigma$ as 
\begin{align*} 
\mu(B)=\mu_{\Sigma} (B):= \sum_{j\geq 1} I_B(\lambda_j)
\end{align*}
for Borel sets $B\subset {\mathbb R}_+.$
Then 
\begin{align}
\label{tau_f}
\tau_f(\Sigma)= {\rm tr}(f(\Sigma)) =  \int_{{\mathbb R}_+} f d\mu_{\Sigma}=\sum_{j\geq 1} f(\lambda_j).
\end{align} 
Note that $\tau_f(\Sigma)$ is well defined under any assumptions 
on $f$ that guarantee the absolute convergence of the series in the right hand side of \eqref{tau_f} (for instance, for any continuous function $f$ such $|f(x)|\lesssim x$ for all $x\geq 0$ in a neighborhood of $0$). If ${\mathbb H}$ is infinite-dimensional and $f$ is continuous, then the condition that $f(0)=0$ is necessary for the convergence of the series $\sum_{j\geq 1} f(\lambda_j).$ On the other hand, if ${\rm dim}({\mathbb H})=d<\infty,$ then
\begin{align*}
\tau_f (\Sigma) = f(0)d+ \tau_{\bar f}(\Sigma),
\end{align*}
where $\bar f(x):= f(x)-f(0).$ Thus, in this case, estimation of $\tau_f(\Sigma)$ could be reduced to estimation of $\tau_{\bar f}(\Sigma)$ with $\bar f(0)=0.$
In what follows, we assume without loss of generality that $f(0)=0.$

Clearly, $\tau_f(U \Sigma U^{-1})= \tau_f(\Sigma)$
for any orthogonal transformation $U:{\mathbb H}\mapsto {\mathbb H},$ so, trace functionals are important examples of orthogonally invariant functionals 
of covariance operator $\Sigma$ that depend only on the eigenvalues of $\Sigma$ (arranged in decreasing order), but not on their eigenvectors. 
The values of trace functionals for a properly chosen collection of test functions $f$ could provide relevant information about the distribution of eigenvalues 
of $\Sigma$ (more precisely, about its spectral measure $\mu_{\Sigma}$), which makes estimation of trace functionals an important problem in statistical analysis of spectral properties of covariance operators. For instance, it is of interest to develop estimators $\hat \mu_n$ of $\mu_{\Sigma}$ for which 
\begin{align*}
\|\hat \mu_n-\mu_{\Sigma}\|_{\mathcal F} := \sup_{f\in {\mathcal F}}\Bigl|\int_{{\mathbb R}_+} fd\hat \mu_n- \int_{{\mathbb R}_+} fd\mu_\Sigma\Bigr|
\end{align*}
converges to zero sufficiently fast for proper classes ${\mathcal F}$ of smooth functions $f:{\mathbb R}_+\mapsto {\mathbb R}.$

A straightforward approach to estimation of the trace functional $\tau_f(\Sigma)$ is to use the plug-in estimator $\tau_f(\hat \Sigma_n)$
based on the sample covariance operator $\hat \Sigma_n: {\mathbb H}\mapsto {\mathbb H},$ defined as
\begin{align*}
\hat \Sigma_n u := n^{-1}\sum_{j=1}^n \langle X_j, u\rangle X_j, u\in {\mathbb H}.
\end{align*} 
The statistic 
\begin{align*}
\tau_f (\hat \Sigma_n) = \sum_{j\geq 1} f(\lambda_j(\hat \Sigma_n)),
\end{align*}
where $\lambda_1(\hat \Sigma_n)\geq \lambda_2 (\hat \Sigma_n)\geq \dots $ are the eigenvalues of $\hat \Sigma_n$ repeated with their multiplicities,
is often called {\it a linear spectral statistic.} The study of asymptotic properties of linear spectral statistics of high-dimensional covariance matrices 
(as well as some other important random matrix ensembles) has been one of the central lines of research in asymptotic theory of random matrices. 
This problem was studied in a high-dimensional framework where ${\mathbb H}={\mathbb R}^d$ and the dimension $d$ of the model 
grows with the sample size $n$ so that $\frac{d}{n}\to \gamma$ as $n\to\infty,$ where $\gamma \in (0,+\infty).$ In this setting, it is natural 
to {\it normalize} the spectral measure $\mu_{\hat \Sigma_n}$ and define $\hat \nu_n:= d^{-1} \mu_{\hat \Sigma_n},$ representing 
the empirical distribution of the eigenvalues of sample covariance. In random matrix theory, the problem of weak convergence of sequences of such normalized random spectral measures $\{\hat \nu_n\}$ to a nonrandom limit probability measure has been studied in great detail since the pathbreaking paper by Marchenko and Pastur \cite{Marchenko_Pastur}. In particular, it follows from the results of \cite{Marchenko_Pastur} that, in the case when $\Sigma=I_d,$ the  
weak limit of sequence $\{\hat \nu_n\}$ is a probability measure $\nu_{\gamma}$ called {\it the Marchenko-Pastur Law} and defined as 
\begin{align*}
\nu_{\gamma}(A):= 
\begin{cases}
\tilde \nu_{\gamma}(A) & {\rm for}\ \gamma \in (0,1]\\
(1-\gamma^{-1})I_{A}(0) + \tilde \nu_{\gamma}(A) & {\rm for}\ \gamma>1,
\end{cases}
\end{align*}
where 
\begin{align*}
\tilde \nu_{\gamma}(dx) := \frac{1}{2\pi \gamma} \frac{\sqrt{(x-a)(b-x)}}{x}I_{[a,b]}(x)dx
\end{align*}
with $a:=(1-\sqrt{\gamma})^2, b:=(1+\sqrt{\gamma})^2.$ More generally, if $\Sigma=\Sigma^{(n)}$ is an arbitrary covariance 
with normalized spectral measure $\mu_n:=d^{-1}\mu_{\Sigma^{(n)}}$ such that sequence $\{\mu_n\}$ converges weakly 
to a probability measure $\mu,$ then the weak limit of normalized spectral measures $\hat \nu_n$ is a probability measure 
$\nu_{\mu, \gamma},$ depending only $\mu$ and on $\gamma,$ that coincides with the so called free multiplicative convolution 
$\mu \boxtimes \nu_{\gamma}$ of measure $\mu$ and Marchenko-Pastur Law $\nu_{\gamma}$ (see \cite{Anderson}).
Finally, fluctuations of random variables $\int_{{\mathbb R}}fd\hat \nu_n$ have been studied in detail, in particular,
it was shown that the finite dimensional distributions of stochastic process 
\begin{align*}
d\Bigl(\int_{\mathbb R} f d\hat \nu_n-\int_{\mathbb R} f\nu_{\mu_n, \gamma_n}\Bigr)= {\rm tr}(f(\hat \Sigma_n)) - d\int_{\mathbb R}fd\nu_{\mu_n,\gamma_n},
\end{align*}
where $\gamma_n := \frac{d}{n}\to \gamma,$ converge weakly to the finite-dimensional distributions of a non-degenerate Gaussian process
as $n\to\infty$ in classes of analytic and smooth functions $f$ (see \cite{Bai_S}, Chapter 9 and references therein). 
All these asymptotic results have been proved not only in the case of Gaussian models, but also for more general models.

Since normalized trace functional $d^{-1}{\rm tr}(f(\hat \Sigma_n))$ provides an estimator of $\int_{{\mathbb R}}f d_{\nu_{\mu, \gamma}}$
(or $\int_{{\mathbb R}}f d_{\nu_{\mu_n, \gamma_n}}$) rather than of $\int_{{\mathbb R}}fd\mu$ (or $\int_{{\mathbb R}}fd\mu_n$), the statistical inference about the spectrum of $\Sigma$
requires a solution of a rather difficult inverse problem of recovery of normalized spectral measure $\mu_n$ of $\Sigma=\Sigma^{(n)}$ based 
on the observations of eigenvalues of the sample covariance (which could be viewed as a multiplicative free deconvolution problem),
see \cite{El Karoui, B-G, A_T_V}.

Concentration bounds for linear spectral statistics in various random matrix models have been also extensively studied in the 
literature, see \cite{Guionnet_Zeitouni, Anderson, Chatterjee, Houdre, Meckes, Adamczak_Wolff} and references therein.
More recent results with deep connections to free probability could be found in \cite{B_B_H}.

Among simple examples of trace functionals well studied in the literature is the log determinant of covariance 
$\log {\rm det}(\Sigma)= {\rm tr}(\log (\Sigma)).$ If ${\rm dim}({\mathbb H})=d=d_n\leq n,$ then it is known that 
\begin{align*}
\frac{\log {\rm det}(\hat \Sigma_n)- a_{n,d}-\log {\rm det}(\Sigma)}{b_{n,d}}
\end{align*}
converges in distribution to a standard normal random variable for some explicit sequences of constants $a_{n,d}, b_{n,d}$ (see \cite{CLZ} and references therein).
This means that $\log {\rm det}(\hat \Sigma_n)-a_{n,d}$ is an asymptotically normal estimator of $\log {\rm det}(\Sigma)$ 
with $a_{n,d}$ providing an explicit bias correction for the plug-in estimator $\log {\rm det}(\hat \Sigma_n).$
The convergence rate of this estimator is typically of the order $\sqrt{\frac{d}{n}}.$ In this case, the problem is relatively simple
since 
\begin{align*}
\log {\rm det}(\hat \Sigma_n)- \log {\rm det}(\Sigma) = \log {\rm det}(\hat \Sigma_n^Z),
\end{align*}
where $\hat \Sigma_n^Z$ is the sample covariance based on i.i.d. standard normal observations $Z_1,\dots, Z_n\sim N(0,I_d).$

In the general case, the plug-in estimator $\tau_f(\hat \Sigma_n)$ also suffers from a large bias (even in the case when $d=o(n)$), 
which has to be reduced in order to improve the convergence rate. In the 80s-90s, Girko (see \cite{Gir, Gir1} and references therein)
developed asymptotically normal estimators of several special trace functionals of high-dimensional covariance (including the Stieltjes transform of its spectral function). His approach was based on constructing a functional $g_f$ (defined implicitly by certain equations called $G$-equations) such that $g_f(\hat \Sigma_n)$ provides a better estimation of $\tau_f(\Sigma)$ than the naive plug-in estimator $\tau_f(\hat \Sigma_n).$
The proof of asymptotic normality of these estimators relied on the martingale central limit theorem, the centering and normalizing parameters 
of the resulting CLT were hard to interpret and the estimators were not asymptotically efficient.

We will study the problem of estimation of $\tau_f(\Sigma)$ and $\mu_{\Sigma}$ in a dimension-free framework with complexity of statistical estimation characterized by the effective rank of $\Sigma:$ 
\begin{align*}
{\bf r}(\Sigma):= \frac{{\mathbb E}\|X\|^2}{\|\Sigma\|} = \frac{{\rm tr}(\Sigma)}{\|\Sigma\|}.
\end{align*}
Note that 
\begin{align*}
{\bf r}(\Sigma) \leq {\rm rank}(\Sigma) \leq {\rm dim}({\mathbb H}).
\end{align*}
If ${\rm dim}({\mathbb H})=d<\infty$ and the spectrum of covariance operator $\Sigma$ is bounded from above and bounded away 
from $0$ by numerical constants, then ${\bf r}(\Sigma)\asymp d.$ On the other hand, if many eigenvalues of $\Sigma$ are close to $0,$
${\bf r}(\Sigma)$ could be much smaller than $d.$ Note also that ${\bf r}(\Sigma)$ is always finite (even when ${\rm dim}({\mathbb H})=\infty$).

It is known (see \cite{Koltchinskii_Lounici}) that the operator norm $\|\Sigma\|$ and the effective rank ${\bf r}(\Sigma)$ characterize the size of operator 
norm error $\|\hat \Sigma_n-\Sigma\|$ of the sample covariance operator $\hat \Sigma_n.$
This is true not only for covariance operators in Hilbert spaces, but also in a more general case of separable Banach spaces.
Namely, it was proved in \cite{Koltchinskii_Lounici} that 
\begin{align}
\label{KL_1}
{\mathbb E}\|\hat \Sigma_n- \Sigma\| \asymp \|\Sigma\|\Bigl(\sqrt{\frac{{\bf r}(\Sigma)}{n}}\vee \frac{{\bf r}(\Sigma)}{n}\Bigr)
\end{align}
and that for all $t>0$ with probability at least $1-e^{-t}$
\begin{align}
\label{KL_2}
\Bigl|\|\hat \Sigma_n-\Sigma\|- {\mathbb E}\|\hat \Sigma_n-\Sigma\|\Bigr| \lesssim \|\Sigma\|\Bigl(\sqrt{\frac{{\bf r}(\Sigma)}{n}}\vee 1\Bigr)\sqrt{\frac{t}{n}} \vee \|\Sigma\|\frac{t}{n}.
\end{align}
Moreover, in \cite{Koltchinskii_Lounici_bilinear, Koltchinskii_Nickl}, the effective rank was used as a complexity parameter in the 
problems of estimation of bilinear forms of unknown covariance operators and linear forms of their eigenvectors.  

In \cite{Koltchinskii_2017, Koltchinskii_2018, Koltchinskii_Zhilova_19, Koltchinskii_2022}, a problem of estimation of H\"older smooth 
functionals $g(\Sigma)$ of unknown covariance $\Sigma$ was studied in several high-dimensional and infinite-dimensional Gaussian models.
The space of operators was equipped with the operator norm and the H\"older $C^s$-norms of the functionals were defined in terms of operator norms 
of Fr\'echet derivatives of the orders up to $s$ viewed as multilinear forms on the space of operators. In particular, functionals of the form 
$g(\Sigma)={\rm tr}(f(\Sigma)B),$ where $f:{\mathbb R}\mapsto {\mathbb R}$ is a smooth function in the real line and $B$ is an operator 
from a unit nuclear norm ball, are H\"older smooth in this sense. The final results obtained in \cite{Koltchinskii_2022} showed that the maximum 
over the H\"older $C^s$-ball $\{g: \|g\|_{C^s}\leq 1\}$ of radius $1$ of the minimax $L_2$-errors of estimation of $g(\Sigma)$
in the class of covariance operators $\Sigma$ with $\|\Sigma\|\lesssim 1$ and ${\bf r}(\Sigma)\leq r$ is of the order 
$\frac{1}{\sqrt{n}}+ (\sqrt{\frac{r}{n}})^s.$ This result implies that there is a phase transition between parametric $n^{-1/2}$
error rate and slower rates depending on the degree $s$ of smoothness of the functional.
More precisely, under the assumptions $r\leq n^{\alpha}$ for some $\alpha\in (0,1)$ and $s\geq \frac{1}{1-\alpha},$ the optimal 
error rate is $n^{-1/2}$ and, if $r\geq n^{\alpha}$ and $s<\frac{1}{1-\alpha},$ the optimal rate is slower than $n^{-1/2}$
(for some functionals in the H\"older $C^s$-ball).
The estimation method used in \cite{Koltchinskii_2022} is based on the bias reduction via linear aggregation of several plug-in estimators 
of $g(\Sigma)$ with different sample sizes (see \cite{Jiao}) and the coefficients of linear combination are chosen in such a way that the biases of plug-in estimators 
almost cancel out. These results could not be applied to trace functionals $\tau_f(\Sigma)$ since such functionals typically do not belong to the H\"older 
$C^s$-ball of radius $1$ (for instance, it is not hard to check that the operator norm of the first derivative of $\tau_f(\Sigma)$ is of the order 
of nuclear norm $\|f'(\Sigma)\|_1,$ which could be as large as ${\bf r}(\Sigma)$). However, we will show in this paper that the estimation method used 
in \cite{Koltchinskii_2022} still yields estimators of $\tau_f(\Sigma)$ with optimal error rates if function $f$ is sufficiently smooth 
(although the error rates of efficient estimators could be slower in this case and the threshold on the smoothness of $f$ for which they hold 
could be larger).

Assuming that ${\bf r}(\Sigma)\lesssim n,$ our goal is to construct estimators of $\tau_f(\Sigma)$ with nearly optimal error rates and, moreover, to achieve asymptotically efficient estimation if ${\bf r}(\Sigma)$ is small comparing with $n.$ Our analysis is primarily based on perturbation theory 
for operator functions. In particular, we use the first order Taylor expansion 
\begin{align*}
\tau_f (\hat \Sigma_n) = \langle f'(\Sigma), \hat \Sigma_n-\Sigma\rangle + R_f(\Sigma, \hat \Sigma_n-\Sigma) 
\end{align*}
of the linear spectral statistic $\tau_f(\hat \Sigma_n)$ with the linear term $\langle f'(\Sigma), \hat \Sigma_n-\Sigma\rangle$ and 
the remainder term $R_f(\Sigma, \hat \Sigma_n-\Sigma)$ to study the concentration properties of $\tau_f (\hat \Sigma_n)$ 
and to show that, under the condition ${\bf r}(\Sigma)=o(n),$ the linear term is typically dominant in the resulting concentration 
bounds.
Since $\tau_f(\hat \Sigma_n)$ is a biased estimator of $\tau_f(\Sigma)$ and the bias could have a substantial impact on the error rate 
when ${\bf r}(\Sigma)$ is relatively large, we will study in what follows estimators $\hat T_{f,m}(X_1,\dots, X_n)$ of $\tau_f(\Sigma)$ based on linear aggregation of several plug-in estimators $\tau_f(\hat \Sigma_{n_j})$ with different sample sizes $n_j, j=1, \dots, m, m\geq 2:$ 
\begin{align*}
\hat T_{f,m}(X_1,\dots, X_n):=\sum_{j=1}^m C_j \tau_f(\hat \Sigma_{n_j}). 
\end{align*}
It is possible to choose the coefficients $C_1,\dots, C_m$ of the linear combination and the sample sizes $n_1,\dots, n_m$
in such a way that the biases of plug-in estimators $\tau_f(\hat \Sigma_{n_j})$ almost cancel out, resulting in a reduced bias of estimator 
$\hat T_{f,m}(X_1,\dots, X_n):$ 
\begin{align*}
|{\mathbb E}\hat T_{f,m}(X_1,\dots, X_n)-\tau_f(\Sigma)| \lesssim_m 
\|f^{(m+1)}\|_{L_{\infty}} \|\Sigma\|^{m+1} {\bf r}(\Sigma)\Bigl(\sqrt{\frac{{\bf r}(\Sigma)}{n}}\vee \sqrt{\frac{\log n}{n}}\Bigr)^{m+1}.
\end{align*}
The proof of this bound relies on deep results in perturbation theory for trace functionals, namely, on the sharp bounds on the remainder of higher order Taylor expansions of $\tau_f(\Sigma)$ (see \cite{P_S_S}).
Using concentration properties of plug-in estimator $\tau_f(\hat \Sigma_n),$ we study the concentration of estimator $\hat T_{f,m}(X_1,\dots, X_n)$ around its expectation and use the resulting concentration inequalities 
along with the bound on the bias to get, under the assumptions that $\|\Sigma\|\lesssim 1,$ $\|f'\|_{L_{\infty}}\lesssim 1,$ $\|f'\|_{\rm Lip}\lesssim 1$ 
and $\|f^{(m+1)}\|_{L_{\infty}}\lesssim 1,$ the error rates of estimator $\hat T_{f,m}(X_1,\dots, X_n)$  in the $L_p$-norms of the order 
\begin{align*}
\frac{\|\Sigma f'(\Sigma)\|_2}{\sqrt{n}}
+
\frac{{\bf r}(\Sigma)}{n}+
{\bf r}(\Sigma)\Bigl(\sqrt{\frac{{\bf r}(\Sigma)}{n}}\Bigr)^{m+1}
\lesssim
\sqrt{\frac{{\bf r}(\Sigma^2)}{n}} +
\frac{{\bf r}(\Sigma)}{n}+
{\bf r}(\Sigma)\Bigl(\sqrt{\frac{{\bf r}(\Sigma)}{n}}\Bigr)^{m+1}.
\end{align*}
In the case when ${\bf r}(\Sigma)\leq r\lesssim n^{\alpha}$ for some $\alpha\in (0,1)$ and $f$ is sufficiently smooth (so, $m$ is sufficiently large), 
the upper bound on the error rate is  
$\sqrt{\frac{r}{n}},$ which is minimax optimal in the set of covariances $\Sigma$ with effective rank bounded by $r.$
We also construct a symmetrized (jackknife) version  $\check T_{f,m}(X_1,\dots, X_n)$ of estimator $\hat T_{f,m}(X_1,\dots, X_n)$
and obtain bounds on the accuracy of approximation of  $\check T_{f,m}(X_1,\dots, X_n)-{\mathbb E}\check T_{f,m}(X_1,\dots, X_n)$ 
by the first order linear term $\langle f'(\Sigma), \hat \Sigma_n-\Sigma\rangle$ of Taylor expansion of plug-in estimator $\tau_f(\hat \Sigma_n).$
These bounds allow us to study the accuracy of normal approximation of estimator $\check T_{f,m}(X_1,\dots, X_n)$ and to show 
its asymptotic efficiency with the $L_2$-error rate $(\sqrt{2}+o(1))\frac{\|\Sigma f'(\Sigma)\|_2}{\sqrt{n}}.$

The construction of the estimators and precise statement of the results are given in Section \ref{Main_results}, where we also introduce the corresponding estimators $\hat \mu_{n,m}$ and $\check \mu_{n,m}$ of the spectral measure $\mu_{\Sigma},$ obtain bounds on 
$\|\hat \mu_{n,m}-\mu_{\Sigma}\|_{\mathcal F}$ and $\|\check \mu_{n,m}-\mu_{\Sigma}\|_{\mathcal F}$ and provide a Gaussian approximation 
for a stochastic process $\sqrt{n}\int_{{\mathbb R}_+} f d(\check \mu_{n,m}-\mu_{\Sigma}), f\in {\mathcal F}$
over certain classes ${\mathcal F}$ of smooth functions. In Section \ref{Conc}, we provide concentration bounds for plug-in estimator $\tau_f(\hat \Sigma_n)$
and for the remainder of its first order Taylor expansion around $\Sigma$ as well as the normal approximation for the first order linear term 
$\langle f'(\Sigma), \hat \Sigma_n-\Sigma\rangle$ of Taylor expansion of $\tau_f(\hat \Sigma_n),$ whereas the uniform versions of these 
results over classes ${\mathcal F}$ of smooth functions are provided in Section \ref{sup_bounds}. In Section \ref{bias_bounds}, we use 
the results from the operator theory on higher order Taylor expansions of traces of operator functions to obtain a decomposition of the bias of plug-in estimator $\tau_f(\hat \Sigma_n),$ yielding the bounds on the bias of $\hat T_{f,m}(X_1,\dots, X_n).$
The proofs of the main results are given in Section \ref{proof_of_main} and the proofs of the lower bounds are given in Section \ref{proof_lower}.

In the rest of this section, we introduce some notations used throughout the paper (some of them have already been used above).

For nonnegative real variables $A,B,$ the notation $A\lesssim B$ means that there exists a numerical constant $C>0$ such that $A\leq CB,$ 
the notation $A\gtrsim B$ means that $B\lesssim A$ and the notation $A\asymp B$ means that $A\lesssim B$ and $A\gtrsim B.$
The signs $\lesssim, \gtrsim$ and $\asymp$ will be provided with subscripts to indicate the dependence of the constants on certain 
parameters, say, $A\lesssim_{\gamma, \beta} B$ means that there exists $C_{\gamma, \beta}>0$ such that $A\leq C_{\gamma,\beta} B.$

The norm notation $\|\cdot\|$ (without any subscripts) is used for the norm of the underlying Hilbert space ${\mathbb H},$ for the operator
norms of linear operators acting in ${\mathbb H}$ and, occasionally, for other norms (if there is no ambiguity). If there is an ambiguity, the norms 
will be provided with subscripts. 

The notation $\|A\|_p$ for $p\geq 1$ means the Schatten $p$-norm of operator $A$ in Hilbert space 
${\mathbb H}:$ $\|A\|_{p}:= {\rm tr}^{1/p}(|A|^p),$ where $|A|:= \sqrt{A^{\ast}A}.$ In particular, $\|A\|_1$ is the nuclear norm, $\|A\|_2$ is the Hilbert--Schmidt norm and $\|A\|_{\infty}=\|A\|$
is the operator norm of $A.$ For $p\geq 1,$ the Schatten class ${\mathcal S}_p$ consists of all linear operators $A$ in ${\mathbb H}$ with 
$\|A\|_p<\infty.$ In particular, ${\mathcal S}_2$ is the space of all Hilbert--Schmidt operators. It is equipped with the Hilbert--Schmidt inner 
product $\langle A, B\rangle= {\rm tr}(A^{\ast}B)$ (for which we use the same notation as for the inner product of Hilbert space ${\mathbb H}$).
The notation $A\preceq B$ for self-adjoint operators $A,B$ means that $B-A$ is positively 
semidefinite. 
For vectors $x,y\in {\mathbb H},$ $x\otimes y$ is the tensor product of $x$ and $y.$ It is the operator in ${\mathbb H}$ defined as 
follows: $(x\otimes y)u = x \langle y,u\rangle, u\in {\mathbb H}.$

In what follows, $C^k ({\mathbb R}^+)$ denotes the space of all functions $f:{\mathbb R}_+\mapsto {\mathbb R}$ that are $k$ times differentiable 
in $(0,\infty)$ with all the derivatives $f^{(j)}, j=0,\dots, k$ being continuous functions in ${\mathbb R}_+.$ Such functions can be extended to $k$ times 
continuously differentiable functions in ${\mathbb R},$ and we will use such an extension whenever it is needed.  
We will use the notations 
\begin{align*}
\|f\|_{L_{\infty}}:= \sup_{x\in {\mathbb R}}|f(x)|\ {\rm and}\ \|f\|_{\rm Lip}:= \sup_{x\neq y}\frac{|f(x)-f(y)|}{|x-y|}.
\end{align*}
Similar norms defined for functions on a subset $B\subset {\mathbb R}$ will be denoted by $\|f\|_{L_{\infty}(B)}$ and $\|f\|_{{\rm Lip}(B)}.$

We will use below Orlicz norms of random variables on a probability space $(\Omega, {\mathcal A}, {\mathbb P}).$
Given a convex increasing function $\psi:{\mathbb R}_+\mapsto {\mathbb R}_+$ with $\psi(0)=0,$ define the $\psi$-norm 
of a random variable $\xi:\Omega \mapsto {\mathbb R}$ as 
\begin{align*}
\|\xi\|_{\psi} := \inf\Bigl\{c>0: {\mathbb E}\psi\Bigl(\frac{|\xi|}{c}\Bigr)\leq 1\Bigr\}.
\end{align*}  
Denote also $L_{\psi}({\mathbb P}):= \{\xi: \|\xi\|_{\psi}<\infty\}.$ Sometimes, it is convenient to use the notations $\|\cdot\|_{L_{\psi}}$
or $\|\cdot\|_{L_{\psi}({\mathbb P})}$ for the $\psi$-norm. If $\psi(u):=u^{p}, u\geq 0, p\geq 1,$ then the $\psi$-norm coincides with the $L_p$-norm
and $L_{\psi}({\mathbb P})=L_p({\mathbb P}).$ One can also consider functions $\psi_{\alpha}(u)= e^{u^{\alpha}}-1, u\geq 0, \alpha\geq 1.$
The corresponding $\psi_{\alpha}$-norms describe various types of exponential decay of the tails of r.v. $\xi.$ In particular, random variables 
in the space $L_{\psi_1}$ have subexponential tails and random variables in the space $L_{\psi_2}$ have subgaussian tails.
It is also well known that, for all $\alpha\geq 1,$ 
\begin{align}
\label{Orlicz_L_p}
\|\xi\|_{\psi_{\alpha}} \asymp \sup_{p\geq 1} p^{-1/\alpha} \|\xi\|_{L_p}.
\end{align}
Thus, one can define the equivalent norm to the $\psi_{\alpha}$-norm by the right hand side of \eqref{Orlicz_L_p},
characterizing the growth rate of the $L_p$-norms. 
Note also that, for $\alpha\in (0,1),$ the function $\psi_{\alpha}$ is not convex and $\|\cdot\|_{\psi_{\alpha}}$
is not a norm. However, two sided bound \eqref{Orlicz_L_p} still holds in this case and the right hand side of \eqref{Orlicz_L_p}
is still a norm. It will be convenient for our purposes to use it as an alternative definition of the $\psi_{\alpha}$-norm that holds in the whole 
range of $\alpha>0.$ It will be also convenient to extend the definition of $L_p$- and $\psi_{\alpha}$-norms to arbitrary functions 
$\xi:\Omega\mapsto {\mathbb R}.$ For instance, the $\psi_{\alpha}$-norm of such non-measurable random variable $\xi$ could be defined as
\begin{align*}
\|\xi\|_{\psi_{\alpha}}^{\ast} := \inf\Bigl\{\|\bar \xi\|_{\psi_{\alpha}}: |\bar \xi| \geq |\xi|, \bar \xi\ {\rm is\ a\ random\ variable}\Bigr\}.
\end{align*}
Similarly, we can define the outer expectation of nonnegative function $\xi:{\Omega}\mapsto {\mathbb R}$
as 
\begin{align*}
{\mathbb E}^{\ast} \xi := \inf\Bigl\{{\mathbb E}\bar \xi: \bar \xi \geq \xi, \bar \xi\ {\rm is\ a\ random\ variable}\Bigr\}.
\end{align*}
In particular, such notations are useful when we deal with expectations (or Orlicz norms) of the sup-norm 
of a stochastic process, since it allows as to avoid measurability issues. This approach is common in empirical 
processes literature.  
In what follows, we drop $^\ast$ sign from outer expectations and outer norms even in the case when they are applied to non-measurable 
random variables.

We will also use in what follows Wasserstein type distances between random variables (or, more precisely, between their distributions).
Namely, for random variables $\xi$ and $\eta,$ define 
\begin{align*}
W_p(\xi, \eta):= \inf\Bigl\{\|\bar \xi-\bar \eta\|_{L_p}: \bar \xi\overset{d}{=}\xi,  \bar \eta\overset{d}{=}\eta\Bigr\}, p\geq 1,
\end{align*}
where the infimum is taken over all random variables $\bar \xi, \bar \eta$ defined on probability space $(\Omega,{\mathcal A}, {\mathbb P})$
satisfying the constraints $\bar \xi\overset{d}{=}\xi,  \bar \eta\overset{d}{=}\eta.$ Similarly, for two stochastic processes 
$\xi, \eta: T\mapsto {\mathbb R},$ define
\begin{align*}
{\mathcal W}_{T,p}(\xi, \eta):= \inf\Bigl\{\Bigl\|\sup_{t\in T}|\bar \xi(t)-\bar \eta(t)|\Bigr\|_{L_p}: \bar \xi\overset{f.d.d.}{=}\xi,  \bar \eta\overset{f.d.d.}{=}\eta\Bigr\}, p\geq 1,
\end{align*}
where the infimum is taken over all stochastic processes $\bar \xi(t), t\in T$ and $\bar \eta(t), t\in T$ defined on probability 
space $(\Omega,{\mathcal A}, {\mathbb P})$ such that
$\bar \xi$ has the same finite dimensional distributions as $\xi$ and $\bar \eta$ has the same finite dimensional distributions as $\eta.$
Replacing the $L_p$-norm in the above definitions by the $\psi_{\alpha}$-norm, one can define the Wasserstein $\psi_{\alpha}$-distances 
$W_p$ and ${\mathcal W}_{T,p}$ that will be also used in what follows.

\section{Main results}
\label{Main_results}

Let $G_{\Sigma}(f), f\in C^1({\mathbb R}_+)$ be a centered Gaussian process with covariance function
\begin{align*}
{\mathbb E}G_{\Sigma}(f)G_{\Sigma}(g) := {\rm tr}(\Sigma^2 f'(\Sigma)g'(\Sigma)), f, g\in C^1({\mathbb R}_+).
\end{align*}

We start this section with the following simple result (which is an immediate consequence of Proposition \ref{Main_m=2_detailed} in Section \ref{Conc}).

\begin{proposition}
\label{Main_m=2}
Let $f\in C^1({\mathbb R}_+)$ with $f(0)=0$ and $\|f'\|_{\rm Lip}<\infty.$
Then, for all $p\geq 1,$
\begin{align*}
\Bigl\|\tau_f(\hat \Sigma_n)-\tau_f(\Sigma)- \langle f'(\Sigma), \hat \Sigma_n-\Sigma\rangle\Bigr\|_{L_{\psi_{1/2}}}
\lesssim \|f'\|_{\rm Lip} \|\Sigma\|^2 \frac{{\bf r}(\Sigma)^2}{n}.
\end{align*}
If, in addition, $\|f'\|_{L_{\infty}}<\infty,$ this implies that 
\begin{align*}
\|\tau_f(\hat \Sigma_n)-\tau_f(\Sigma)\|_{L_{\psi_{1/2}}}
&\lesssim \frac{\|\Sigma f'(\Sigma)\|_2}{\sqrt{n}}+ 
\|f'\|_{\rm Lip} \|\Sigma\|^2 \frac{{\bf r}(\Sigma)^2}{n}
\\
&
\lesssim  
\|f'\|_{L_{\infty}} \|\Sigma\|\sqrt{\frac{r(\Sigma^2)}{n}}+ 
\|f'\|_{\rm Lip} \|\Sigma\|^2 \frac{{\bf r}(\Sigma)^2}{n},
\end{align*}
and, moreover, 
\begin{align*}
W_{\psi_{1/2}} \Bigl(\sqrt{n/2}(\tau_f(\hat \Sigma_n)-\tau_f(\Sigma)), G_{\Sigma}(f)\Bigr)
&\lesssim \frac{\|\Sigma f'(\Sigma)\|_2}{\sqrt{n}}
+\|f'\|_{\rm Lip} \|\Sigma\|^2 \frac{{\bf r}(\Sigma)^2}{\sqrt{n}}
\\
&
\lesssim \|f'\|_{L_{\infty}} \|\Sigma\|\sqrt{\frac{{\bf r}(\Sigma^2)}{n}}
+\|f'\|_{\rm Lip} \|\Sigma\|^2 \frac{{\bf r}(\Sigma)^2}{\sqrt{n}}.
\end{align*}
\end{proposition}

Let ${\mathcal F}_{m}\subset C^{m+1}({\mathbb R}_+)$ be the class of all functions $f\in C^{m+1}({\mathbb R}_+)$ such that $f(0)=0$ and  
\begin{align*}
\max_{1\leq j\leq m+1}\|f^{(j)}\|_{L_{\infty}}\leq 1.
\end{align*}
Consider the following stochastic process 
\begin{align*}
\tilde G_n (f) := \sqrt{n/2} \int_{{\mathbb R}_+} f d(\mu_{\hat \Sigma_n}-\mu_{\Sigma})= \sqrt{n/2}(\tau_f(\hat \Sigma_n)-\tau_f(\Sigma)), f\in {\mathcal F}_1. 
\end{align*}
The next proposition provides bounds on the sup-norm of $\mu_{\hat \Sigma_n}-\mu_{\Sigma}$ over the class ${\mathcal F}_1$ as well 
as the approximation of stochastic process $\tilde G_n$ by the Gaussian process $G_{\Sigma}$ in the Wasserstein 
${\mathcal W}_{{\mathcal F}_1, \psi_{1/2}}$-distance (it follows from Proposition \ref{emp_pr_first_detail} of Section \ref{sup_bounds}).

\begin{proposition}
\label{emp_pr_first}
Suppose that $\|\Sigma\|\lesssim 1$ and ${\bf r}(\Sigma)\lesssim n.$ 
Then 
\begin{align}
&
\Big\|\|\hat \mu_{\hat \Sigma_n} - \mu_{\Sigma}\|_{{\mathcal F}_1}\Bigr\|_{L_{\psi_{1/2}}}
\lesssim \sqrt{\frac{{\bf r}(\Sigma^2)}{n}}
+ \frac{{\bf r}(\Sigma)^2}{n}
\end{align}
and 
\begin{align*}
&
\nonumber
{\mathcal W}_{{\mathcal F}_1, \psi_{1/2}} (\tilde G_n, G_{\Sigma})
\lesssim
\sqrt{\frac{{\bf r}(\Sigma^2)}{n}}+ \frac{{\bf r}(\Sigma)^2}{\sqrt{n}}.
\end{align*}
\end{proposition}

Note that the ``remainder terms" of the order $\frac{{\bf r}(\Sigma)^2}{n}$ in the first two bounds of Proposition \ref{Main_m=2}
($\frac{{\bf r}(\Sigma)^2}{\sqrt{n}}$ in the last bound) and similar terms in the bounds of Proposition \ref{emp_pr_first} could be 
prohibitively large. For instance, to make the remainder term negligible and to achieve the overall error rate of the order $\sqrt{\frac{{\bf r}(\Sigma)}{n}}$
(for which there exists a minimax lower bound, see Proposition \ref{min_max_lower_eff_rank}) for the plug-in estimator 
$\tau_f(\hat \Sigma_n),$ the condition ${\bf r}(\Sigma)\lesssim n^{1/3}$ should be satisfied. Even stronger conditions on the effective rank 
are needed to deduce the asymptotic normality of properly normalized linear spectral statistic  
$$\frac{\sqrt{n}(\tau_f(\hat \Sigma_n)-\tau_f(\Sigma))}{\sqrt{2}\|\Sigma f'(\Sigma)\|_2}$$ 
from the last bound of Proposition \ref{Main_m=2} and to establish the asymptotic efficiency of $\tau_f(\hat \Sigma_n).$

We will show that the remainder terms in the bounds of this type could be substantially improved for certain estimators of $\tau_f(\Sigma)$
with reduced bias, provided that function $f$ is sufficiently smooth.
To construct such estimators of trace functional $\tau_f(\Sigma),$ we will use a method of bias reduction based on linear aggregation of plug-in estimators with different 
samples sizes \cite{Jiao}. The sample sizes of plug-in estimators and the coefficients of their linear combination are chosen to ensure that the biases of plug-in estimators 
almost cancel out.
Namely, given $m\geq 2,$ let $n_1,\dots, n_m$ be the sample sizes of plug-in estimators. Denote $\vec{n}:=(n_1,\dots, n_m).$ Assume that, for some $c>1,$
 $n/c\leq n_1<\dots<n_m\leq n.$ Let 
 \begin{align*}
 \hat T_{f}(X_1,\dots, X_n) =\hat T_{f,m}(X_1,\dots, X_n)=\hat T_{f,\vec{n}}(X_1,\dots, X_n):=\sum_{j=1}^m C_j \tau_f(\hat \Sigma_{n_j}), 
 \end{align*}
where the coefficients $C_1,\dots C_m$ provide the unique solution of the following system of $m$ linear 
equations: 
\begin{align}
\label{cond_C_j_1}
\sum_{j=1}^m C_j=1
\end{align}
and 
\begin{align}
\label{cond_C_j_2}
\sum_{j=1}^m \frac{C_j}{n_j^l}=0, l=1,\dots, m-1.
\end{align}
It is easy to see that 
\begin{align*}
C_j:= \prod_{i\neq j} \frac{n_j}{n_j-n_i}, j=1,\dots, m.
\end{align*}
We will also assume that the following assumption holds:
\begin{align}
\label{assume_C_j}
\sum_{j=1}^m |C_j| \lesssim_m 1.
\end{align}
For this, it is necessary that $n_{j+1}-n_j\asymp n.$ A possible choice of the sample sizes is $n_j=q^{j-m}n,$ $j=1,\dots, m$ for some $q>1.$
The following signed measure
\begin{align*}
\hat \mu_n = \hat \mu_{n,m}=\hat \mu_{n,\vec{n}}:=\sum_{j=1}^m C_j \mu_{\hat \Sigma_{n_j}}
\end{align*}
will be used as an estimator of spectral measure $\mu_{\Sigma}.$ Clearly,
\begin{align*}
 \hat T_{f, \vec{n}}(X_1,\dots, X_n)=\int_{{\mathbb R}_+} fd\hat \mu_{n,\vec{n}}.
 \end{align*}

In what follows, we assume that $\|\Sigma\|$ is bounded from above by a numerical constant whereas the effective 
rank ${\bf r}(\Sigma)$ could be large. 
We state the main results of the paper in a somewhat simplified form in terms of $\psi_{1/2}$-norms of the errors  
and the corresponding Wasserstein distances used for Gaussian approximation. 
Since $\|\xi\|_{L_p}\lesssim p^2 \|\xi\|_{\psi_{1/2}},$ these results also imply the bounds 
on the $L_p$-norms for arbitrary $p\geq 1.$
More detailed and somewhat more complicated versions 
of the main theorems including the bounds on the $L_p$-norms (and Wasserstein $W_p$ distances) with more precise and more explicit 
dependence on the value of $p$ and on $\|\Sigma\|$  will be stated and proved in Section \ref{proof_of_main}.

\begin{theorem}
\label{Th_M_YYY}
Suppose $f\in C^{m+1}({\mathbb R}_{+})$ for some $m\geq 2,$ $f(0)=0,$ $\|f'\|_{{\rm Lip}}\lesssim 1$ and $\|f^{(m+1)}\|_{L_{\infty}}\lesssim 1.$ 
Suppose also 
that $\|\Sigma\|\lesssim 1$ and ${\bf r}(\Sigma)\lesssim n.$ Then 
\begin{align*}
&
\|\hat T_{f,m}(X_1,\dots, X_n)-\tau_f(\Sigma)\|_{L_{\psi_{1/2}}} 
\lesssim_m
\frac{\|\Sigma f'(\Sigma)\|_2}{\sqrt{n}}
+
\frac{{\bf r}(\Sigma)}{n}+
{\bf r}(\Sigma)\Bigl(\sqrt{\frac{{\bf r}(\Sigma)}{n}}\Bigr)^{m+1}.
\end{align*}
\end{theorem}

\begin{remark}
\label{rem_on_eff_rank}
\normalfont
Note that in the case when the function $f'$ is uniformly bounded from above and bounded away from zero by numerical constants,
we have 
\begin{align*} 
\|\Sigma f'(\Sigma)\|_2 = \Bigl(\sum_{k\geq 1} \lambda_k^2 f'(\lambda_k)^2\Bigr)^{1/2}\asymp \|\Sigma\|_2 = \|\Sigma\|\sqrt{{\bf r}(\Sigma^2)},
\end{align*}
where $\lambda_1\geq \lambda_2\geq \dots \geq 0$ are the eigenvalues of $\Sigma.$
Thus, the upper bound of Theorem \ref{Th_M_YYY} could be rewritten as 
\begin{align}
\label{bd_eff_ranks}
\|\hat T_{f,m}(X_1,\dots, X_n)-\tau_f(\Sigma)\|_{L_{\psi_{1/2}}} 
\lesssim_m
\sqrt{\frac{{\bf r}(\Sigma^2)}{n}} +
\frac{{\bf r}(\Sigma)}{n}+
{\bf r}(\Sigma)\Bigl(\sqrt{\frac{{\bf r}(\Sigma)}{n}}\Bigr)^{m+1}.
\end{align}
In particular, it follows from this result that the upper bound on the error rate is of the order $\sqrt{\frac{{\bf r}(\Sigma)}{n}},$ provided that 
${\bf r}(\Sigma)\lesssim n^{m/(m+2)}$ (in Proposition \ref{lower_first_term} below, we show a minimax bound of the order $\sqrt{\frac{r}{n}}$
in the class of covariance operators with ${\bf r}(\Sigma)\leq r$).
\end{remark}

\begin{remark}
\label{rem_on_eff_rank_02}
\normalfont
To better understand how this bound could depend on the rate of decay of the eigenvalues of $\Sigma$ and on the degree of smoothness of function $f,$ assume that ${\rm dim}({\mathbb H})=d\lesssim n.$ If the eigenvalues of $\Sigma$ are uniformly 
bounded from above and bounded away from zero by numerical constants (so that $\lambda_k\asymp 1, k=1,\dots, d$),
then ${\bf r}(\Sigma)\asymp d$ and ${\bf r}(\Sigma^2)\asymp d$ implying that 
\begin{align}
\label{bd_d_spec}
\|\hat T_{f,m}(X_1,\dots, X_n)-\tau_f(\Sigma)\|_{L_{\psi_{1/2}}} 
\lesssim_m  \sqrt{\frac{d}{n}} 
+ d\Bigl(\sqrt{\frac{d}{n}}\Bigr)^{m+1}.
\end{align}
Therefore, if $d\lesssim n^{\alpha}$ for some $\alpha\in (0,1)$ and $m+1\geq \frac{1+\alpha}{1-\alpha},$ then the first term in 
the right hand side of bound \eqref{bd_d_spec} is dominant and the bound becomes 
\begin{align*}
\|\hat T_{f,m}(X_1,\dots, X_n)-\tau_f(\Sigma)\|_{L_{\psi_{1/2}}}  \lesssim_m  \sqrt{\frac{d}{n}},
\end{align*}
otherwise the second term is dominant and the error rate is slower than $\sqrt{\frac{d}{n}}.$ 
Faster error rates are possible when the eigenvalues $\lambda_k$ decay with certain rate. Assume, for instance,  that, for some $\beta>0,$  
\begin{align*}
\lambda_k \asymp k^{-\beta}, k=1,\dots, d.
\end{align*}
In this case 
\begin{align*}
{\bf r}(\Sigma) \asymp \sum_{k=1}^d k^{-\beta} \asymp 
\begin{cases}
d^{1-\beta} & {\rm for}\ \beta<1\\
\log d          & {\rm for}\ \beta=1\\
1                 & {\rm for}\ \beta>1.
\end{cases}
\end{align*}
On the other hand, 
\begin{align*}
{\bf r}(\Sigma^2) \asymp \sum_{k=1}^d k^{-2\beta} \asymp 
\begin{cases}
d^{1-2\beta} & {\rm for}\ \beta<1/2\\
\log d          & {\rm for}\ \beta=1/2\\
1                 & {\rm for}\ \beta>1/2.
\end{cases}
\end{align*}
Therefore, we can easily get from bound \eqref{bd_eff_ranks} that, for $m\geq 2,$ 
\begin{align*}
\|\hat T_{f,m}(X_1,\dots, X_n)-\tau_f(\Sigma)\|_{L_{\psi_{1/2}}} 
\lesssim_m 
\begin{cases}  
\frac{d^{1/2-\beta}}{n^{1/2}} + \frac{d^{(1-\beta)(m+3)/2}}{n^{(m+1)/2}} & {\rm for}\ \beta<1/2\\
\sqrt{\frac{\log d}{n}} + \frac{d^{(m+3)/4}}{n^{(m+1)/2}} & {\rm for}\ \beta=1/2\\
\frac{1}{\sqrt{n}} +\frac{d^{(1-\beta)(m+3)/2}}{n^{(m+1)/2}} & {\rm for}\ \beta\in (1/2,1)\\
\frac{1}{\sqrt{n}}   & {\rm for}\ \beta\geq 1.\\
\end{cases}
\end{align*}
Comparing the last bound with \eqref{bd_d_spec}, one can see that, in this case, consistent estimation could be possible 
even when $d$ is much larger than $n$ (depending, of course, on the values of $\beta$ and $m$). 
\end{remark}

\begin{remark}
\normalfont
\label{rem_on_eff_rank_03}
Even faster rates could be possible and could be easily derived from the bound of Theorem \ref {Th_M_YYY} when $f'(\lambda)\to 0$ as $\lambda\to 0.$ 
\end{remark}

\begin{remark}
\normalfont
\label{rem_on_eff_rank_04}
It is not hard to check that the bound of Theorem \ref{Th_M_YYY} and bound \eqref{bd_eff_ranks} also 
hold for any function $f\in C^{m+1}({\mathbb R}_{+})$ such that $f(0)=0$ with constants in the inequalities $\lesssim$ depending on $m$ and on function $f$ through 
$\|f'\|_{{\rm Lip}([0,A])}$ and $\|f^{(m+1)}\|_{L_{\infty}([0,A])}$ for $A=2\|\Sigma\|.$ 
\end{remark}

Although, at the moment, we could not prove minimax optimality of the error rate of Theorem \ref{Th_M_YYY} for a given function $f:{\mathbb R}_+\mapsto {\mathbb R},$ we will provide some partial results in this direction. 

Let ${\mathcal S}(a,r):= \{\Sigma: \|\Sigma\|\leq a, {\bf r}(\Sigma)\leq r\},\ a>0, r\geq 1.$
We start with the following rather simple proposition.

\begin{proposition}
\label{lower_first_term}
Let $f$ be a continuously differentiable function in ${\mathbb R}_+$ with $f(0)=0$ and suppose that there exist numbers 
$0\leq \gamma_1< \gamma_2\leq 1$ such that $|f'(x)|\geq \lambda>0$ for all $x\in [\gamma_1 a, \gamma_2 a]$ and 
for some constant $\lambda>0.$ Then 
\begin{align*}
\inf_{T_{n,f}}\sup_{\Sigma\in {\mathcal S}(a,r)}{\mathbb E}_{\Sigma}^{1/2}(T_{n,f}(X_1,\dots, X_n)-\tau_{f}(\Sigma))^2
\gtrsim_{\gamma_1,\gamma_2, \lambda} \ a\sqrt{\frac{r}{n}},
\end{align*}
where the infimum is taken over all the estimators $T_{n,f}(X_1,\dots, X_n)$ of $\tau_f(\Sigma).$
\end{proposition}

We will also describe a somewhat artificial class of functional estimation problems where the bound of  Theorem \ref{Th_M_YYY} yields error rates for which some minimax optimality claims could be made. Namely, consider a problem of estimation of a ``piecewise trace functional" $g:D_g \mapsto {\mathbb R}$ defined on a set 
of covariance operators 
$
D_g := \bigcup_{j=1}^N B(\Sigma_j, \delta),
$
where $N\geq 1,$ $\Sigma_j, j=1,\dots, N$ are given covariance operators such that $\|\Sigma_i-\Sigma_j\|\geq 4\delta, i\neq j$
and $B(\Sigma_j,\delta):= \{\Sigma: \|\Sigma-\Sigma_j\|<\delta\}, j=1,\dots, N$ are operator norm balls with centers $\Sigma_j$ of radius $\delta>0,$
and given by the following formula
\begin{align}
\label{repr_g_D_g}
g(\Sigma):= \sum_{j=1}^N I_{B(\Sigma_j,\delta)}(\Sigma) \tau_{f_j}(\Sigma), \Sigma\in D_g
\end{align}
where, for all $j=1,\dots, N,$ 
\begin{align*}
f_j \in {\mathcal H}_m:=\{f\in C^{m+1}({\mathbb R}): f(0)=0, \|f'\|_{L_{\infty}}\leq 1, \|f''\|_{L_{\infty}}\leq 1, \|f^{(m+1)}\|_{L_{\infty}}\leq 1\}.
\end{align*}
Let ${\mathcal G}_{\delta}$ be the class of all such ``piecewise trace functionals" $g$ for a given $\delta>0,$ an arbitrary $N\geq 1,$
arbitrary $\Sigma_1,\dots, \Sigma_N$ and arbitrary functions $f_1, \dots, f_N\in {\mathcal H}_m.$ 
For a functional $g:D_g \mapsto {\mathbb R}$ with representation \eqref{repr_g_D_g}, consider the following 
estimator $\tilde T_g(X_1,\dots, X_n)$ of $g(\Sigma)$ based on i.i.d. observations $X_1,\dots, X_n\sim N(0,\Sigma):$
\begin{align*}
\tilde T_g(X_1,\dots, X_n):= \sum_{j=1}^N I_{B(\Sigma_j,\delta)}(\hat \Sigma_n) \hat T_{f_j, m}(X_1,\dots, X_n).
\end{align*}

\begin{proposition}
\label{upper_piecewise}
Suppose that $a\lesssim 1,$ $r\leq n$ and $\delta \geq C\sqrt{\frac{r}{n}}$ for a sufficiently large numerical constant $C>0.$
Then 
\begin{align*}
\sup_{g\in {\mathcal G}_{\delta}}\sup_{\Sigma\in {\mathcal S}(a,r)\cap D_g}{\mathbb E}_{\Sigma}^{1/2}(\tilde T_g(X_1,\dots, X_n)-g(\Sigma))^2
&
\lesssim_{m,a} \sqrt{\frac{r}{n}} + r \Bigl(\sqrt{\frac{r}{n}}\Bigr)^{m+1} + r \exp\{-n(\delta\wedge \delta^2)/2\}
\\
&
\lesssim_{m,a} \sqrt{\frac{r}{n}} + r \Bigl(\sqrt{\frac{r}{n}}\Bigr)^{m+1} + (\delta^{-1}\vee \delta^{-2})\frac{r}{n}.
\end{align*}
\end{proposition} 

On the other hand, the following minimax lower bounds for estimation of functionals from class ${\mathcal G}_{\delta}.$ hold.

\begin{proposition}
\label{min_max_lower_eff_rank}
Suppose that $a\asymp 1$ and $r\leq n.$ If $\delta \gtrsim n^{-1/2},$ then 
\begin{align*}
\sup_{g\in {\mathcal G}_{\delta}}\inf_{T_{n,g}}\sup_{\Sigma\in {\mathcal S}(a,r)\cap D_g}{\mathbb E}_{\Sigma}^{1/2}(T_{n,g}(X_1,\dots, X_n)-g(\Sigma))^2
\gtrsim_{m,a} \sqrt{\frac{r}{n}}, 
\end{align*}
where the infimum is taken over all the estimators $T_{n,g}(X_1,\dots, X_n)$ of $g(\Sigma).$
Moreover, if $\delta = c\sqrt{\frac{r}{n}}$ with a sufficiently small numerical constant $c>0,$
then 
\begin{align*}
\sup_{g\in {\mathcal G}_{\delta}}\inf_{T_{n,g}}\sup_{\Sigma\in {\mathcal S}(a,r)\cap D_g}{\mathbb E}_{\Sigma}^{1/2}(T_{n,g}(X_1,\dots, X_n)-g(\Sigma))^2
\gtrsim_{m,a} \sqrt{\frac{r}{n}} + r \Bigl(\sqrt{\frac{r}{n}}\Bigr)^{m+1}.
\end{align*}
\end{proposition}

\begin{remark}
\normalfont
Despite the fact that propositions \ref{upper_piecewise} and \ref{min_max_lower_eff_rank} shed some light on minimax rates 
of estimation of piecewise trace functionals, they do not provide a definitive solution of this problem. To do this, one would have 
to prove minimax lower bounds for larger values of $\delta\geq C\sqrt{\frac{r}{n}}$ and/or develop estimators and prove upper 
bounds for smaller values of $\delta \leq c\sqrt{\frac{r}{n}}$ (and also to understand minimax rates for intermediate values 
of $\delta \in [c\sqrt{\frac{r}{n}}, C\sqrt{\frac{r}{n}}]$).
\end{remark}

The next result provides a bound on the estimation error of spectral measure $\mu_{\Sigma}$ by the measure $\hat \mu_{n,m}$ in the sup-norm over the class ${\mathcal F}_m$ of smooth functions.

\begin{theorem}
\label{Th_M1_YYY}
Suppose that $\|\Sigma\|\lesssim 1$ and ${\bf r}(\Sigma)\lesssim n.$ 
Then, for all $m\geq 2,$
\begin{align}
&
\Big\|\|\hat \mu_{n,m} - \mu_{\Sigma}\|_{{\mathcal F}_m}\Bigr\|_{L_{\psi_{1/2}}}
\lesssim_m \sqrt{\frac{{\bf r}(\Sigma^2)}{n}}
+
\frac{{\bf r}(\Sigma)}{n}
+  {\bf r}(\Sigma)\Bigl(\sqrt{\frac{{\bf r}(\Sigma)}{n}}\Bigr)^{m+1}.
\end{align}
\end{theorem}

We will now turn to another estimator of trace functional which is also based on an approach discussed in \cite{Jiao} and used in \cite{Koltchinskii_2022}.
Namely, denote by ${\mathcal A}_{\rm sym}$ the $\sigma$-algebra generated by all random variables of the form $\psi(X_1\otimes X_1,\dots, X_n\otimes X_n),$
where $\psi$ is a symmetric Borel function of $n$ variables with values in the space $L_{\rm sa}({\mathbb H})$ of bounded self-adjoint operators in ${\mathbb H}.$  
Denote 
\begin{align*}
 \check T_{f}(X_1,\dots, X_n) =\check T_{f,m}(X_1,\dots, X_n)=\check T_{f,\vec{n}}(X_1,\dots, X_n):= {\mathbb E}(\hat T_{f,\vec{n}}(X_1,\dots, X_n)|{\mathcal A}_{\rm sym}).
\end{align*}
Thus, estimator $\check T_{f,\vec{n}}(X_1,\dots, X_n)$ is the symmetrization of statistic $\hat T_{f,\vec{n}}(X_1,\dots, X_n),$
which could be easily expressed in terms of $U$-statistics and interpreted as a jackknife estimator. Indeed, if $h(x_1,\dots, x_k)$ is a symmetric 
kernel of order $k\leq n$ on $L_{\rm sa}({\mathbb H}),$ then 
\begin{align*} 
{\mathbb E}(h(X_1\otimes X_1,\dots, X_k\otimes X_k)|{\mathcal A}_{\rm sym})= 
\frac{1}{{n\choose m}}\sum_{i_1<\dots <i_k} h(X_{i_1}\otimes X_{i_1}, \dots, X_{i_k}\otimes X_{i_k}).
\end{align*}
Since plug-in estimator $f(\hat \Sigma_{n_i})$ is a symmetric function of $n_i$ variables $X_1\otimes X_1,\dots, X_{n_i}\otimes X_{n_i},$
we have 
\begin{align*}
 \check T_{f}(X_1,\dots, X_n) = \sum_{j=1}^m C_j U_n f(\hat \Sigma_{n_j}).
\end{align*}
We can also define a signed measure 
\begin{align*}
\check \mu_n = \check \mu_{n,m}=\check \mu_{n,\vec{n}}:={\mathbb E} (\hat \mu_n|{\mathcal A}_{\rm sym})
=\sum_{j=1}^m C_j  U_n\mu_{\hat \Sigma_{n_j}}
\end{align*}
as a symmetrization of $\hat \mu_n$ and we have 
\begin{align*}
 \check T_{f}(X_1,\dots, X_n) = \int_{{\mathbb R}_+} fd\check\mu_n.
\end{align*}

\begin{theorem}
\label{Th_M_A}
Under the conditions of theorems \ref{Th_M_YYY} and \ref{Th_M1_YYY}, the claims of these theorems hold also for estimator $\check T_{f,\vec{n}}(X_1,\dots, X_n)$
and for signed measure $\check \mu_{n,m}.$ 
\end{theorem}

In addition to this, estimator $\check T_{f,\vec{n}}(X_1,\dots, X_n)$ could be approximated by a linear form $\langle f'(\Sigma), \hat \Sigma_n-\Sigma\rangle$
of $\hat \Sigma_n-\Sigma,$ which provides a way to establish normal approximation for this estimator and (under proper assumptions) to establish its asymptotic 
efficiency. Namely, the following result holds.

\begin{theorem}
\label{Th_M_B_YYY}
Suppose $f\in C^{m+1}({\mathbb R}_{+})$ for some $m\geq 2,$ $f(0)=0,$ $\|f'\|_{\rm Lip}\lesssim 1$ and $\|f^{(m+1)}\|_{L_{\infty}}\lesssim 1.$
Suppose also 
that $\|\Sigma\|\lesssim 1$ and ${\bf r}(\Sigma)\lesssim n.$ Then
\begin{align*}
&
\Bigl\|\check T_{f,m}(X_1,\dots, X_n)-\tau_f(\Sigma)- \langle f'(\Sigma), \hat \Sigma_n-\Sigma\rangle\Bigr\|_{L_{\psi_{1/2}}} 
\lesssim_m 
\frac{{\bf r}(\Sigma)}{n}
+ 
{\bf r}(\Sigma)\Bigl(\sqrt{\frac{{\bf r}(\Sigma)}{n}}\Bigr)^{m+1}.
\end{align*}
As a consequence, 
\begin{align*}
&
W_{\psi_{1/2}}\Bigl(\sqrt{\frac{n}{2}}\Bigl(\check T_{f,m}(X_1,\dots, X_n)-\tau_f(\Sigma)\Bigr), G_{\Sigma}(f)\Bigr) 
\lesssim_m 
\frac{{\bf r}(\Sigma)}{\sqrt{n}}
+{\bf r}(\Sigma)\sqrt{n}\Bigl(\sqrt{\frac{{\bf r}(\Sigma)}{n}}\Bigr)^{m+1}.
\end{align*}
\end{theorem}

Finally, we provide a result on Gaussian approximation of stochastic process
\begin{align*}
\check G_n(f):=\sqrt{\frac{n}{2}}\Bigl(\check T_{f,m}(X_1,\dots, X_n)-\tau_f(\Sigma)\Bigr), f\in {\mathcal F}_m.
\end{align*}

\begin{theorem}
\label{Th_M_C_YYY}
Suppose that $\|\Sigma\|\lesssim 1$ and ${\bf r}(\Sigma)\lesssim n.$
Then,
for all $m \geq 2,$  
\begin{align*}
&
\nonumber
{\mathcal W}_{{\mathcal F}_m, \psi_{1/2}} (\check G_n, G_{\Sigma})
\lesssim_m 
\frac{{\bf r}(\Sigma)}{\sqrt{n}}
+ 
{\bf r}(\Sigma)\sqrt{n}\Bigl(\sqrt{\frac{{\bf r}(\Sigma)}{n}}\Bigr)^{m+1}.
\end{align*}
\end{theorem}

\begin{remark}
\normalfont
Note that $\frac{G_{\Sigma}(f)}{\|\Sigma f'(\Sigma)\|_2}=:Z\sim N(0,1).$ Thus, the second 
bound of Theorem \ref{Th_M_B_YYY} implies that 
\begin{align*}
&
W_{\psi_{1/2}}\Bigl(\frac{\sqrt{n}(\check T_{f,m}(X_1,\dots, X_n)-\tau_f(\Sigma))}{\sqrt{2}\|\Sigma f'(\Sigma)\|_2}, Z\Bigr) 
\lesssim_m 
\frac{{\bf r}(\Sigma)}{\|\Sigma f'(\Sigma)\|_2}\Bigl(\frac{1}{\sqrt{n}}
+\sqrt{n}\Bigl(\sqrt{\frac{{\bf r}(\Sigma)}{n}}\Bigr)^{m+1}\Bigr).
\end{align*}
As in Remark \ref{rem_on_eff_rank}, assume that function $f'$ is uniformly bounded from above and bounded away from zero by numerical constants.
Then, we get 
\begin{align*}
&
W_{\psi_{1/2}}\Bigl(\frac{\sqrt{n}(\check T_{f,m}(X_1,\dots, X_n)-\tau_f(\Sigma))}{\sqrt{2}\|\Sigma f'(\Sigma)\|_2}, Z\Bigr) 
\lesssim_m 
\frac{{\bf r}(\Sigma)}{\sqrt{{\bf r}(\Sigma^2)}}\Bigl(\frac{1}{\sqrt{n}}
+\sqrt{n}\Bigl(\sqrt{\frac{{\bf r}(\Sigma)}{n}}\Bigr)^{m+1}\Bigr).
\end{align*}
Finally, assume that ${\rm dim}({\mathbb H})=d<\infty$ and also that $\lambda_j\asymp 1, j=1,\dots, d.$ In this case, ${\bf r}(\Sigma)\asymp d$ and 
${\bf r}(\Sigma^2)\asymp d,$ and the bound simplifies as follows: 
\begin{align*}
&
W_{\psi_{1/2}}\Bigl(\frac{\sqrt{n}(\check T_{f,m}(X_1,\dots, X_n)-\tau_f(\Sigma))}{\sqrt{2}\|\Sigma f'(\Sigma)\|_2}, Z\Bigr) 
\lesssim_m 
\sqrt{\frac{d}{n}}
+\sqrt{nd}\Bigl(\sqrt{\frac{d}{n}}\Bigr)^{m+1}.
\end{align*}
Under the assumption $d=o(n^{m/(m+2)})$ as $n\to\infty,$ this proves the convergence of r.v.
\begin{align*}
\frac{\sqrt{n}(\check T_{f,m}(X_1,\dots, X_n)-\tau_f(\Sigma))}{\sqrt{2}\|\Sigma f'(\Sigma)\|_2} \overset{d}{\to} Z
\end{align*}
in distribution to a standard normal r.v. as well as the convergence to $1$ of properly normalized mean squared error of estimator 
$\check T_{f,m}(X_1,\dots, X_n):$
\begin{align*}
\frac{n{\mathbb E}_{\Sigma}(\check T_{f,m}(X_1,\dots, X_n)-\tau_f(\Sigma))^2}{2\|\Sigma f'(\Sigma)\|_2^2} \to 1\ {\rm as}\ n\to\infty.
\end{align*}
In particular, the convergence holds if $d\leq n^{\alpha}$
for some $\alpha\in (0,1)$ and $m+1>\frac{1+\alpha}{1-\alpha}.$ Theorem \ref{Th_M_C_YYY}
implies that, in this case, the convergence of stochastic processes $\frac{\check G_n(f)}{\|\Sigma f'(\Sigma)\|_2}, f\in {\mathcal F}_m$
to the process $\frac{G_{\Sigma}(f)}{\|\Sigma f'(\Sigma)\|_2}, f\in {\mathcal F}_m$ also holds.
\end{remark}

The following local minimax lower bound shows the asymptotic efficiency of estimator $\check T_{f,m}(X_1,\dots, X_n).$
Its proof easily follows from Theorem 2.3 in \cite{Koltchinskii_Li} (whose proof is based on van Trees inequality).

\begin{proposition}
\label{local_min_max}
Let $\Sigma_0$ be a finite rank covariance operator: ${\rm rank}(\Sigma_0)=r.$ Moreover, suppose that $\Sigma_0$ is invertible in the finite-dimensional subspace $L:={\rm Im}(\Sigma_0)$ of dimension $r.$ Let ${\mathcal S}_L$ be the set of all covariance operators $\Sigma$ with ${\rm Im}(\Sigma)\subset L$
(and, hence, of rank $\leq r$).
Let $f\in C^1({\mathbb R}_+)$ be a function such that $\|f'\|_{{\rm Lip}}<\infty$ and $\|f'(\Sigma_0)\|_2>0.$ Then, there exists 
a constant $D>0$ such that, for all 
$\delta\leq \frac{1}{2\|\Sigma_0^{-1}\|}\wedge 1,$
\begin{align*}
&
\inf_{T_n} \sup_{\Sigma\in {\mathcal S}_L,\|\Sigma-\Sigma_0\|_2\leq \delta} \frac{n {\mathbb E}_{\Sigma}(T_n(X_1,\dots, X_n)-\tau_f(\Sigma))^2}
{2\|\Sigma f'(\Sigma)\|_2^2}
\\
&
\geq 1 - D\|\Sigma_0\|^2 \|\Sigma_0^{-1}\|^2 \Bigl(\frac{\|f'\|_{\rm Lip}\delta}{\|f'(\Sigma_0)\|_2} + \|\Sigma_0\|^2 \|\Sigma_0^{-1}\|^2 \delta + \frac{\|\Sigma_0\|^2}{\delta^2 n}\Bigr),
\end{align*}
where the infimum is taken over all estimators $T_n(X_1,\dots, X_n)$ based on i.i.d. observations $X_1,\dots, X_n\sim N(0,\Sigma).$
\end{proposition}

The bound of Proposition \ref{local_min_max} implies the following version of H\'ajek-LeCam local asymptotic minimax bound: 
\begin{align*}
&
\lim_{c\to \infty}\liminf_{n\to \infty}\inf_{T_n} \sup_{\Sigma\in {\mathcal S}_L,\|\Sigma-\Sigma_0\|_2\leq \frac{c}{\sqrt{n}}} 
\frac{n {\mathbb E}_{\Sigma}(T_n(X_1,\dots, X_n)-\tau_f(\Sigma))^2}
{2\|\Sigma f'(\Sigma)\|_2^2}\geq 1.
\end{align*}

\section{Concentration and normal approximation bounds}
\label{Conc}

Our goal in this section is to study concentration and normal approximation properties of the linear term $\langle f'(\Sigma), \hat \Sigma_n-\Sigma\rangle$
and the remainder $R_f(\Sigma,\hat \Sigma_n-\Sigma)$ of the first order Taylor expansion of trace functional $\tau_f(\hat \Sigma_n),$ which would yield 
concentration inequalities for $\tau_f(\hat \Sigma_n)$ itself. There are many results on concentration inequalities 
for linear spectral statistics and normalized spectral measures of random matrices available in the literature (see \cite{Guionnet_Zeitouni, Anderson, Chatterjee, Houdre, Meckes, Adamczak_Wolff, B_B_H} and references therein), but our aim will be to express concentration properties in terms of effective rank ${\bf r}(\Sigma).$ 

Consider an operator function $f(A),$ where $f:{\mathbb R}\mapsto {\mathbb R}$ and $A:{\mathbb H}\mapsto {\mathbb H}$ is a self-adjoint operator. 
If ${\mathbb H}$ is finite-dimensional, the spectral representation of $A$ could be written as $A=\sum_{\lambda \in \sigma(A)} \lambda P_{\lambda},$
where $\sigma(A)$ is the spectrum of $A$ and $P_{\lambda}$ is the orthogonal projection on the eigen-space of $A,$ corresponding to its eigenvalue 
$\lambda\in \sigma(A).$  The following formula for the derivative of $f(A)$ in the direction of a self-adjoint operator $H$ is well known (see, e.g., \cite{Bhatia}, Theorem V.3.3):
\begin{align}
\label{D_K_formula}
(Df)(A)[H] = \frac{d}{dt} f(A+tH)_{\vert_{t=0}}= \sum_{\lambda, \mu\in \sigma(A)}f^{[1]}(\lambda, \mu) P_{\lambda}HP_{\mu},
\end{align}
where 
\begin{align*}
f^{[1]}(\lambda,\mu):= 
\begin{cases}
\frac{f(\lambda)-f(\mu)}{\lambda-\mu} & \lambda\neq \mu\\
f'(\lambda) & \lambda=\mu
\end{cases}
\end{align*}
is the first order divided difference of function $f.$
Formula \eqref{D_K_formula} holds for any function $f$ continuously differentiable in an open interval that contains $\sigma(A)$ and, moreover,
in this case, the operator function $A\mapsto f(A)$ is Fr\'echet differentiable at $A.$ The following formula for the derivative of trace functional $\tau_f$  is an immediate 
consequence of \eqref{D_K_formula}: 
\begin{align*}
(D\tau_f)(A)[H]= \frac{d}{dt} \tau_f(A+tH)_{\vert_{t=0}}= \sum_{\lambda\in \sigma(A)} f'(\lambda){\rm tr}(P_{\lambda}HP_{\lambda})=\langle f'(A),H\rangle.
\end{align*}
Similar differentiation formulas hold in the infinite-dimensional case, although more smoothness of function $f$ could be required (for instance, that $f$ belongs to the Besov class $B^{1}_{\infty,1}({\mathbb R})$) and, in the case of operator $A$ with continuous spectrum, 
the double discrete sum in \eqref{D_K_formula} should be replaced by a double operator integral with respect to the resolution of identity of operator $A$
(see, e.g., \cite{Peller-06, Peller-16} and references therein).

Let 
\begin{align*}
R_f(A;H):= \tau_f(A+H)-\tau_f(A) - \langle f'(A),H\rangle
\end{align*}
be the remainder of the first order Taylor expansion of functional $\tau_f(A).$ 
In the finite dimensional case, it is easy to show the following bound on $R_f(A;H).$

\begin{proposition}
\label{rem_first_order}
Suppose $f\in C^1({\mathbb R})$ and $\|f'\|_{\rm Lip}<\infty.$
Then 
\begin{align*}
|R_f(A;H)|\leq \frac{\|f'\|_{\rm Lip}}{2} \|H\|_2^2.
\end{align*}
\end{proposition}

\begin{proof}
Note that 
\begin{align*}
\frac{d}{dt} \tau_f(A+tH)= \langle f'(A+tH), H\rangle, t\in [0,1].
\end{align*}
Therefore, 
\begin{align*}
R_f(A;H)= \tau_f(A+H)-\tau_f(A) - \langle f'(A),H\rangle = \int_0^1 \langle f'(A+tH)-f'(A), H\rangle dt,
\end{align*}
implying the bound 
\begin{align*}
|R_f(A;H)| \leq \int_0^1 \|f'(A+tH)-f'(A)\|_2 dt \|H\|_2.
\end{align*}
It is well known that, if $f'$ is Lipschitz, then the operator function $A\mapsto f'(A)$ is also Lipschitz with respect 
to the Hilbert--Schmidt norm (see, e.g. \cite{Pot_Suk} and references therein), which implies that
\begin{align*} 
\|f'(A+tH)-f'(A)\|_2 \leq \|f'\|_{{\rm Lip}} t \|H\|_2, t\in [0,1].
\end{align*}
The claim of the proposition now easily follows. 

\qed
\end{proof}

Clearly, we have
\begin{align*}
\tau_f (\hat \Sigma_n)&= \tau_f(\Sigma)+ \langle f'(\Sigma), \hat \Sigma_n-\Sigma\rangle + R_f(\Sigma, \hat \Sigma_n-\Sigma),
\end{align*}
which implies that
\begin{align*}
\tau_f(\hat \Sigma_n)-{\mathbb E}\tau_f(\hat \Sigma_n)= \langle f'(\Sigma), \hat \Sigma_n-\Sigma\rangle + R_f(\Sigma, \hat \Sigma_n-\Sigma)-{\mathbb E}R_f(\Sigma, \hat \Sigma_n-\Sigma).
\end{align*}
Thus, the study of concentration and normal approximation properties of plug-in estimator $\tau_f(\hat \Sigma_n)$ reduces to the study of concentration and normal approximation 
of the linear term $\langle f'(\Sigma), \hat \Sigma_n-\Sigma\rangle$ and concentration of the remainder  $R_f(\Sigma, \hat \Sigma_n-\Sigma).$
In what follows, we will obtain concentration and normal approximation bounds with explicit dependence on the effective rank ${\bf r}(\Sigma)$
and some other relevant quantities.

First, we deal with concentration properties of random variables 
$\langle f'(\Sigma), \hat \Sigma_n-\Sigma\rangle.$ 
For this, we will review several well know facts. 

\begin{proposition}
\label{sub_exp}
Let $\{\xi_k\}$ be independent centered random variables with $\sup_{k\geq 1}\|\xi_k\|_{\psi_1}<\infty$ and let $\{a_k\}$
be a sequence of real numbers with $\sum_{k\geq 1}a_k^2<\infty.$ Then, the series $\sum_{k\geq 1}a_k \xi_k$ converges 
a.s. and, for all $p\geq 1,$ 
\begin{align*}
\Bigl\| \sum_{k\geq 1}a_k \xi_k \Bigr\|_{L_p} \lesssim \sup_{k\geq 1}\|\xi_k\|_{\psi_1} \Bigl(\Bigl(\sum_{k\geq 1}a_k^2\Bigr)^{1/2}\sqrt{p} \vee \sup_{k\geq 1}|a_k| p\Bigr).
\end{align*}
\end{proposition}

\begin{proof}
The result easily follows from Bernstein's inequality for sums of independent subexponential random variables (see \cite{Vershynin}, Theorem 2.8.1) by expressing the $p$-th moments in terms of the tail probabilities and bounding the integrals. 

\qed
\end{proof}

\begin{proposition}
\label{prop_lin_form}
Let $B$ be a bounded self-adjoint operator in ${\mathbb H}.$ Then, for all $p\geq 1,$
\begin{align*}
\Bigl\|\langle\hat \Sigma_n-\Sigma, B\rangle\Bigr\|_{L_p} \lesssim \|\Sigma^{1/2}B\Sigma^{1/2}\|_2 \sqrt{\frac{p}{n}} \vee \|\Sigma^{1/2}B\Sigma^{1/2}\| \frac{p}{n}.
\end{align*}
\end{proposition}

\begin{proof}
It is enough to prove the inequality in the case when covariance operator $\Sigma$ is of finite rank $N$ (the general case easily follows by a finite rank approximation 
and passing to the limit as $N\to \infty$). Let $\phi_1, \dots, \phi_N$ be the orthonormal eigenvectors of operator $\Sigma^{1/2} B\Sigma^{1/2}$ (acting 
in ${\rm Im}(\Sigma)$) corresponding 
to its eigenvalues $\mu_1,\dots, \mu_N.$ Clearly, we can represent $X=\Sigma^{1/2}Z,$ $Z=\sum_{j=1}^N g_j \phi_j,$ where $\{g_j\}$ are i.i.d. standard normal random variables. 
We can also represent $X_i =\Sigma^{1/2}Z_i, i=1,\dots, n,$ 
$Z_i= \sum_{j=1}^N g_{ij} \phi_j, i=1,\dots, n$ with $\{g_{ij}\}$ being i.i.d. standard normal.
Then, it follows that 
\begin{align*}
\langle B X, X\rangle = \langle \Sigma^{1/2}B\Sigma^{1/2}Z,Z\rangle = \sum_{j=1}^N \mu_j g_j^2
\end{align*}
and 
\begin{align*}
\langle \hat \Sigma_n-\Sigma, B\rangle&= n^{-1} \sum_{i=1}^n \langle B X_i, X_i\rangle-{\mathbb E}\langle BX,X\rangle
\\
&
=\sum_{i=1}^n \sum_{j=1}^N n^{-1} \mu_j (g_{i,j}^2-1).
\end{align*}
Since $\|g_{i,j}^2-1\|_{\psi_1}\lesssim 1,$
we can apply the bound of Proposition \ref{sub_exp} to get
\begin{align*}
\Bigl\|\langle \hat \Sigma_n-\Sigma, B\rangle\Bigr\|_{L_p} \lesssim \Bigl(\sum_{j=1}^N \mu_j^2\Bigr)^{1/2} \sqrt{\frac{p}{n}}\vee \max_{1\leq j\leq N}|\mu_j| \frac{p}{n},
\end{align*}
implying the claim.

\qed
\end{proof}


In particular, we will use the following immediate corollary of Proposition \ref{prop_lin_form}.

\begin{proposition}
\label{prop_lin}
Let $f\in C^1({\mathbb R}_+).$
Then, for all $p\geq 1,$
\begin{align*}
\Bigl\|\langle f'(\Sigma), \hat \Sigma_n-\Sigma\rangle\Bigr\|_{L_p} \lesssim \|\Sigma f'(\Sigma)\|_2 \sqrt{\frac{p}{n}} \vee \|\Sigma f'(\Sigma)\| \frac{p}{n}.
\end{align*}
\end{proposition}

Another immediate consequence of Proposition \ref{prop_lin_form} is the following concentration bound for ${\rm tr}(\hat \Sigma_n).$

\begin{proposition}
\label{trace_conc}
For all $p\geq 1,$
\begin{align*}
\Bigl\|{\rm tr}(\hat \Sigma_n)-{\rm tr}(\Sigma)\Bigr\|_{L_p} \lesssim \|\Sigma\|_2 \sqrt{\frac{p}{n}} \vee \|\Sigma\| \frac{p}{n}.
\end{align*}
\end{proposition}

Next we look at normal approximation in Wasserstein $W_2$-distances of random variables 
\begin{align*}
G_n(f):= \sqrt{\frac{n}{2}} \langle f'(\Sigma), \hat \Sigma_n-\Sigma\rangle, f\in C^1({\mathbb R}_+).
\end{align*}
Recall that $G_{\Sigma}(f), f\in C^1({\mathbb R}_+)$ is a centered Gaussian process with covariance function 
\begin{align*}
{\mathbb E} G_{\Sigma}(f) G_{\Sigma}(g):= {\rm tr} (\Sigma^2 f'(\Sigma)g'(\Sigma)), f,g\in C^1({\mathbb R}_+).
\end{align*}

\begin{proposition}
\label{norm_approx_lin}
For all $f\in C^1({\mathbb R}_+)$ and all $p\geq 1,$
\begin{align*}
W_p(G_n(f), G_{\Sigma}(f)) \lesssim \frac{\|\Sigma f'(\Sigma)\|_2\sqrt{p}\vee \|\Sigma f'(\Sigma)\| p}{\sqrt{n}}. 
\end{align*}
\end{proposition}

\begin{proof}
We will use the construction from the proof of Proposition \ref{prop_lin_form} with $\phi_1,\phi_2,\dots$ being the orthonormal eigenvectors of covariance operator $\Sigma$
corresponding to its eigenvalues $\lambda_1\geq \lambda_2\geq \dots.$ Then $X=\sum_{k\geq 1} \sqrt{\lambda}_k g_k \phi_k$ with i.i.d. standard normal r.v. $\{g_k\}$ and 
$X_j=\sum_{k\geq 1} \sqrt{\lambda}_k g_{j,k} \phi_k, j=1,\dots, n$ with  i.i.d. standard normal r.v. $\{g_{j,k}\}.$
The following representation holds for stochastic process $G_n(f), f\in C^1({\mathbb R}_+):$ 
\begin{align*}
G_n(f) = \sum_{k\geq 1} \lambda_k f'(\lambda_k) \frac{1}{\sqrt{2n}}\sum_{i=1}^n (g_{i,k}^2-1).
\end{align*}
On the other hand, Gaussian process $G_{\Sigma}(f)$ could be represented as 
\begin{align*}
G_{\Sigma}(f) := \sum_{k\geq 1} \lambda_k f'(\lambda_k) \eta_k,
\end{align*}
where $\{\eta_k\}$ are also i.i.d. standard normal r.v. Note that the series representing processes $G_n(f)$ and $G_{\Sigma}(f)$ converge a.s. 
and also in the $\psi_1$-norm (see Proposition \ref{sub_exp}).
By Theorem 2.1 in \cite{Rio}, for i.i.d. $g_i\sim N(0,1)$ and $\eta\sim N(0,1),$
\begin{align*}
W_{\psi_1}\Bigl(\frac{1}{\sqrt{2n}}\sum_{i=1}^n (g_{i}^2-1), \eta\Bigr) \lesssim \frac{1}{\sqrt{n}},
\end{align*} 
which implies that 
\begin{align*}
\sup_{k\geq1}W_{\psi_1}\Bigl(\frac{1}{\sqrt{2n}}\sum_{i=1}^n (g_{i,k}^2-1), \eta_k\Bigr) \lesssim \frac{1}{\sqrt{n}}
\end{align*}
and, moreover, random vectors $(g_{1,k},\dots, g_{n,k}, \eta_k), k\geq 1$ are independent. Therefore, one can construct independent random variables 
$(\bar \xi_k, \bar \eta_k), k\geq 1,$ such that 
\begin{align*}
\bar \xi_k \overset{d}{=}\frac{1}{2\sqrt{n}}\sum_{i=1}^n (g_{i}^2-1),\  \bar \eta_k\overset{d}{=}\eta_k
\end{align*}
and 
\begin{align*}
\sup_{k\geq 1}\|\bar \xi_k-\bar \eta_k\|_{\psi_1}\lesssim \frac{1}{\sqrt{n}}.
\end{align*}
Let $\zeta_k:= \bar \xi_k-\bar \eta_k, k\geq 1.$
Then, $\{\zeta_k\}$ are independent centered r.v. with $\sup_{k\geq 1}\|\zeta_k\|_{\psi_1}\lesssim \frac{1}{\sqrt{n}}.$ 
Also, stochastic process 
\begin{align*}
\bar G_n(f) := \sum_{k\geq 1} \lambda_k f'(\lambda_k) \bar \xi_k, f\in {\mathcal F}
\end{align*} 
has the same finite dimensional distributions as $G_n(f), f\in {\mathcal F}$ and stochastic process 
\begin{align*}
\bar G_{\Sigma}(f) := \sum_{k\geq 1} \lambda_k f'(\lambda_k) \bar \eta_k, f\in {\mathcal F}
\end{align*}
has the same finite dimensional distributions as $G_{\Sigma}(f), f\in {\mathcal F}.$ 
By Proposition \ref{sub_exp}, for all $p\geq 1,$
\begin{align*}
&
W_p(G_n(f), G_{\Sigma}(f)) \leq \|\bar G_n(f)-\bar G_{\Sigma}(f)\|_{L_p} 
= \Bigl\|\sum_{k\geq 1} \lambda_k f'(\lambda_k)\zeta_k\Bigr\|_{L_p} 
\\
&
\lesssim \frac{1}{\sqrt{n}}\Bigl(\Bigl(\sum_{k\geq 1}\lambda_k^2 (f'(\lambda_k))^2\Bigr)\sqrt{p}\vee 
\sup_{k\geq 1}|\lambda_k||f'(\lambda_k)| p\Bigr),
\end{align*}
which completes the proof.

\qed
\end{proof}

Note that $\frac{G_{\Sigma}(f)}{\|\Sigma f'(\Sigma)\|_2}$ is a standard normal random variable. Thus, the bound of Proposition \ref{norm_approx_lin} implies that 
\begin{align*}
W_p\Bigl(\frac{G_n(f)}{\|\Sigma f'(\Sigma)\|_2}, Z\Bigr) \lesssim \frac{\sqrt{p}}{\sqrt{n}}\vee \frac{\|\Sigma f'(\Sigma)\|}{\|\Sigma f'(\Sigma)\|_2}\frac{p}{\sqrt{n}} \lesssim \frac{p}{\sqrt{n}},
\end{align*}
yielding that 
\begin{align*}
W_{\psi_1}\Bigl(\frac{G_n(f)}{\|\Sigma f'(\Sigma)\|_2}, Z\Bigr) \lesssim \frac{1}{\sqrt{n}},
\end{align*}
where $Z$ is a standard normal random variable.

Several bounds for sample covariance operators will be useful in the study of concentration properties of the remainder 
$R_f(\Sigma, \hat \Sigma_n-\Sigma)$ of the first order Taylor expansion as well as for other purposes. Its proof is based 
on 

\begin{proposition}
\label{HS_error} 
The following bound holds:
\begin{align*}
{\mathbb E}^{1/2}\|\hat \Sigma_n-\Sigma\|_2^2 \lesssim \|\Sigma\| \frac{{\bf r}(\Sigma)}{\sqrt{n}}.
\end{align*}
\end{proposition}

\begin{proof}
Indeed,
\begin{align*}
{\mathbb E}\|\hat \Sigma_n-\Sigma\|_2^2 = \frac{{\mathbb E}\|X\otimes X-\Sigma\|_2^2}{n}
\leq \frac{{\mathbb E}\|X\otimes X\|_2^2}{n}= \frac{{\mathbb E}\|X\|^4}{n}.
\end{align*}
It follows from the Gaussian concentration inequality that 
\begin{align*}
{\mathbb E}^{1/4} (\|X\|-{\mathbb E}\|X\|)^4 \lesssim \|\Sigma\|^{1/2},
\end{align*}
which implies that 
\begin{align*}
{\mathbb E}^{1/4}\|X\|^4 \leq {\mathbb E}\|X\| + C \|\Sigma\|^{1/2} \leq 
{\mathbb E}^{1/2}\|X\|^2 + C \|\Sigma\|^{1/2} = \|\Sigma\|^{1/2}\sqrt{{\bf r}(\Sigma)}+ C \|\Sigma\|^{1/2}
\lesssim \|\Sigma\|^{1/2}\sqrt{{\bf r}(\Sigma)}.
\end{align*}
The claim now easily follows.

\qed
\end{proof}

The following result could be viewed as an extension of the bound of Proposition \ref{HS_error} to Schatten $p$-norms for $p\geq 2.$
Its proof is based on  non-commutative Khintchine inequality \cite{Lust_Piquard, LP-Pisier} (these inequalities were first used to obtain 
bounds for the sample covariance in \cite{Rudelson}).


\begin{proposition}
\label{prop_schat}
Suppose ${\bf r}(\Sigma)\lesssim n.$
Then, for all $p\geq 2,$
\begin{align*}
{\mathbb E}^{1/p} \|\hat \Sigma_n-\Sigma\|_p^p
&
\lesssim
 \|\Sigma\| {\bf r}(\Sigma)^{1/p}(\sqrt{{\bf r}(\Sigma)}\vee \sqrt{p\log n}) \Bigl(1\vee \Bigl(\frac{p}{n}\Bigr)^{1/4-1/2p}\vee \Bigl(\frac{p}{n}\Bigr)^{1/2-1/p}\Bigr)\sqrt{\frac{p}{n}}
\\
&
\lesssim_p 
\|\Sigma\| {\bf r}(\Sigma)^{1/p} \Bigl(\sqrt{\frac{{\bf r}(\Sigma)}{n}} \vee \sqrt{\frac{\log n}{n}}\Bigr).
\end{align*}
\end{proposition}

\begin{remark}
\normalfont
If ${\rm dim}({\mathbb H})=d<\infty,$ $\|\hat \Sigma_n-\Sigma\|_p \leq d^{1/p}\|\hat \Sigma_n-\Sigma\|.$ In this case, it easily follows from \eqref{KL_1} and \eqref{KL_2} that 
\begin{align*}
{\mathbb E}^{1/p} \|\hat \Sigma_n-\Sigma\|_p^p \lesssim_p d^{1/p} \|\Sigma\|\sqrt{\frac{{\bf r}(\Sigma)}{n}}.
\end{align*}
Proposition \ref{prop_schat} shows that, in the bounds of this type, the factor $d^{1/p}$ can be replaced by ${\bf r}(\Sigma)^{1/p}$ (at least, for $p\geq 2$).
\end{remark}

\begin{proof}
Without loss of generality, we can assume that ${\rm dim}({\mathbb H})<\infty$ (otherwise, we can use a finite-dimensional approximation).

By symmetrization inequality, we have 
\begin{align*}
{\mathbb E} \|\hat \Sigma_n-\Sigma\|_p^p\leq 2^p {\mathbb E}\Bigl\|\frac{1}{n}\sum_{j=1}^n \eps_j (X_j\otimes X_j)\Bigr\|_p^p
=\frac{2^p}{n^{p/2}} {\mathbb E}{\mathbb E}_{\eps}\Bigl\|\frac{1}{\sqrt{n}}\sum_{j=1}^n \eps_j (X_j\otimes X_j)\Bigr\|_p^p,
\end{align*}
where $\eps_1,\dots, \eps_n$ are i.i.d. Rademacher random variables independent of $X_1,\dots, X_n.$
By non-commutative Khintchine inequality \cite{Lust_Piquard, LP-Pisier},
\begin{align*}
{\mathbb E}_{\eps}^{1/p}\Bigl\|\frac{1}{\sqrt{n}}\sum_{j=1}^n \eps_j (X_j\otimes X_j)\Bigr\|_p^p 
&\lesssim 
C\sqrt{p} \Bigl\|\Bigl(\frac{1}{n}\sum_{j=1}^n (X_j\otimes X_j)^2\Bigr)^{1/2}\Bigr\|_p
\\
&
=C\sqrt{p} \Bigl\|\frac{1}{n}\sum_{j=1}^n (X_j\otimes X_j)^2\Bigr\|_{p/2}^{1/2},
\end{align*}
which implies 
\begin{align*}
{\mathbb E}_{\eps}\Bigl\|\frac{1}{\sqrt{n}}\sum_{j=1}^n \eps_j (X_j\otimes X_j)\Bigr\|_p^p \lesssim C^p(\sqrt{p})^p 
\Bigl\|\frac{1}{n}\sum_{j=1}^n \|X_j\|^2(X_j\otimes X_j)\Bigr\|_{p/2}^{p/2}.
\end{align*}
Therefore,
\begin{align*}
{\mathbb E} \|\hat \Sigma_n-\Sigma\|_p^p\leq 
\frac{2^p C^p (\sqrt{p})^p}{n^{p/2}}{\mathbb E}\Bigl\|\frac{1}{n}\sum_{j=1}^n \|X_j\|^2(X_j\otimes X_j)\Bigr\|_{p/2}^{p/2}
\end{align*}
and 
\begin{align*}
{\mathbb E}^{1/p} \|\hat \Sigma_n-\Sigma\|_p^p\leq 
\frac{2 C \sqrt{p}}{\sqrt{n}}{\mathbb E}^{1/p}\Bigl\|\frac{1}{n}\sum_{j=1}^n \|X_j\|^2(X_j\otimes X_j)\Bigr\|_{p/2}^{p/2}.
\end{align*}
Note that 
\begin{align*}
\frac{1}{n}\sum_{j=1}^n \|X_j\|^2(X_j\otimes X_j) \preceq \max_{1\leq j\leq n}\|X_j\|^2 \hat \Sigma_n,
\end{align*}
which implies that 
\begin{align*}
\Bigl\|\frac{1}{n}\sum_{j=1}^n \|X_j\|^2(X_j\otimes X_j)\Bigr\|_{p/2} &\leq \max_{1\leq j\leq n}\|X_j\|^2 \|\hat \Sigma_n\|_{p/2}
= \max_{1\leq j\leq n}\|X_j\|^2 \Bigl({\rm tr}(\hat \Sigma_n^{p/2})\Bigr)^{2/p}
\\
&
\leq \max_{1\leq j\leq n}\|X_j\|^2 \|\hat \Sigma_n\|{\bf r}(\Sigma_n)^{2/p}= \max_{1\leq j\leq n}\|X_j\|^2 \|\hat \Sigma_n\|^{1-2/p}\Bigl({\rm tr}(\hat \Sigma_n)\Bigr)^{2/p}.
\end{align*}
This yields the following bounds:
\begin{align}
\label{basic_sch_p}
\nonumber
{\mathbb E}^{1/p} \|\hat \Sigma_n-\Sigma\|_p^p
&\leq 2C\frac{\sqrt{p}}{\sqrt{n}} {\mathbb E}^{1/p} \max_{1\leq j\leq n}\|X_j\|^p \|\hat \Sigma_n\|^{p/2-1} {\rm tr}(\hat \Sigma_n)
\\
&
\nonumber
= 2C\frac{\sqrt{p}}{\sqrt{n}}\Bigl\|\max_{1\leq j\leq n}\|X_j\| \|\hat \Sigma_n\|^{1/2-1/p} {\rm tr}(\hat \Sigma_n)^{1/p}\Bigr\|_{L_p}
\\
&
\nonumber
\leq 
2C\frac{\sqrt{p}}{\sqrt{n}}\Bigl\|\max_{1\leq j\leq n}\|X_j\|\Bigr\|_{L_{3p}} \Bigl\|\|\hat \Sigma_n\|^{1/2-1/p}\Bigr\|_{L_{3p}} \Bigl\|{\rm tr}(\hat \Sigma_n)^{1/p}\Bigr\|_{L_{3p}}
\\
&
=
2C\frac{\sqrt{p}}{\sqrt{n}}\Bigl\|\max_{1\leq j\leq n}\|X_j\|\Bigr\|_{L_{3p}} \Bigl\|\|\hat \Sigma_n\|\Bigr\|_{L_{3p}}^{1/2-1/p} \Bigl\|{\rm tr}(\hat \Sigma_n)\Bigr\|_{L_{3}}^{1/p}.
\end{align}
Using Gaussian concentration inequality for $\|X\|,$ we get 
\begin{align*}
\Bigl\|\max_{1\leq j\leq n} \|X_j\|\Bigr\|_{L_p} &\leq {\mathbb E}\|X\| + {\mathbb E}^{1/p}\max_{1\leq j\leq n}\Bigl|\|X_j\|-{\mathbb E}\|X\|\Bigr|^p
\\
&
\leq {\mathbb E}\|X\| + {\mathbb E}^{1/p\log n}\max_{1\leq j\leq n}\Bigl|\|X_j\|-{\mathbb E}\|X\|\Bigr|^{p\log n}
\\
&
\leq {\mathbb E}\|X\| + \Bigl({\mathbb E}\sum_{j=1}^n\Bigl|\|X_j\|-{\mathbb E}\|X\|\Bigr|^{p\log n}\Bigr)^{1/p\log n}
\\
&
\leq {\mathbb E}\|X\| + n^{1/p\log n} {\mathbb E}^{1/p\log n}\Bigl|\|X\|-{\mathbb E}\|X\|\Bigr|^{p\log n}
\\
&
\leq \|\Sigma\|^{1/2}\sqrt{{\bf r}(\Sigma)}+  e^{1/p}C \|\Sigma\|^{1/2}\sqrt{p\log n}.
\end{align*}
Thus,
\begin{align}
\label{L_p_max_norm}
\Bigl\|\max_{1\leq j\leq n} \|X_j\|\Bigr\|_{L_p} \lesssim \|\Sigma\|^{1/2}(\sqrt{{\bf r}(\Sigma)}+ \sqrt{p\log n}).
\end{align}
As a consequence of \eqref{KL_1}, \eqref{KL_2}, we have under the assumption ${\bf r}(\Sigma)\lesssim n$ that 
\begin{align}
\label{L_p_n_Sigma}
\Bigl\|\|\hat \Sigma_n\|\Bigr\|_{L_p} \leq \|\Sigma\|+\|\|\hat \Sigma_n-\Sigma\|\|_{L_p}  \lesssim \|\Sigma\|\Bigl(1\vee \sqrt{\frac{p}{n}}\vee \frac{p}{n}\Bigr).
\end{align}
Finally, using the bound of Proposition \ref{trace_conc}, we get
\begin{align}
\label{L_p_trace}
\nonumber
\Bigl\|{\rm tr}(\hat \Sigma_n)\Bigr\|_{L_p} &\leq {\rm tr}(\Sigma)+ \Bigl\|{\rm tr}(\hat \Sigma_n)-{\rm tr}(\Sigma)\Bigr\|_{L_p} 
\lesssim \|\Sigma\|{\bf r}(\Sigma) + \|\Sigma\|_2 \sqrt{\frac{p}{n}} + \|\Sigma\|\frac{p}{n}
\\
&
\leq \|\Sigma\|{\bf r}(\Sigma) + \|\Sigma\|^{1/2} \sqrt{{\bf r}(\Sigma)}\|\Sigma\|^{1/2}\sqrt{\frac{p}{n}} + \|\Sigma\|\frac{p}{n}
\lesssim \|\Sigma\|\Bigl({\bf r}(\Sigma)\vee \frac{p}{n}\Bigr).
\end{align}
Using bounds \eqref{L_p_max_norm} and \eqref{L_p_n_Sigma} with $3p$ instead of $p$ and bound \eqref{L_p_trace} with 
$p=3,$ we easily get from \eqref{basic_sch_p} that 
\begin{align*}
{\mathbb E}^{1/p}\|\hat \Sigma_n-\Sigma\|_p^p \lesssim 
 \|\Sigma\| {\bf r}(\Sigma)^{1/p}(\sqrt{{\bf r}(\Sigma)}\vee \sqrt{p\log n}) \Bigl(1\vee \Bigl(\frac{p}{n}\Bigr)^{1/4-1/2p}\vee \Bigl(\frac{p}{n}\Bigr)^{1/2-1/p}\Bigr)\sqrt{\frac{p}{n}},
\end{align*}
which completes the proof.
\qed
\end{proof}

We are interested in concentration 
inequalities for the remainder $R_f(\Sigma, \hat \Sigma_n-\Sigma)$ of the first order Taylor expansion around its expectation. 
Namely, we will prove the following result.

\begin{theorem}
\label{prop_rem}
Let $f\in C^1({\mathbb R}_+)$ with $f'$ being Lipschitz in ${\mathbb R}$ and $f(0)=0.$
Suppose ${\bf r}(\Sigma)\lesssim n.$
Then, for all $p\geq 1,$
\begin{align}
\label{conc_rem_end}
\Bigl\|R_f(\Sigma, \hat \Sigma_n-\Sigma)-{\mathbb E}R_f(\Sigma, \hat \Sigma_n-\Sigma)\Bigr\|_{L_p} &\lesssim 
\|f'\|_{\rm Lip} \|\Sigma\|^2\Bigl(\frac{{\bf r}(\Sigma)}{\sqrt{n}}\sqrt{\frac{p}{n}} \vee \Bigl(\frac{{\bf r}(\Sigma)}{\sqrt{n}}\vee 1\Bigr)\frac{p}{n}\vee \Bigl(\frac{p}{n}\Bigr)^{2}\Bigr).
\end{align}
\end{theorem}

For the proof, we need several auxiliary statements. 

Let $g:S\mapsto {\mathbb R}$ be a locally Lipschitz function on a metric space $(S,d).$ Its local Lipschitz constant $(Lg)(x)$ at point $x\in S$ is defined as 
\begin{align*}
(Lg)(x):= \inf_{U\ni x} \sup_{x_1,x_2\in U} \frac{|f(x_1)-f(x_2)|}{d(x_1,x_2)},
\end{align*}
where the infimum is taken over all balls centered at $x.$ In the case when $S={\mathbb R}^N,$ it is equipped with the standard Euclidean norm and $f$ is continuously differentiable, 
$(Lg)(x)$ coincides with the Euclidean norm of the gradient $(\nabla g)(x).$ The next statement is a well known form of Gaussian concentration inequality. 

\begin{proposition}
\label{L_p_conce}
Let ${\mathcal Z}$ be a standard normal random variable in ${\mathbb R}^N.$ For any locally Lipschitz function $g:{\mathbb R}^N \mapsto {\mathbb R}$
and for all $p\geq 1,$
\begin{align*}
\|g({\mathcal Z})-{\mathbb E}g({\mathcal Z})\|_{L_p}\lesssim \sqrt{p}\|(Lg)({\mathcal Z})\|_{L_p}.
\end{align*}
\end{proposition}

The following simple lemma will be useful.

\begin{lemma}
\label{lem_ABC}
For all $u_i, v_i\in {\mathbb H}, i=1,\dots, n,$
\begin{align*}
&
\Bigl\|\sum_{i=1}^n u_i\otimes v_i\Bigr\|_2 \leq \Bigl(\sum_{i=1}^n \|u_i\|^2\Bigr)^{1/2} \Bigl\|\sum_{j=1}^n v_j\otimes v_j\Bigr\|^{1/2}.
\end{align*}
\end{lemma}

\begin{proof}
Observe that, for an orthonormal basis $\{e_j:j\geq 1\}$ of ${\mathbb H},$
\begin{align*}
&
\Bigl\|\sum_{i=1}^n u_i\otimes v_i\Bigr\|_2 = \Bigl\|\Bigl(\sum_{i=1}^n u_i\otimes e_i\Bigr)\Bigl(\sum_{j=1}^n e_j\otimes v_j\Bigr)\Bigr\|_2
\leq \Bigl\|\sum_{i=1}^n u_i\otimes e_i\Bigr\|_2 \Bigl\|\sum_{j=1}^n e_j\otimes v_j\Bigr\|
\\
&
\leq \Bigl(\sum_{i=1}^n \|u_i\|^2\Bigr)^{1/2} \sup_{\|s\|\leq 1, \|t\|\leq 1} \sum_{j=1}^n \langle e_j, s\rangle \langle v_j, t\rangle
\leq \Bigl(\sum_{i=1}^n \|u_i\|^2\Bigr)^{1/2} \sup_{\|s\|\leq 1} \Bigl(\sum_{j=1}^n \langle e_j, s\rangle^2\Bigr)^{1/2} \sup_{\|t\|\leq 1}\Bigl(\sum_{j=1}^n \langle v_j, t\rangle^2\Bigr)^{1/2}
\\
&
\leq\Bigl(\sum_{i=1}^n \|u_i\|^2\Bigr)^{1/2} \Bigl\|\sum_{j=1}^n v_j\otimes v_j\Bigr\|^{1/2}.
\end{align*}

\qed
\end{proof}

The next proposition provides a simple bound on the local Lipschitz constant $(L\hat \Sigma_n)(X_1,\dots, X_n)$ of the function 
${\mathbb H}\times \dots \times {\mathbb H}\ni (X_1,\dots, X_n)\mapsto \hat \Sigma_n (X_1,\dots, X_n)\in {\mathcal S}_2.$

\begin{proposition}
\label{prop_one}
The following bound holds:
\begin{align*}
(L\hat \Sigma_n)(X_1,\dots, X_n)\leq \frac{2\|\hat \Sigma_n\|^{1/2}}{\sqrt{n}}.
\end{align*}
\end{proposition}

\begin{proof}
Let $\hat \Sigma_n'= \hat \Sigma_n(x_1,\dots, x_n)$ and $\hat \Sigma_n''= \hat \Sigma_n(y_1,\dots, y_n)$ 
for $x_1,\dots, x_n, y_1,\dots, y_n\in {\mathbb H}.$ Then
\begin{align*}
\|\hat \Sigma_n'-\hat \Sigma_n''\|_2 &= \Bigl\|n^{-1}\sum_{j=1}^n x_j\otimes x_j- n^{-1}\sum_{j=1}^n y_j\otimes y_j\Bigr\|_2 
\\
&
\leq 
\Bigl\|n^{-1}\sum_{j=1}^n (x_j-y_j)\otimes x_j \Bigr\|_2 + 
 \Bigl\|n^{-1}\sum_{j=1}^n y_j\otimes (x_j-y_j) \Bigr\|_2.
\end{align*}
The bound of Lemma \ref{lem_ABC} implies that 
\begin{align*}
\Bigl\|n^{-1}\sum_{j=1}^n (x_j-y_j)\otimes x_j \Bigr\|_2\leq n^{-1/2}\Bigl\|n^{-1} \sum_{j=1}^n x_j\otimes x_j\Bigr\|^{1/2} \Bigl(\sum_{i=1}^n\|x_i-y_i\|^2\Bigr)^{1/2} 
\leq \frac{\|\hat \Sigma_n'\|^{1/2}}{\sqrt{n}}\Bigl(\sum_{i=1}^n\|x_i-y_i\|^2\Bigr)^{1/2}.
\end{align*}
Similarly, 
\begin{align*}
\Bigl\|n^{-1}\sum_{j=1}^n y_j\otimes (x_j-y_j) \Bigr\|_2 \leq \frac{\|\hat \Sigma_n''\|^{1/2}}{\sqrt{n}}\Bigl(\sum_{i=1}^n\|x_i-y_i\|^2\Bigr)^{1/2}.
\end{align*}
Thus, 
\begin{align*}
\|\hat \Sigma_n'-\hat \Sigma_n''\|_2\leq \frac{\|\hat \Sigma_n'\|^{1/2}+\|\hat \Sigma_n''\|^{1/2}}{\sqrt{n}}\Bigl(\sum_{i=1}^n\|x_i-y_i\|^2\Bigr)^{1/2}.
\end{align*}
If now we let $x_i \to X_i$ and $y_i\to X_i$ for all $i=1,\dots, n,$ we have that $\|\hat \Sigma_n'\|^{1/2}\to \|\hat \Sigma_n\|^{1/2}$
and $\|\hat \Sigma_n''\|^{1/2}\to \|\hat \Sigma_n\|^{1/2},$ implying that 
\begin{align*}
(L\hat \Sigma_n)(X_1,\dots, X_n)\leq \frac{2\|\hat \Sigma_n\|^{1/2}}{\sqrt{n}}.
\end{align*}

\qed

\end{proof}

In the proofs of concentration inequalities for smooth functionals of sample covariance $\hat \Sigma_n,$ it is convenient to assume that  covariance operator $\Sigma$ is of finite rank $N$ (the general case could be then handled by a finite rank approximation).
In this case, we can write $X=\Sigma^{1/2}Z,$ $Z=\sum_{j=1}^N g_j \phi_j,$ where $g_1,g_2,\dots$ are i.i.d. standard normal random variables 
and $\phi_j, j\geq 1$ is an orthonormal basis of ${\mathbb H}.$ We can also write $X_j=\Sigma^{1/2}Z_j, j\geq 1,$ where $Z_1,Z_2, \dots$
are i.i.d. copies of $Z.$ Finally, denote ${\mathcal Z}:= (Z_1,\dots, Z_n).$ Clearly, ${\mathcal Z}$ takes values in an $n N$-dimensional Euclidean 
space and it can be identified with a standard normal random variable in ${\mathbb R}^{nN},$ so, the bound of Proposition \ref{L_p_conce}
can be applied to locally Lipschitz functions of ${\mathcal Z}.$
Let 
\begin{align*}
\tilde \Sigma_n({\mathcal Z}):= \hat \Sigma_n(\Sigma^{1/2}Z_1, \dots, \Sigma^{1/2} Z_n).
\end{align*}
Then, it follows from the bound of Proposition \ref{prop_one} that 
\begin{align}
\label{L_Sigma_Z}
(L\tilde \Sigma_n)({\mathcal Z})\leq \frac{2\|\Sigma\|^{1/2}\|\hat \Sigma_n\|^{1/2}}{\sqrt{n}}.
\end{align}

The next proposition provides a concentration bound for Schatten norm errors $\|\hat \Sigma_n-\Sigma\|_r$ of sample covariance 
for all $r\geq 2.$

\begin{proposition}
\label{conc_Schatten_hat}
Suppose ${\bf r}(\Sigma)\lesssim n.$
Then, for all $p\geq 1$ and $r\in [2,+\infty],$
\begin{align*}
\Bigl\|\|\hat \Sigma_n-\Sigma\|_r - {\mathbb E}\|\hat \Sigma_n-\Sigma\|_r\Bigr\|_{L_p} \lesssim \|\Sigma\|\Bigl(\sqrt{\frac{p}{n}}\vee \frac{p}{n}\Bigr).
\end{align*}
\end{proposition}

\begin{proof}
Assume that $\Sigma$ is of finite rank $N$ and represent $\hat \Sigma_n$ as a function of ${\mathcal Z}:$ $\hat \Sigma_n= \tilde \Sigma_n({\mathcal Z}).$ 
Since, for $r\geq 2,$ 
\begin{align*}
|\|A\|_r - \|B\|_r| \leq \|A-B\|_r \leq \|A-B\|_2,
\end{align*}
bound \eqref{L_Sigma_Z} implies that 
\begin{align*}
(L\|\tilde \Sigma_n-\Sigma\|_r)({\mathcal Z})\leq \frac{2\|\Sigma\|^{1/2}\|\hat \Sigma_n\|^{1/2}}{\sqrt{n}}.
\end{align*}
Hence, we can use the bound of Proposition \ref{L_p_conce} to get 
\begin{align}
\label{bd_co_co}
&
\nonumber
\Bigl\|\|\hat \Sigma_n-\Sigma\|_r - {\mathbb E}\|\hat \Sigma_n-\Sigma\|_r\Bigr\|_{L_p} \lesssim \|\Sigma\|^{1/2}\frac{\sqrt{p}}{\sqrt{n}} \|\|\hat \Sigma_n\|^{1/2}\|_{L_p}
\\
&
\lesssim \|\Sigma\|^{1/2} {\mathbb E}\|\hat \Sigma_n\|^{1/2}\sqrt{\frac{p}{n}} + \|\Sigma\|^{1/2}\|\|\hat \Sigma_n\|^{1/2}-{\mathbb E}\|\hat \Sigma_n\|^{1/2}\|_{L_p}\sqrt{\frac{p}{n}}.
\end{align}
Also, using bound \eqref{L_Sigma_Z} along with the bound 
\begin{align*}
\Bigl|\|A\|^{1/2}- \|B\|^{1/2}\Bigr| \leq \frac{\|A-B\|}{\|A\|^{1/2}+ \|B\|^{1/2}}  \leq \frac{\|A-B\|_2}{\|A\|^{1/2}+ \|B\|^{1/2}} 
\end{align*}
that holds for all self-adjoint positively semidefinite operators $A,B,$ it is easy to see that 
\begin{align}
\label{bd_dva}
(L\|\tilde \Sigma_n\|^{1/2})({\mathcal Z})\leq \frac{\|\Sigma\|^{1/2}}{\sqrt{n}}.
\end{align}
Using again the bound of Proposition \ref{L_p_conce}, we get 
\begin{align*}
\|\|\hat \Sigma_n\|^{1/2}-{\mathbb E}\|\hat \Sigma_n\|^{1/2}\|_{L_p} \lesssim \|\Sigma\|^{1/2} \sqrt{\frac{p}{n}}.
\end{align*}
It is enough to substitute the last bound into bound \eqref{bd_co_co} and also observe that 
\begin{align}
\label{bd_hat_1/2}
\nonumber
{\mathbb E}\|\hat \Sigma_n\|^{1/2} &\leq {\mathbb E}^{1/2}\|\hat \Sigma_n\|
\leq \|\Sigma\|^{1/2} + {\mathbb E}^{1/2}\|\hat \Sigma_n-\Sigma\|
\\
&
\lesssim \|\Sigma\|^{1/2} + \|\Sigma\|^{1/2} \Bigl(\sqrt{\frac{{\bf r}(\Sigma)}{n}}\vee \frac{{\bf r}(\Sigma)}{n}\Bigr)^{1/2}\lesssim \|\Sigma\|^{1/2},
\end{align}
to complete the proof.

\qed
\end{proof}

\begin{remark}
\normalfont
For $r=\infty,$ the inequality of Proposition \ref{conc_Schatten_hat} also follows 
from bound \eqref{KL_2}.
\end{remark}

We now turn to the proof of Theorem \ref{prop_rem}.

\begin{proof}
First we assume that ${\rm rank}(\Sigma)=N<\infty.$ Let $\lambda_1\geq \dots \geq \lambda_N>0$ be the non-zero 
eigenvalues of $\Sigma$ (repeated with their multiplicities) and let $\phi_1,\dots, \phi_N$ be the corresponding orthonormal 
eigenvectors. In this case, $X= \sum_{k=1}^N \sqrt{\lambda_k} g_k \phi_k,$ where $g_1,\dots, g_N$ are i.i.d. standard normal 
random variables and $X_j= \sum_{k=1}^N \sqrt{\lambda_k} g_{j,k} \phi_k, j=1,\dots, n,$ $\{g_{j,k}: 1\leq k\leq N, j=1,\dots, n\}$
being also i.i.d. standard normal. Since $X, X_1,\dots , X_n$ take values in $L_N:= {\rm l.s.}(\{\phi_1,\dots, \phi_N\})\subset {\mathbb H},$ we can assume that operators 
$\Sigma$ and $\hat \Sigma_n$ act in subspace $L_N.$ Let $f\in C^1({\mathbb R}).$ Then, the derivative $D_H R_f(A,H)[H_1]$ 
of function $H\mapsto  R_f(A,H)$ in the direction of operator $H_1$ is well defined for all self-adjoint operators $H,H_1: L_N\mapsto L_N$
and is given by the following formula:
\begin{align*}
D_H R_f(A,H)[H_1] = \langle f'(A+H)-f'(A), H_1\rangle.
\end{align*}
Therefore, the local Lipschitz constant of the map ${\mathcal S}_2(L_N)\ni H\mapsto R_f(A,H)$ could be bounded 
as follows:
\begin{align*}
(L R_f(A,\cdot))(H) \leq \|f'(A+H)-f'(A)\|_2.
\end{align*}
Together with \eqref{L_Sigma_Z}, this implies that 
\begin{align*}
(L R_f(\Sigma, \tilde \Sigma_n-\Sigma))({\mathcal Z})\leq 
\frac{2\|\Sigma\|^{1/2}\|\hat \Sigma_n\|^{1/2}}{\sqrt{n}}
\|f'(\hat \Sigma_n)-f'(\Sigma)\|_2,
\end{align*}
where ${\mathcal Z}:= (Z_1,\dots, Z_n),$ $Z_j=\sum_{k=1}^n g_{j,k}\phi_k, j=1,\dots, n$ and $\tilde \Sigma_n({\mathcal Z}):= \hat \Sigma_n(\Sigma^{1/2}Z_1, \dots, \Sigma^{1/2} Z_n).$

Thus, it follows from Proposition \ref{L_p_conce} that the following bound holds for all $p\geq 1:$
\begin{align}
\label{conc_rem_start}
\nonumber
\Bigl\|R_f(\Sigma, \hat \Sigma_n-\Sigma)-{\mathbb E}R_f(\Sigma, \hat \Sigma_n-\Sigma)\Bigr\|_{L_p} &\lesssim \sqrt{p} \Bigl\|(L R_f(\Sigma, \tilde \Sigma_n-\Sigma))({\mathcal Z})\Bigr\|_{L_p}
\\
&
\lesssim \sqrt{p}\frac{\|\Sigma\|^{1/2}}{\sqrt{n}}
\Bigl\|\|\hat \Sigma_n\|^{1/2}\|f'(\hat \Sigma_n)-f'(\Sigma)\|_2\Bigr\|_{L_p}. 
\end{align}
Note that 
\begin{align*}
{\mathbb E}\|\hat \Sigma_n\|^{1/2}\|f'(\hat \Sigma_n)-f'(\Sigma)\|_2 \leq 
{\mathbb E}^{1/2}\|\hat \Sigma_n\|\ {\mathbb E}^{1/2}\|f'(\hat \Sigma_n)-f'(\Sigma)\|_2^2.
\end{align*}
By bound \eqref{bd_hat_1/2},
we also have 
$
{\mathbb E}^{1/2}\|\hat \Sigma_n\|\lesssim \|\Sigma\|^{1/2} 
$
provided that ${\bf r}(\Sigma)\lesssim n.$ If $f'$ is a Lipschitz function, then the operator function  $A\mapsto f'(A)$ is also 
Lipschitz with respect to the Hilbert-Schmidt norm (see \cite{Pot_Suk}). Therefore, using bound of Proposition \ref{HS_error},  we get 
\begin{align*}
{\mathbb E} \|f'(\hat \Sigma_n)-f'(\Sigma)\|_2^2 \leq \|f'\|_{\rm Lip}^2 {\mathbb E}\|\hat \Sigma_n-\Sigma\|_2^2
\leq \|f'\|_{\rm Lip}^2\|\Sigma\|^2\frac{{\bf r}(\Sigma)^2}{n}. 
\end{align*}  
Thus, we get 
\begin{align*}
{\mathbb E}^{1/2} \|f'(\hat \Sigma_n)-f'(\Sigma)\|_2^2 \lesssim \|f'\|_{\rm Lip} \|\Sigma\|\frac{{\bf r}(\Sigma)}{\sqrt{n}}
\end{align*}
and 
\begin{align}
\label{expect_prod}
{\mathbb E}\|\hat \Sigma_n\|^{1/2}\|f'(\hat \Sigma_n)-f'(\Sigma)\|_2 \lesssim  \|f'\|_{\rm Lip} \|\Sigma\|^{3/2}\frac{{\bf r}(\Sigma)}{\sqrt{n}}.
\end{align}
We now use Gaussian concentration bound of Proposition \ref{L_p_conce}
to control 
\begin{align*}
\Bigl\|\|\hat \Sigma_n\|^{1/2}\|f'(\hat \Sigma_n)-f'(\Sigma)\|_2- {\mathbb E}\|\hat \Sigma_n\|^{1/2}\|f'(\hat \Sigma_n)-f'(\Sigma)\|_2\Bigr\|_{L_p}.
\end{align*}
It is easy to get from bound \eqref{L_Sigma_Z} and the Lipschitz property of operator function $f'(A)$ that  
\begin{align}
\label{bd_odin}
(L\|f'(\tilde \Sigma_n)-f'(\Sigma)\|_2)({\mathcal Z}) 
\leq 2\|f'\|_{\rm Lip} 
\frac{\|\Sigma\|^{1/2}\|\hat \Sigma_n\|^{1/2}}{\sqrt{n}}.
\end{align} 
It follows from \eqref{bd_odin} and \eqref{bd_dva} that 
\begin{align*}
(L\|\tilde \Sigma_n\|^{1/2}\|f'(\tilde \Sigma_n)-f'(\Sigma)\|_2)({\mathcal Z}) 
&
\leq 
\frac{\|\Sigma\|^{1/2}}{\sqrt{n}}\|f'(\hat \Sigma_n)-f'(\Sigma)\|_2 + 2\|f'\|_{\rm Lip} 
\frac{\|\Sigma\|^{1/2}\|\hat \Sigma_n\|}{\sqrt{n}}
\\
&
\leq \|f'\|_{\rm Lip} \frac{\|\Sigma\|^{1/2}}{\sqrt{n}}(\|\hat \Sigma_n-\Sigma\|_2+ 2\|\hat \Sigma_n\|).
\end{align*}
Using again Proposition \ref{L_p_conce}, we get 
\begin{align}
\label{conc_prod}
&
\nonumber
\Bigl\|\|\hat \Sigma_n\|^{1/2}\|f'(\hat \Sigma_n)-f'(\Sigma)\|_2- {\mathbb E}\|\hat \Sigma_n\|^{1/2}\|f'(\hat \Sigma_n)-f'(\Sigma)\|_2\Bigr\|_{L_p}
\\
&
\lesssim \sqrt{p} \|f'\|_{\rm Lip} \frac{\|\Sigma\|^{1/2}}{\sqrt{n}}\Bigl(\Bigl\|\|\hat \Sigma_n-\Sigma\|_2\Bigr\|_{L_p}+\Bigl\|\|\hat \Sigma_n\|\Bigr\|_{L_p}\Bigr).
\end{align}
It follows from Proposition \ref{conc_Schatten_hat} (for $r=\infty$) and \eqref{KL_1} that, under the assumption ${\bf r}(\Sigma)\lesssim n,$
\begin{align}
\label{co_op}
\nonumber
\Bigl\|\|\hat \Sigma_n\|\Bigr\|_{L_p} &\leq \|\Sigma\| + {\mathbb E}\|\hat \Sigma_n-\Sigma\|+\Bigl\|\|\hat \Sigma_n-\Sigma\| -{\mathbb E}\|\hat \Sigma_n-\Sigma\|\Bigr\|_{L_p}
\\
&
\leq \|\Sigma\| + C\|\Sigma\|\Bigl(\sqrt{\frac{{\bf r}(\Sigma)}{n}}\vee \sqrt{\frac{p}{n}}\vee \frac{p}{n}\Bigr)
\lesssim \|\Sigma\| \Bigl(\sqrt{\frac{p}{n}}\vee \frac{p}{n}\vee 1\Bigr).
\end{align}
On the other hand, the bounds of Propositions \ref{HS_error} and Proposition \ref{conc_Schatten_hat} (for $r=2$) 
imply that 
\begin{align}
\label{co_HS}
&
\nonumber
\Bigl\|\|\hat \Sigma_n-\Sigma\|_2\Bigr\|_{L_p} \leq {\mathbb E}\|\hat \Sigma_n-\Sigma\|_2 + \Bigl\|\|\hat \Sigma_n-\Sigma\|_2 -{\mathbb E}\|\hat \Sigma_n-\Sigma\|_2\Bigr\|_{L_p}
\\
&
\lesssim 
\|\Sigma\| \Bigl(\frac{{\bf r}(\Sigma)}{\sqrt{n}} \vee \sqrt{\frac{p}{n}}\vee \frac{p}{n}\Bigr).
\end{align}
Thus, we can use bounds \eqref{conc_prod}, \eqref{co_op} and \eqref{co_HS} to conclude that 
\begin{align*}
&
\Bigl\|\|\hat \Sigma_n\|^{1/2}\|f'(\hat \Sigma_n)-f'(\Sigma)\|_2- {\mathbb E}\|\hat \Sigma_n\|^{1/2}\|f'(\hat \Sigma_n)-f'(\Sigma)\|_2\Bigr\|_{L_p}
\\
&
\lesssim \|f'\|_{\rm Lip} \|\Sigma\|^{3/2}\Bigl(\frac{{\bf r}(\Sigma)}{\sqrt{n}}\sqrt{\frac{p}{n}} \vee \sqrt{\frac{p}{n}}\vee \Bigl(\frac{p}{n}\Bigr)^{3/2}\Bigr).
\end{align*}
Combining the last bound with \eqref{expect_prod}, we get 
\begin{align*}
&
\Bigl\|\|\hat \Sigma_n\|^{1/2}\|f'(\hat \Sigma_n)-f'(\Sigma)\|_2\Bigr\|_{L_p}
\\
&
\lesssim \|f'\|_{\rm Lip} \|\Sigma\|^{3/2}\Bigl(\frac{{\bf r}(\Sigma)}{\sqrt{n}}\Bigl(\sqrt{\frac{p}{n}}\vee 1\Bigr) \vee \sqrt{\frac{p}{n}}\vee \Bigl(\frac{p}{n}\Bigr)^{3/2}\Bigr).
\end{align*}
Together with \eqref{conc_rem_start}, this yields bound \eqref{conc_rem_end} in the case of covariance $\Sigma$ of finite rank.

In the general case, let $\lambda_1\geq \lambda_2\geq \dots \geq 0$  be the eigenvalues of $\Sigma$ (repeated with their multiplicities) and let $\phi_1,\phi_2,\dots$ be an orthonormal basis of corresponding 
eigenvectors. Then, $X=\sum_{k\geq 1}\sqrt{\lambda_k} g_k \phi_k,$ $\{g_k\}$ being i.i.d. standard normal r.v. Denote $X^{(N)}:= \sum_{k=1}^N \sqrt{\lambda_k} g_k \phi_k$ and 
let 
\begin{align*}
\Sigma^{(N)}:= {\mathbb E}(X^{(N)}\otimes X^{(N)})= \sum_{k=1}^N \lambda_k (\phi_k\otimes \phi_k).
\end{align*} 
Define $X_j^{(N)}:= \sum_{k=1}^N \sqrt{\lambda_k} g_{j,k} \phi_k, j=1,\dots, n,$ where $\{g_{j,k}\}$ are i.i.d. standard normal r.v. and let $\hat \Sigma_n^{(N)}$ be the sample covariance 
based on $X_1^{(N)}, \dots, X_n^{(N)}.$ Then, 
\begin{align}
\label{conc_rem_N}
&
\nonumber
\Bigl\|R_f(\Sigma^{(N)}, \hat \Sigma_n^{(N)}-\Sigma^{(N)})-{\mathbb E}R_f(\Sigma^{(N)}, \hat \Sigma_n^{(N)}-\Sigma^{(N)})\Bigr\|_{L_p} 
\\
&
\lesssim 
\|f'\|_{\rm Lip} \|\Sigma^{(N)}\|^2\Bigl(\frac{{\bf r}(\Sigma^{(N)})}{\sqrt{n}}\Bigl(\sqrt{\frac{p}{n}}\vee \frac{p}{n}\Bigr) \vee \frac{p}{n}\vee \Bigl(\frac{p}{n}\Bigr)^{2}\Bigr),
\end{align}
and it remains to pass in the above inequality to the limit as $N\to \infty.$ 

Observe that $\|\Sigma^{(N)}\|=\lambda_1=\|\Sigma\|$ and 
\begin{align*}
{\bf r}(\Sigma^{(N)}) = \sum_{j=1}^N \frac{\lambda_j}{\lambda_1} \to \sum_{j=1}^{\infty} \frac{\lambda_j}{\lambda_1}= {\bf r}(\Sigma)\ {\rm as}\ N\to\infty.
\end{align*}
We will prove that 
\begin{align}
\label{rem_N_inf}
R_f(\Sigma^{(N)}, \hat \Sigma_n^{(N)}-\Sigma^{(N)}) \to R_f(\Sigma, \hat \Sigma_n-\Sigma)\ {\rm as}\ N\to\infty\ {\rm a.s.}
\end{align} 
Indeed, since $f(0)=0,$ $f\in C^1({\mathbb R}_+)$ and $\|f'\|_{\rm Lip}<\infty,$ 
we have 
\begin{align*}
|f'(x)| \leq |f'(0)| + \|f'\|_{\rm Lip}x\ {\rm and}\  |f(x)|\leq |f'(0)|x+ \|f'\|_{\rm LIp}\frac{x^2}{2}, x\geq 0.
\end{align*}
Therefore,
\begin{align}
\label{odin_N}
|\tau_f(\Sigma)-\tau_f(\Sigma^{(N)})| \leq  \sum_{k>N} |f(\lambda_k)| \leq |f'(0)|\sum_{k>N} \lambda_k + \frac{\|f'\|_{\rm Lip}}{2} \sum_{k>N} \lambda_k^2\to 0\ {\rm as}\ N\to \infty.
\end{align} 
Note also that 
\begin{align*}
{\mathbb E}\|X-X^{(N)}\|\leq {\mathbb E}^{1/2}\|X-X^{N}\|^2 = (\sum_{k>N} \lambda_k)^{1/2}\to 0\ {\rm as}\ N\to \infty.
\end{align*} 
By Gaussian concentration, it is easy to conclude that $X^{(N)}\to X$ a.s. and, similarly, $X_j^{(N)}\to X_j$ a.s. for  $j=1,\dots, n$ as $N\to \infty.$
This easily implies that $\|\hat \Sigma_n^{(N)}-\hat \Sigma_n\|\to 0$ as $N\to \infty,$ and, since ${\rm rank}(\hat \Sigma_n^{(N)})\leq n, {\rm rank}(\hat \Sigma_n)\leq n,$
the convergence also holds in Hilbert-Schmidt and nuclear norms. Note also that 
\begin{align*}
\|\hat \Sigma_n\|\leq n^{-1}\sum_{j=1}^n \|X_j\otimes X_j\|= n^{-1}\sum_{j=1}^n \|X_j\|^2 =: A
\end{align*}
and, similarly,
\begin{align*}
\|\hat \Sigma_n^{(N)}\|\leq n^{-1}\sum_{j=1}^n \|X_j^{(N)}\|^2 \leq n^{-1}\sum_{j=1}^n \|X_j\|^2 =A.
\end{align*}
Since $f$ is Lipschitz on $[0,A],$ it can be extended to a Lipschitz function $\bar f: {\mathbb R}_+\mapsto {\mathbb R}$ with preservation of 
its Lipschitz constant (that might depend on $A$). Since $\sigma (\hat \Sigma_n)\subset [0,A]$ and $\sigma(\hat \Sigma_n^{(N)})\subset [0,A],$
we have $f(\hat \Sigma_n)= \bar f (\hat \Sigma_n)$ and $f(\hat \Sigma_n^{(N)})=\bar f (\hat \Sigma_n^{(N)}).$ 
Therefore, taking also into account that ${\rm rank}(f(\hat \Sigma_n))\leq n$ and ${\rm rank}(f(\hat \Sigma_n^{(N)}))\leq n,$ we get
\begin{align}
\label{dva_N}
\nonumber
&
|\tau_f (\hat \Sigma_n^{(N)})- \tau_f(\hat \Sigma_n)|= |{\rm tr}(\bar f(\hat \Sigma_n^{(N)}))- {\rm tr} (\bar f(\hat \Sigma_n))|\leq 
\|\bar f(\hat \Sigma_n^{(N)})-\bar f(\hat \Sigma_n)\|_1 
\\
&
\leq \sqrt{2n} \|\bar f(\hat \Sigma_n^{(N)})-\bar f(\hat \Sigma_n)\|_2 \leq \sqrt{2n}\|\bar f\|_{\rm Lip} \|\hat \Sigma_n^{(N)}-\hat \Sigma_n\|_2
=\sqrt{2n}\|f\|_{\rm Lip([0,A])} \|\hat \Sigma_n^{(N)}-\hat \Sigma_n\|_2\to 0\ {\rm a.s.}
\end{align}
as $N\to\infty.$
Finally, note that 
\begin{align*}
\|f'(\Sigma^{(N)})-f'(\Sigma)\|_2 
\leq \|f'\|_{\rm Lip} \|\Sigma^{(N)}-\Sigma\|_2\to 0\ {\rm as}\ N\to\infty
\end{align*}
and 
\begin{align*}
\|(\hat \Sigma_n^{(N)}-\Sigma^{(N)})-(\hat \Sigma_n-\Sigma)\|_2 \leq \|\hat \Sigma_n^{(N)}-\hat \Sigma_n\|_2 + \|\Sigma^{(N)}-\Sigma\|_2 \to 0\ {\rm as}\ N\to\infty,
\end{align*}
implying that 
\begin{align}
\label{tri_N} 
\langle f'(\Sigma^{(N)}), \hat \Sigma_n^{(N)}-\Sigma^{(N)}\rangle \to  \langle f'(\Sigma), \hat \Sigma_n - \Sigma\rangle \ {\rm as}\ N\to\infty \ {\rm a.s.}
\end{align} 
It follows from \eqref{odin_N}, \eqref{dva_N} and \eqref{tri_N} that \eqref{rem_N_inf} holds. 
 
 Next we show that 
 \begin{align}
\label{exepect_rem_N_inf}
{\mathbb E}R_f(\Sigma^{(N)}, \hat \Sigma_n^{(N)}-\Sigma^{(N)}) \to {\mathbb E}R_f(\Sigma, \hat \Sigma_n-\Sigma)\ {\rm as}\ N\to\infty.
\end{align} 
For this, note that, in view of Proposition \ref{rem_first_order}, 
\begin{align*} 
|R_f(\Sigma^{(N)}, \hat \Sigma_n^{(N)}-\Sigma^{(N)})| \leq \frac{\|f'\|_{\rm Lip}}{2} \|\hat \Sigma_n^{(N)}-\Sigma^{(N)}\|_2^2. 
\end{align*} 
Using bound  \eqref{co_HS}  with $p=4,$ we get 
\begin{align*}
\sup_{N\geq 1}{\mathbb E}|R_f(\Sigma^{(N)}, \hat \Sigma_n^{(N)}-\Sigma^{(N)})|^2 &\leq \frac{\|f'\|_{\rm Lip}^2}{4} 
\sup_{N\geq 1}{\mathbb E}\|\hat \Sigma_n^{(N)}-\Sigma^{(N)}\|_2^4
=\frac{\|f'\|_{\rm Lip}^2}{4}\sup_{N\geq 1}\Bigl\|\|\hat \Sigma_n^{(N)}-\Sigma^{(N)}\|_2\Bigr\|_{L_4}^4
\\
&
\lesssim \|f'\|_{\rm Lip}^2\|\Sigma\|^4 \sup_{N\geq 1}\frac{{\bf r}(\Sigma^{(N)})^4}{n^2}\lesssim \|f'\|_{\rm Lip}^2\|\Sigma\|^4 \frac{{\bf r}(\Sigma)^4}{n^2} <\infty.
\end{align*}
Thus, one can pass to the limit as $N\to\infty$ under the expectation to prove \eqref{exepect_rem_N_inf}.

Finally, observe that, by bound \eqref{conc_rem_N} for all $p'>p\geq 1,$
\begin{align*}
&
\sup_{N\geq 1}\Bigl\|R_f(\Sigma^{(N)}, \hat \Sigma_n^{(N)}-\Sigma^{(N)})-{\mathbb E}R_f(\Sigma^{(N)}, \hat \Sigma_n^{(N)}-\Sigma^{(N)})\Bigr\|_{L_{p'}} 
\\
&
\lesssim 
\|f'\|_{\rm Lip} \sup_{N\geq 1}\|\Sigma^{(N)}\|^2\Bigl(\frac{{\bf r}(\Sigma^{(N)})}{\sqrt{n}}\Bigl(\sqrt{\frac{p'}{n}}\vee \frac{p'}{n}\Bigr) \vee \frac{p'}{n}\vee \Bigl(\frac{p'}{n}\Bigr)^{2}\Bigr)
\\
&
\lesssim 
\|f'\|_{\rm Lip} \|\Sigma\|^2\Bigl(\frac{{\bf r}(\Sigma)}{\sqrt{n}}\Bigl(\sqrt{\frac{p'}{n}}\vee \frac{p'}{n}\Bigr) \vee \frac{p'}{n}\vee \Bigl(\frac{p'}{n}\Bigr)^{2}\Bigr)<\infty.
\end{align*}
Since also 
\begin{align}
R_f(\Sigma^{(N)}, \hat \Sigma_n^{(N)}-\Sigma^{(N)})-{\mathbb E}R_f(\Sigma^{(N)}, \hat \Sigma_n^{(N)}-\Sigma^{(N)}) \to R_f(\Sigma, \hat \Sigma_n-\Sigma)-{\mathbb E} R_f(\Sigma, \hat \Sigma_n-\Sigma)  
\end{align} 
as $N\to \infty$ a.s., it is possible to pass to the limit as $N\to\infty$ in bound \eqref{conc_rem_N} to complete the proof.
\qed
\end{proof}

The next corollary is immediate.

\begin{corollary}
\label{conc_tau_f}
Let $f\in C^{1}({\mathbb R}_+)$ be such that $f(0)=0,$ $\|f'\|_{L_{\infty}}<\infty$ and $\|f'\|_{\rm Lip}<\infty.$ 
Suppose that ${\bf r}(\Sigma)\lesssim n.$
Then, for all $p\geq 1,$
\begin{align*}
&
\Bigl\|\tau_f(\hat \Sigma_n)-{\mathbb E}\tau_f(\hat \Sigma_n)\Bigr\|_{L_p} 
\\
&
\lesssim_m 
\Bigl(\|\Sigma f'(\Sigma)\|_2\sqrt{\frac{p}{n}}\vee \|\Sigma f'(\Sigma)\|\frac{p}{n}\Bigr)
+\|f'\|_{\rm Lip} \|\Sigma\|^2
\Bigl(\frac{{\bf r}(\Sigma)}{\sqrt{n}}\sqrt{\frac{p}{n}}\vee \Bigl(\frac{{\bf r}(\Sigma)}{\sqrt{n}}\vee 1\Bigr)\frac{p}{n} \vee\Bigl(\frac{p}{n}\Bigr)^{2}\Bigr)
\\
&
\lesssim \|f'\|_{L_{\infty}} \|\Sigma\|\Bigl(\sqrt{{\bf r}(\Sigma^2)}\sqrt{\frac{p}{n}}\vee \frac{p}{n}\Bigr)
+\|f'\|_{\rm Lip} \|\Sigma\|^2
\Bigl(\frac{{\bf r}(\Sigma)}{\sqrt{n}}\sqrt{\frac{p}{n}}\vee \Bigl(\frac{{\bf r}(\Sigma)}{\sqrt{n}}\vee 1\Bigr)\frac{p}{n} \vee\Bigl(\frac{p}{n}\Bigr)^{2}\Bigr).
\end{align*}
\end{corollary}

\begin{proof}
Using the first order Taylor expansion, we get 
\begin{align*}
\tau_f(\hat \Sigma_n)-{\mathbb E}\tau_f(\hat \Sigma_n) = \langle f'(\Sigma),\hat \Sigma_n-\Sigma\rangle + R_f(\Sigma, \hat \Sigma_n-\Sigma)- 
{\mathbb E}R_f(\Sigma, \hat \Sigma_n-\Sigma).
\end{align*}
It remains to use bounds of Proposition \ref{prop_lin} and Theorem \ref{prop_rem} and to observe that 
$\|\Sigma f'(\Sigma)\|_2 \leq \|f'\|_{L_{\infty}}\|\Sigma\|_2 = \|f'\|_{L_{\infty}}\|\Sigma\| \sqrt{{\bf r}(\Sigma^2)}$ and $\|\Sigma f'(\Sigma)\|\leq \|f'\|_{L_{\infty}}\|\Sigma\|.$


\qed
\end{proof}

The results of this section also imply the following simple proposition (that itself implies Proposition \ref{Main_m=2} of Section \ref{Main_results}).

\begin{proposition}
\label{Main_m=2_detailed} 
Let $f\in C^1({\mathbb R}_+)$ with $f(0)=0$ and $\|f'\|_{\rm Lip}<\infty.$
Then, for all $p\geq 1,$
\begin{align*}
\Bigl\|\tau_f(\hat \Sigma_n)-\tau_f(\Sigma)- \langle f'(\Sigma), \hat \Sigma_n-\Sigma\rangle\Bigr\|_{L_p}
\lesssim \|f'\|_{\rm Lip} \|\Sigma\|^2 \Bigl(\frac{{\bf r}(\Sigma)^2}{n} \vee 
\frac{p}{n} \vee \Bigl(\frac{p}{n}\Bigr)^2\Bigr).
\end{align*}
If, in addition, $\|f'\|_{L_{\infty}}<\infty,$ this implies that 
\begin{align*}
\|\tau_f(\hat \Sigma_n)-\tau_f(\Sigma)\|_{L_p}
&\lesssim \Bigl(\|\Sigma f'(\Sigma)\|_2\sqrt{\frac{p}{n}}\vee \|\Sigma f'(\Sigma)\|\frac{p}{n}\Bigr)+ 
\|f'\|_{\rm Lip} \|\Sigma\|^2 \Bigl(\frac{{\bf r}(\Sigma)^2}{n} \vee 
\frac{p}{n} \vee \Bigl(\frac{p}{n}\Bigr)^2\Bigr)
\\
&
\lesssim  \|f'\|_{L_{\infty}} \|\Sigma\|\Bigl(\sqrt{{\bf r}(\Sigma^2)}\sqrt{\frac{p}{n}}\vee \frac{p}{n}\Bigr)+ 
\|f'\|_{\rm Lip} \|\Sigma\|^2 \Bigl(\frac{{\bf r}(\Sigma)^2}{n} \vee 
\frac{p}{n} \vee \Bigl(\frac{p}{n}\Bigr)^2\Bigr),
\end{align*}
and, moreover, 
\begin{align*}
&
W_p \Bigl(\sqrt{n/2}(\tau_f(\hat \Sigma_n)-\tau_f(\Sigma)), G_{\Sigma}(f)\Bigr)
\lesssim \frac{\|\Sigma f'(\Sigma)\|_2\sqrt{p}\vee \|\Sigma f'(\Sigma)\| p}{\sqrt{n}}
+\|f'\|_{\rm Lip} \|\Sigma\|^2 \Bigl(\frac{{\bf r}(\Sigma)^2}{\sqrt{n}} \vee 
\frac{p}{\sqrt{n}} \vee \frac{p^2}{n^{3/2}}\Bigr)
\\
&
\lesssim \|f'\|_{L_{\infty}} \|\Sigma\|\Bigl(\sqrt{{\bf r}(\Sigma^2)}\sqrt{\frac{p}{n}}\vee \frac{p}{\sqrt{n}}\Bigr)
+\|f'\|_{\rm Lip} \|\Sigma\|^2 \Bigl(\frac{{\bf r}(\Sigma)^2}{\sqrt{n}} \vee 
\frac{p}{\sqrt{n}} \vee \frac{p^2}{n^{3/2}}\Bigr).
\end{align*}
\end{proposition}

\begin{proof}
It easily follows from Proposition \ref{rem_first_order} and \eqref{rem_N_inf} that a.s.
\begin{align}
\label{bd_on_R_f_gen}
|R_f(\Sigma, \hat \Sigma_n-\Sigma)| \leq \frac{\|f'\|_{\rm Lip}}{2} \|\hat \Sigma_n-\Sigma\|_2^2.
\end{align}
Therefore, using the bounds of propositions \ref{HS_error} and \ref{conc_Schatten_hat}, we get
\begin{align*}
&
\Bigl\|\tau_f(\hat \Sigma_n)-\tau_f(\Sigma)- \langle f'(\Sigma), \hat \Sigma_n-\Sigma\rangle\Bigr\|_{L_p}
= \Bigl\|R_f(\Sigma, \hat \Sigma_n-\Sigma)\Bigr\|_{L_p} \leq \frac{\|f'\|_{\rm Lip}}{2} \Bigl\|\|\hat \Sigma_n-\Sigma\|_2^2\Bigr\|_{L_p}
\\
&
=\frac{\|f'\|_{\rm Lip}}{2} \Bigl\|\|\hat \Sigma_n-\Sigma\|_2\Bigr\|_{L_{2p}}^2 \lesssim \|f'\|_{\rm Lip} \|\Sigma\|^2 \Bigl(\frac{{\bf r}(\Sigma)^2}{n} \vee 
\frac{p}{n} \vee \Bigl(\frac{p}{n}\Bigr)^2\Bigr).
\end{align*}
The last two bounds follow from propositions \ref{prop_lin} and \ref{norm_approx_lin}.
\qed 
\end{proof}

\section{Bounds on sup-norms of empirical spectral processes}
\label{sup_bounds}

Our next goal is to develop bounds on the sup-norms of empirical spectral processes, such as 
\begin{align*}
\tau_f (\hat \Sigma_n)- {\mathbb E} \tau_f(\hat \Sigma_n) = \int_{\mathbb R_+}f d\mu_{\hat \Sigma_n} -  {\mathbb E}\int_{\mathbb R_+}f d\mu_{\hat \Sigma_n}, f\in {\mathcal F}
\end{align*}
over a class ${\mathcal F}$ of sufficiently smooth functions $f:{\mathbb R}_+\mapsto {\mathbb R}$ satisfying the condition $f(0)=0.$ 
This problem could be reduced to bounding the sup-norms of the linear process 
$\langle f'(\Sigma), \hat \Sigma_n-\Sigma\rangle, f\in {\mathcal F}$ and the centered remainder process
$R_f(\Sigma, \hat \Sigma_n-\Sigma)-{\mathbb E}R_f(\Sigma, \hat \Sigma_n-\Sigma), f\in {\mathcal F}.$

Consider the following pseudometric: for $f,g\in C^1({\mathbb R}_+),$
\begin{align*}
\tau(f,g):= \|f'-g'\|_{L_{\infty}}=\sup_{x\in {\mathbb R}_{+}}|f'(x)-g'(x)|,
\end{align*}
and denote 
\begin{align*}
\rho_{\tau}=\rho_{\tau}({\mathcal F}):=\sup_{f\in {\mathcal F}}\|f'\|_{L_{\infty}}.
\end{align*}
We will use standard notations $N_{\tau}({\mathcal F}, \eps)$ for the $\eps$-covering number of ${\mathcal F}$ with respect to pseudometric $\tau$ and 
\begin{align*}
H_{\tau}({\mathcal F},\eps):= \log N_{\tau}({\mathcal F}, \eps), \eps>0
\end{align*}
be the corresponding $\eps$-entropy of ${\mathcal F}.$ Similar notations will be used for other pseudometrics. 

The following result for the process $\langle f'(\Sigma), \hat \Sigma_n-\Sigma\rangle, f\in {\mathcal F}$ holds.

\begin{proposition}
\label{entr_bd_lin_part}
Let ${\mathcal F}\subset C^1({\mathbb R}_+)$ be a class of functions $f$ satisfying the condition $f(0)=0.$
Suppose ${\bf r}(\Sigma)\lesssim n$ and let $\delta\in [0,\rho_{\tau}/2].$ For all $p\geq 1,$
\begin{align*}
&
\Big\|\sup_{f\in {\mathcal F}}|\langle f'(\Sigma), \hat \Sigma_n-\Sigma\rangle|\Bigr\|_{L_p}
\lesssim 
\|\Sigma\|_2 \sqrt{\frac{p}{n}} \int_{\delta}^{\rho_{\tau}} H_{\tau}^{1/2}({\mathcal F};\eps) d\eps + \|\Sigma\| \frac{p}{n}\int_{\delta}^{\rho_{\tau}} H_{\tau}({\mathcal F};\eps) d\eps
+ \delta \|\Sigma\| \Bigl(\frac{{\bf r}(\Sigma)}{\sqrt{n}} \vee \frac{p}{\sqrt{n}}\Bigr).
\end{align*}
In particular, for $\delta=0,$
\begin{align*}
&
\Big\|\sup_{f\in {\mathcal F}}|\langle f'(\Sigma), \hat \Sigma_n-\Sigma\rangle|\Bigr\|_{L_p}
\lesssim 
\|\Sigma\|_2 \sqrt{\frac{p}{n}} \int_{0}^{\rho_{\tau}} H_{\tau}^{1/2}({\mathcal F};\eps) d\eps + \|\Sigma\| \frac{p}{n}\int_{0}^{\rho_{\tau}} H_{\tau}({\mathcal F};\eps) d\eps.
\end{align*}
\end{proposition}

A consequence is the following statement.

\begin{proposition}
\label{conc_F_smooth_lin}
Let ${\mathcal F}\subset C^{m+1}({\mathbb R}_+)$ be a class of functions $f$ such that $f(0)=0$ and  
\begin{align*}
\max_{1\leq j\leq m+1}\|f^{(j)}\|_{L_{\infty}}\leq 1.
\end{align*}
Suppose that ${\bf r}(\Sigma)\lesssim n.$
Then, for $m \geq 2,$  
\begin{align*}
&
\Big\|\sup_{f\in {\mathcal F}}|\langle f'(\Sigma), \hat \Sigma_n-\Sigma\rangle|\Bigr\|_{L_p}
\lesssim 
(\|\Sigma\|+1)^{1/2}\|\Sigma\|_2 \sqrt{\frac{p}{n}}  + (\|\Sigma\|+1)\|\Sigma\| \frac{p}{n}
\end{align*}
and, for $m=1,$
\begin{align*}
&
\Big\|\sup_{f\in {\mathcal F}}|\langle f'(\Sigma), \hat \Sigma_n-\Sigma\rangle|\Bigr\|_{L_p}
\lesssim 
(\|\Sigma\|+1)^{1/2}\|\Sigma\|_2 \sqrt{\frac{p}{n}}  + (\|\Sigma\|+1)\|\Sigma\| \log n\frac{p}{n}.
\end{align*}
\end{proposition}

We also have the following result on approximation of stochastic process 
\begin{align*}
G_n(f)= \sqrt{\frac{n}{2}}\langle f'(\Sigma), \hat \Sigma_n-\Sigma\rangle, f\in {\mathcal F}
\end{align*}
by the Gaussian process $G_{\Sigma}(f), f\in {\mathcal F}$ in Wasserstein distance ${\mathcal W}_{{\mathcal F},p}.$

\begin{proposition}
\label{entr_Wasser}
Let ${\mathcal F}\subset C^1({\mathbb R}_+)$ be a class of functions $f$ satisfying the condition $f(0)=0.$
Suppose ${\bf r}(\Sigma)\lesssim n$ and let $\delta\in [0,\rho_{\tau}/2].$ For all $p\geq 1,$
\begin{align*}
&
{\mathcal W}_{{\mathcal F},p} (G_n, G_{\Sigma})
\lesssim 
\|\Sigma\|_2 \sqrt{\frac{p}{n}} \int_{\delta}^{\rho_{\tau}} H_{\tau}^{1/2}({\mathcal F};\eps) d\eps + \|\Sigma\| \frac{p}{\sqrt{n}}\int_{\delta}^{\rho_{\tau}} H_{\tau}({\mathcal F};\eps) d\eps
+ \delta \|\Sigma\| {\bf r}(\Sigma)\frac{p}{\sqrt{n}}.
\end{align*}
In particular, for $\delta=0,$
\begin{align*}
&
{\mathcal W}_{{\mathcal F},p} (G_n, G_{\Sigma})
\lesssim 
\|\Sigma\|_2 \sqrt{\frac{p}{n}} \int_{0}^{\rho_{\tau}} H_{\tau}^{1/2}({\mathcal F};\eps) d\eps + \|\Sigma\| \frac{p}{\sqrt{n}}\int_{0}^{\rho_{\tau}} H_{\tau}({\mathcal F};\eps) d\eps.
\end{align*}
\end{proposition}

In particular, the following statement is a consequence of Proposition \ref{entr_Wasser}.

\begin{proposition}
\label{conc_F_smooth_Wasser}
Let ${\mathcal F}\subset C^{m+1}({\mathbb R}_+)$ be a class of functions $f$ such that $f(0)=0$ and  
\begin{align*}
\max_{1\leq j\leq m+1}\|f^{(j)}\|_{L_{\infty}}\leq 1.
\end{align*}
Suppose also that ${\bf r}(\Sigma)\lesssim n.$
Then, for $m\geq 2,$
\begin{align*}
&
{\mathcal W}_{{\mathcal F},p} (G_n, G_{\Sigma})
\lesssim 
(\|\Sigma\|+1)^{1/2}\|\Sigma\|_2 \sqrt{\frac{p}{n}}  + (\|\Sigma\|+1)\|\Sigma\| \frac{p}{\sqrt{n}}.
\end{align*}
On the other hand, for $m=1,$
\begin{align*}
&
{\mathcal W}_{{\mathcal F},p} (G_n, G_{\Sigma})
\lesssim 
(\|\Sigma\|+1)^{1/2}\|\Sigma\|_2 \sqrt{\frac{p}{n}}  + (\|\Sigma\|+1)\|\Sigma\| \log n \frac{p}{\sqrt{n}}.
\end{align*}
\end{proposition}

For $f,g:{\mathbb R}_+\mapsto {\mathbb R},$ denote 
\begin{align*}
d(f,g):= \|f''-g''\|_{L_{\infty}}=\sup_{x\in {\mathbb R}_{+}}|f''(x)-g''(x)|.
\end{align*}
Let 
\begin{align*}
\rho_d=\rho_d({\mathcal F}):=\sup_{f\in {\mathcal F}}\|f''\|_{L_{\infty}}.
\end{align*}

For the remainder process $R_f(\Sigma, \hat \Sigma_n-\Sigma)-{\mathbb E}R_f(\Sigma, \hat \Sigma_n-\Sigma), f\in {\mathcal F},$
the following result will be proved.

\begin{proposition}
\label{entr_bd}
Let ${\mathcal F}\subset C^2({\mathbb R}_+)$ be a class of functions $f:{\mathbb R}_+\mapsto {\mathbb R}$ satisfying the condition $f(0)=0.$
Suppose ${\bf r}(\Sigma)\lesssim n$ and let $\delta\in [0,\rho_d/2].$ Then, for all $p\geq 1,$
\begin{align*}
&
\Big\|\sup_{f\in {\mathcal F}}|R_f(\Sigma, \hat \Sigma_n-\Sigma)-{\mathbb E}R_f(\Sigma, \hat \Sigma_n-\Sigma)|\Bigr\|_{L_p}
\\
&
\lesssim 
\|\Sigma\|^2\Bigl(\frac{{\bf r}(\Sigma)}{\sqrt{n}}\sqrt{\frac{p}{n}} \int_{\delta}^{\rho_d} (H_{d}^{1/2}({\mathcal F}, \eps)+1)d\eps+\Bigl(\frac{{\bf r}(\Sigma)}{\sqrt{n}}\vee 1\Bigr)\frac{p}{n}
\int_{\delta}^{\rho_d} (H_{d}({\mathcal F}, \eps)+1)d\eps 
\\
&
+ \Bigl(\frac{p}{n}\Bigr)^{2} \int_{\delta}^{\rho_d} (H_{d}^2({\mathcal F}, \eps)+1)d\eps\Bigr)
+ \delta \|\Sigma\|^2 \Bigl(\frac{{\bf r}(\Sigma)^2}{n} \vee \frac{p}{n}\vee \Bigl(\frac{p}{n}\Bigr)^2\Bigr).
\end{align*}
In particular, for $\delta=0,$
\begin{align*}
&
\Big\|\sup_{f\in {\mathcal F}}|R_f(\Sigma, \hat \Sigma_n-\Sigma)-{\mathbb E}R_f(\Sigma, \hat \Sigma_n-\Sigma)|\Bigr\|_{L_p}
\\
&
\lesssim 
\|\Sigma\|^2\Bigl(\frac{{\bf r}(\Sigma)}{\sqrt{n}}\sqrt{\frac{p}{n}} \int_{0}^{\rho_d} (H_{d}^{1/2}({\mathcal F}, \eps)+1)d\eps+\Bigl(\frac{{\bf r}(\Sigma)}{\sqrt{n}}\vee 1\Bigr)\frac{p}{n}\int_{0}^{\rho_d} (H_{d}({\mathcal F}, \eps)+1)d\eps 
\\
&
+ \Bigl(\frac{p}{n}\Bigr)^{2} \int_{0}^{\rho_d} (H_{d}^2({\mathcal F}, \eps)+1)d\eps\Bigr).
\end{align*}
\end{proposition}

In particular, Proposition \ref{entr_bd} implies the following statement.

\begin{proposition}
\label{conc_F_smooth}
Let ${\mathcal F}\subset C^{m+1}({\mathbb R}_+)$ be a class of functions $f$ such that $f(0)=0$ and 
\begin{align*}
\max_{2\leq j\leq m+1}\|f^{(j)}\|_{L_{\infty}}\leq 1.
\end{align*}
Suppose that ${\bf r}(\Sigma)\lesssim n.$
Then,
for $m \geq 4,$  
\begin{align}
\label{bd_m>3_A'}
&
\nonumber
\Big\|\sup_{f\in {\mathcal F}}|R_{f}(\Sigma, \hat \Sigma_n-\Sigma)-{\mathbb E}R_{f}(\Sigma, \hat \Sigma_n-\Sigma)|\Bigr\|_{L_p}
\\
&
\lesssim_m 
\|\Sigma\|^2\Bigl((\|\Sigma\|+1)^{1/2}\frac{{\bf r}(\Sigma)}{\sqrt{n}}\sqrt{\frac{p}{n}} +(\|\Sigma\|+1)\Bigl(\frac{{\bf r}(\Sigma)}{\sqrt{n}}\vee 1\Bigr)\frac{p}{n} 
+ (\|\Sigma\|+1)^2\Bigl(\frac{{\bf r}(\Sigma)^2}{n}\vee 1\Bigr)\Bigl(\frac{p}{n}\Bigr)^{2}\Bigr);
\end{align}
for $m=3,$
\begin{align}
\label{bd_m_3_A'}
&
\nonumber
\Big\|\sup_{f\in {\mathcal F}}|R_{f}(\Sigma, \hat \Sigma_n-\Sigma)-{\mathbb E}R_{f}(\Sigma, \hat \Sigma_n-\Sigma)|\Bigr\|_{L_p}
\\
&
\lesssim 
\|\Sigma\|^2\Bigl((\|\Sigma\|+1)^{1/2}\frac{{\bf r}(\Sigma)}{\sqrt{n}}\sqrt{\frac{p}{n}} +(\|\Sigma\|+1)\Bigl(\frac{{\bf r}(\Sigma)}{\sqrt{n}}\vee 1\Bigr)\frac{p}{n} 
+ (\|\Sigma\|+1)^2 \Bigl(\frac{{\bf r}(\Sigma)^2}{n}\vee \log n\Bigr)\Bigl(\frac{p}{n}\Bigr)^{2}\Bigr);
\end{align}
and, for $m=2,$
\begin{align}
\label{bd_m_2_A'}
&
\nonumber
\Big\|\sup_{f\in {\mathcal F}}|R_{f}(\Sigma, \hat \Sigma_n-\Sigma)-{\mathbb E}R_{f}(\Sigma, \hat \Sigma_n-\Sigma)|\Bigr\|_{L_p}
\\
&
\lesssim 
\|\Sigma\|^2\Bigl((\|\Sigma\|+1)^{1/2}\frac{{\bf r}(\Sigma)}{\sqrt{n}}\sqrt{\frac{p}{n}} +(\|\Sigma\|+1)\log n \Bigl(\frac{{\bf r}(\Sigma)}{\sqrt{n}}\vee 1\Bigr)\frac{p}{n}
+ (\|\Sigma\|+1)^2\frac{p^2}{n}\Bigr).
\end{align}
\end{proposition}

Finally, we state the bounds for the stochastic process 
$
\tau_f (\hat \Sigma_n)- {\mathbb E} \tau_f(\hat \Sigma_n), f\in {\mathcal F}
$
that immediately follow from propositions \ref{conc_F_smooth_lin}, \ref{conc_F_smooth_Wasser} and \ref{conc_F_smooth}.

\begin{proposition}
Let ${\mathcal F}\subset C^{m+1}({\mathbb R}_+)$ be a class of functions $f$ such that $f(0)=0$ and 
\begin{align*}
\max_{1\leq j\leq m+1}\|f^{(j)}\|_{L_{\infty}}\leq 1.
\end{align*}
Suppose that ${\bf r}(\Sigma)\lesssim n.$
Then,
for $m \geq 4,$  
\begin{align}
\label{bd_m>3_A'_B}
&
\nonumber
\Big\|\sup_{f\in {\mathcal F}}|\tau_f (\hat \Sigma_n)- {\mathbb E} \tau_f(\hat \Sigma_n)|\Bigr\|_{L_p}
\lesssim_m (\|\Sigma\|+1)^{1/2}\|\Sigma\|\sqrt{{\bf r}(\Sigma^2)} \sqrt{\frac{p}{n}}  + (\|\Sigma\|+1)\|\Sigma\| \frac{p}{n}
\\
&
+
\|\Sigma\|^2\Bigl((\|\Sigma\|+1)^{1/2}\frac{{\bf r}(\Sigma)}{\sqrt{n}}\sqrt{\frac{p}{n}} +(\|\Sigma\|+1)\Bigl(\frac{{\bf r}(\Sigma)}{\sqrt{n}}\vee 1\Bigr)\frac{p}{n} 
+ (\|\Sigma\|+1)^2\Bigl(\frac{{\bf r}(\Sigma)^2}{n}\vee 1\Bigr)\Bigl(\frac{p}{n}\Bigr)^{2}\Bigr);
\end{align}
for $m=3,$
\begin{align}
\label{bd_m_3_A'_B}
&
\nonumber
\Big\|\sup_{f\in {\mathcal F}}|\tau_f (\hat \Sigma_n)- {\mathbb E} \tau_f(\hat \Sigma_n)|\Bigr\|_{L_p}
\lesssim (\|\Sigma\|+1)^{1/2}\|\Sigma\|\sqrt{{\bf r}(\Sigma^2)} \sqrt{\frac{p}{n}}  + (\|\Sigma\|+1)\|\Sigma\| \frac{p}{n}
\\
&
+\|\Sigma\|^2\Bigl((\|\Sigma\|+1)^{1/2}\frac{{\bf r}(\Sigma)}{\sqrt{n}}\sqrt{\frac{p}{n}} +(\|\Sigma\|+1)\Bigl(\frac{{\bf r}(\Sigma)}{\sqrt{n}}\vee 1\Bigr)\frac{p}{n} 
+ (\|\Sigma\|+1)^2 \Bigl(\frac{{\bf r}(\Sigma)^2}{n}\vee \log n\Bigr)\Bigl(\frac{p}{n}\Bigr)^{2}\Bigr);
\end{align}
and, for $m=2,$
\begin{align}
\label{bd_m_2_A'_B}
&
\nonumber
\Big\|\sup_{f\in {\mathcal F}}|\tau_f (\hat \Sigma_n)- {\mathbb E} \tau_f(\hat \Sigma_n)|\Bigr\|_{L_p}
\lesssim (\|\Sigma\|+1)^{1/2}\|\Sigma\|\sqrt{{\bf r}(\Sigma^2)} \sqrt{\frac{p}{n}}  + (\|\Sigma\|+1)\|\Sigma\| \frac{p}{n}
\\
&
+
\|\Sigma\|^2\Bigl((\|\Sigma\|+1)^{1/2}\frac{{\bf r}(\Sigma)}{\sqrt{n}}\sqrt{\frac{p}{n}} +(\|\Sigma\|+1)\log n \Bigl(\frac{{\bf r}(\Sigma)}{\sqrt{n}}\vee 1\Bigr)\frac{p}{n}
+ (\|\Sigma\|+1)^2\frac{p^2}{n}\Bigr).
\end{align}
\end{proposition}

Denote 
\begin{align*}
T_n(f):=\sqrt{\frac{n}{2}}\Bigl(\tau_f (\hat \Sigma_n)- {\mathbb E} \tau_f(\hat \Sigma_n)\Bigr), f\in {\mathcal F}
\end{align*}

\begin{proposition}
Let ${\mathcal F}\subset C^{m+1}({\mathbb R}_+)$ be a class of functions $f$ such that $f(0)=0$ and 
\begin{align*}
\max_{1\leq j\leq m+1}\|f^{(j)}\|_{L_{\infty}}\leq 1.
\end{align*}
Suppose that ${\bf r}(\Sigma)\lesssim n.$
Then,
for $m \geq 4,$  
\begin{align}
\label{bd_m>3_A'_C}
&
\nonumber
{\mathcal W}_{{\mathcal F}, p} (T_n, G_{\Sigma})
\lesssim_m (\|\Sigma\|+1)^{1/2}\|\Sigma\| \sqrt{{\bf r}(\Sigma^2)}\sqrt{\frac{p}{n}}  + (\|\Sigma\|+1)\|\Sigma\| \frac{p}{\sqrt{n}}
\\
&
+
\|\Sigma\|^2\Bigl((\|\Sigma\|+1)^{1/2}{\bf r}(\Sigma)\sqrt{\frac{p}{n}} +(\|\Sigma\|+1)\Bigl({\bf r}(\Sigma)\frac{p}{n}\vee \frac{p}{\sqrt{n}}\Bigr)
+ (\|\Sigma\|+1)^2\Bigl(\frac{{\bf r}(\Sigma)^2}{n}\vee 1\Bigr)\frac{p^2}{n^{3/2}}\Bigr);
\end{align}
for $m=3,$
\begin{align}
\label{bd_m_3_A'_C}
&
\nonumber
{\mathcal W}_{{\mathcal F}, p} (T_n, G_{\Sigma})
\lesssim (\|\Sigma\|+1)^{1/2}\|\Sigma\| \sqrt{{\bf r}(\Sigma^2)}\sqrt{\frac{p}{n}}  + (\|\Sigma\|+1)\|\Sigma\| \frac{p}{\sqrt{n}}
\\
&
+\|\Sigma\|^2\Bigl((\|\Sigma\|+1)^{1/2}{\bf r}(\Sigma)\sqrt{\frac{p}{n}} +(\|\Sigma\|+1)\Bigl({\bf r}(\Sigma)\frac{p}{n} \vee \frac{p}{\sqrt{n}}\Bigr)
+ (\|\Sigma\|+1)^2 \Bigl(\frac{{\bf r}(\Sigma)^2}{n}\vee \log n\Bigr)\frac{p^2}{n^{3/2}}\Bigr);
\end{align}
and, for $m=2,$
\begin{align}
\label{bd_m_2_A'_C}
&
\nonumber
{\mathcal W}_{{\mathcal F}, p} (T_n, G_{\Sigma})
\lesssim (\|\Sigma\|+1)^{1/2}\|\Sigma\| \sqrt{{\bf r}(\Sigma^2)} \sqrt{\frac{p}{n}}  + (\|\Sigma\|+1)\|\Sigma\| \frac{p}{\sqrt{n}}
\\
&
+
\|\Sigma\|^2\Bigl((\|\Sigma\|+1)^{1/2}{\bf r}(\Sigma)\sqrt{\frac{p}{n}} +(\|\Sigma\|+1)\log n \Bigl({\bf r}(\Sigma)\frac{p}{n}\vee \frac{p}{\sqrt{n}}\Bigr)
+ (\|\Sigma\|+1)^2\frac{p^2}{\sqrt{n}}\Bigr).
\end{align}
\end{proposition}

We will provide below detailed proofs of propositions \ref{entr_bd} and \ref{conc_F_smooth}. The proofs of propositions \ref{entr_bd_lin_part}, \ref{conc_F_smooth_lin}, \ref{entr_Wasser} and \ref{conc_F_smooth_Wasser}
are quite similar and we will only provide a few comments on their proofs.

We start with the proof of Proposition \ref{entr_bd}.

\begin{proof}
In what follows, we write $\rho:=\rho_d.$
We will use the notations of the proof of Theorem \ref{prop_rem}. 
Recall that, by Proposition \ref{rem_first_order}, for all $f$ satisfying $\|f''\|_{L_{\infty}}\leq \delta,$
\begin{align*} 
|R_f(\Sigma^{(N)}, \hat \Sigma_n^{(N)}-\Sigma^{(N)})| \leq \frac{\delta}{2} \|\hat \Sigma_n^{(N)}-\Sigma^{(N)}\|_2^2. 
\end{align*} 
In view of \eqref{rem_N_inf}, we can pass to the limit as $N\to\infty$ to get that, for all such $f,$
$|R_f(\Sigma, \hat \Sigma_n-\Sigma)| \leq \frac{\delta}{2} \|\hat \Sigma_n-\Sigma\|_2^2.$ In other words, 
\begin{align*} 
\sup_{\|f''\|_{L_{\infty}}\leq \delta}|R_f(\Sigma, \hat \Sigma_n-\Sigma)| \leq \frac{\delta}{2} \|\hat \Sigma_n-\Sigma\|_2^2. 
\end{align*} 
By bound  \eqref{co_HS}, for all $p\geq 1,$
\begin{align} 
\label{small_ball_R_f_A}
\nonumber
\Bigl\|\sup_{\|f''\|_{L_{\infty}}\leq \delta}|R_f(\Sigma, \hat \Sigma_n-\Sigma)|\Bigr\|_{L_p} &\leq \frac{\delta}{2} \Bigl\|\|\hat \Sigma_n-\Sigma\|_2^2\Bigr\|_{L_p} 
=\frac{\delta}{2} \Bigl\|\|\hat \Sigma_n-\Sigma\|_2\Bigr\|_{L_{2p}}^2 
\\
&
\lesssim \delta\|\Sigma\|^2 \Bigl(\frac{{\bf r}(\Sigma)^2}{n} \vee \frac{p}{n}\vee \Bigl(\frac{p}{n}\Bigr)^2\Bigr),
\end{align} 
which easily implies that 
\begin{align}
\label{small_ball_R_f} 
\Bigl\|\sup_{\|f''\|_{L_{\infty}}\leq \delta}|R_f(\Sigma, \hat \Sigma_n-\Sigma)-{\mathbb E}R_f(\Sigma, \hat \Sigma_n-\Sigma)|\Bigr\|_{L_p} 
\lesssim \delta\|\Sigma\|^2 \Bigl(\frac{{\bf r}(\Sigma)^2}{n} \vee \frac{p}{n}\vee \Bigl(\frac{p}{n}\Bigr)^2\Bigr).
\end{align}

The rest of the proof is based on a standard chaining argument for stochastic process
\begin{align*}
Y_n(f):= R_f(\Sigma, \hat \Sigma_n-\Sigma)-{\mathbb E}R_f(\Sigma, \hat \Sigma_n-\Sigma), f\in {\mathcal F}.
\end{align*}
Without loss of generality, we can assume that $0\in {\mathcal F}$ (otherwise, we can just replace class ${\mathcal F}$ by ${\mathcal F}\cup \{0\}$).
Let $\eps_k:= \rho 2^{-k}, k\geq 0,$
let ${\mathcal F}_0=\{0\}$ and, for $k\geq 1,$ let ${\mathcal F}_k\subset {\mathcal F}$ be an $\eps_k$-net of cardinality $N_d({\mathcal F},\eps_k).$
For $f\in {\mathcal F}$ and $k\geq 0,$ let 
\begin{align*}
\pi_k f\in {\rm Argmin}_{g\in {\mathcal F}_k} d(f,g).
\end{align*}
Then $\pi_0 f=0$ and $d(f,\pi_k f)\leq \eps_k, k\geq 0.$

Note that 
\begin{align*}
Y_n(f)-Y_n(\pi_k f)= R_{f-\pi_k f}(\Sigma, \hat \Sigma_n-\Sigma)-{\mathbb E}R_{f-\pi_k f}(\Sigma, \hat \Sigma_n-\Sigma)
\end{align*}
It follows from bound \eqref{small_ball_R_f} that, for all $k\geq 1$ and $p\geq 1,$ 
\begin{align}
\label{f_pi_k_f}
\Bigl\|\sup_{f\in {\mathcal F}}|Y_n(f)-Y_n(\pi_k f)|\Bigr\|_{L_p}\lesssim \eps_k \|\Sigma\|^2 \Bigl(\frac{{\bf r}(\Sigma)^2}{n} \vee \frac{p}{n}\vee \Bigl(\frac{p}{n}\Bigr)^2\Bigr).
\end{align}
Let $K$ be such that $\delta\in [\eps_{K+3}, \eps_{K+2}).$
Then, the following representation holds:
\begin{align*}
Y_n(\pi_{K+1}f)=\sum_{k=0}^{K}(Y_n(\pi_{k+1}f)-Y_n(\pi_k f))
\end{align*}
and it implies the bound 
\begin{align*}
\sup_{f\in {\mathcal F}}|Y_n(\pi_{K+1}f)| \leq \sum_{k=0}^{K}\sup_{f\in {\mathcal F}}|Y_n(\pi_{k+1}f)-Y_n(\pi_k f)|.
\end{align*}
Denote ${\mathcal G}_k:= \Bigl\{f_1-f_2: f_1\in {\mathcal F}_k, f_2\in {\mathcal F}_{k+1}, d(f_1,f_2)\leq 3\eps_{k+1}\Bigr\}.$
Then 
\begin{align*}
\sup_{f\in {\mathcal F}}|Y_n(\pi_{k+1}f)-Y_n(\pi_k f)|\leq \max_{g\in {\mathcal G}_k}|Y_n(g)|,
\end{align*}
and we get 
\begin{align*}
\Big\|\sup_{f\in {\mathcal F}}|Y_n(\pi_{K+1}f)|\Bigr\|_{L_p} \leq \sum_{k=0}^{K} \Bigl\|\max_{g\in {\mathcal G}_k}|Y_n(g)|\Bigr\|_{L_p}.
\end{align*}
Note that 
\begin{align*}
{\rm card}({\mathcal G}_k)\leq N_d({\mathcal F},\eps_k) N_d({\mathcal F}, \eps_{k+1})\leq e N_d({\mathcal F}, \eps_{k+1})^2=: N_k\geq e.
\end{align*}
Thus, since $\log N_k\geq 1,$
\begin{align*}
\Bigl\|\max_{g\in {\mathcal G}_k}|Y_n(g)|\Bigr\|_{L_p} &
\leq \Bigl\|\max_{g\in {\mathcal G}_k}|Y_n(g)|\Bigr\|_{L_{p\log N_k}}= \Bigl({\mathbb E}\max_{g\in {\mathcal G}_k}|Y_n(g)|^{p\log N_k}\Bigr)^{1/p\log N_k}
\\
&
\leq \Bigl({\mathbb E} \sum_{g\in {\mathcal G}_k} |Y_n(g)|^{p\log N_k}\Bigr)^{1/p\log N_k}
=\Bigl(\sum_{g\in {\mathcal G}_k} {\mathbb E}|Y_n(g)|^{p\log N_k}\Bigr)^{1/p\log N_k}
\\
&
\leq \Bigl({\rm card}({\mathcal G}_k)
\max_{g\in {\mathcal G}_k} {\mathbb E}|Y_n(g)|^{p\log N_k}\Bigr)^{1/p\log N_k}
\leq N_k^{1/p\log N_k} \max_{g\in {\mathcal G}_k} \Bigl({\mathbb E}|Y_n(g)|^{p\log N_k}\Bigr)^{1/p\log N_k}
\\
&
= e^{1/p} \max_{g\in {\mathcal G}_k}\|Y_n(g)\|_{L_{p\log N_k}}.
\end{align*}
Using bound of Theorem \ref{prop_rem} and taking into account that, for all $g\in {\mathcal G}_k,$ 
$
\|g''\|_{L_{\infty}}\leq 3\eps_{k+1},
$
we get 
\begin{align*}
&
\max_{g\in {\mathcal G}_k}\|Y_n(g)\|_{L_{p\log N_k}}
\\
&
\lesssim \eps_{k+1} \|\Sigma\|^2\Bigl(\frac{{\bf r}(\Sigma)}{\sqrt{n}}\sqrt{\frac{p\log N_k}{n}}\vee \Bigl(\frac{{\bf r}(\Sigma)}{\sqrt{n}}\vee 1\Bigr)\frac{p\log N_k}{n}\vee \Bigl(\frac{p\log N_k}{n}\Bigr)^{2}\Bigr)
\end{align*}
Therefore,
\begin{align*}
\Big\|\sup_{f\in {\mathcal F}}|Y_n(\pi_{K+1}f)|\Bigr\|_{L_p}
&\lesssim
\|\Sigma\|^2\Bigl(\frac{{\bf r}(\Sigma)}{\sqrt{n}}\sqrt{\frac{p}{n}}\sum_{k=0}^K \eps_{k+1}\sqrt{\log N_k}+ \Bigl(\frac{{\bf r}(\Sigma)}{\sqrt{n}}\vee 1\Bigr)\frac{p}{n}\sum_{k=0}^K \eps_{k+1}\log N_k
\\
&
+ \Bigl(\frac{p}{n}\Bigr)^{2} \sum_{k=0}^{K}\eps_{k+1}\log^2 N_k\Bigr).
\end{align*}
Observe that 
\begin{align*}
&
\sum_{k=0}^K \eps_{k+1}\sqrt{\log N_k} \lesssim \int_{\delta}^\rho (H_{d}^{1/2}({\mathcal F}, \eps)+1)d\eps, 
\sum_{k=0}^K \eps_{k+1}\log N_k \lesssim \int_{\delta}^\rho (H_{d}({\mathcal F}, \eps)+1)d\eps, 
\\
&
\sum_{k=0}^K \eps_{k+1}\log^2 N_k \lesssim \int_{\delta}^\rho (H_{d}^2({\mathcal F}, \eps)+1)d\eps.
\end{align*}
Thus, we get 
\begin{align*}
&
\Big\|\sup_{f\in {\mathcal F}}|Y_n(\pi_{K+1}f)|\Bigr\|_{L_p}
\\
&
\lesssim
\|\Sigma\|^2\Bigl(\frac{{\bf r}(\Sigma)}{\sqrt{n}}\sqrt{\frac{p}{n}} \int_{\delta}^\rho (H_{d}^{1/2}({\mathcal F}, \eps)+1)d\eps+\Bigl(\frac{{\bf r}(\Sigma)}{\sqrt{n}}\vee 1\Bigr)\frac{p}{n}\int_{\delta}^\rho (H_{d}({\mathcal F}, \eps)+1)d\eps 
\\
&
+ \Bigl(\frac{p}{n}\Bigr)^{2} \int_{\delta}^\rho (H_{d}^2({\mathcal F}, \eps)+1)d\eps\Bigr).
\end{align*}
Combining the last bound with \eqref{f_pi_k_f} (for $k=K+1$), we get 
\begin{align*}
&
\Big\|\sup_{f\in {\mathcal F}}|Y_n(f)|\Bigr\|_{L_p}\leq \Big\|\sup_{f\in {\mathcal F}}|Y_n(\pi_{K+1}f)|\Bigr\|_{L_p}+ 
\Big\|\sup_{f\in {\mathcal F}}|Y_n(f)-Y_n(\pi_{K+1}f)|\Bigr\|_{L_p}
\\
&
\lesssim 
\|\Sigma\|^2\Bigl(\frac{{\bf r}(\Sigma)}{\sqrt{n}}\sqrt{\frac{p}{n}} \int_{\delta}^\rho (H_{d}^{1/2}({\mathcal F}, \eps)+1)d\eps+\Bigl(\frac{{\bf r}(\Sigma)}{\sqrt{n}}\vee 1\Bigr)\frac{p}{n}\int_{\delta}^\rho (H_{d}({\mathcal F}, \eps)+1)d\eps 
\\
&
+ \Bigl(\frac{p}{n}\Bigr)^{2} \int_{\delta}^\rho (H_{d}^2({\mathcal F}, \eps)+1)d\eps\Bigr)
+ \delta \|\Sigma\|^2 \Bigl(\frac{{\bf r}(\Sigma)^2}{n} \vee \frac{p}{n}\vee \Bigl(\frac{p}{n}\Bigr)^2\Bigr),
\end{align*}
implying the claim.

\qed
\end{proof}

Next we turn to the proof of Proposition \ref{conc_F_smooth}.

\begin{proof}
Let ${\mathcal F}\subset C^{m+1}({\mathbb R}_+)$ be a class of functions $f$ satisfying the assumptions of the proposition and, in addition, 
such that ${\rm supp}(f)\subset [0,A]$ for some $A\geq 1.$ By a well known theorem (due to Kolmogorov), the $\eps$-entropy 
of the H\"older ball $\{h: \max_{0\leq j\leq m-1}\|h^{(j)}\|_{L_{\infty}[0,A]}\leq 1\}$ with respect to the sup-norm distance on $[0,A]$ is upper bounded 
by $\lesssim A\eps^{-1/(m-1)}.$ Define 
\begin{align*}
{\mathcal G}:= \{f'': f\in {\mathcal F}\}\subset \{f'': \max_{0\leq j\leq m-1}\|f^{(2+j)}\|_{L_{\infty}[0,A]}\leq 1\}
\end{align*}
and $\tilde d(g_1,g_2):= \|g_1-g_2\|_{L_{\infty}[0,A]}, g_1,g_2\in {\mathcal G}.$
Then
\begin{align*}
H_d({\mathcal F},\eps)=H_{\tilde d}({\mathcal G}, \eps) \lesssim A\eps^{-1/(m-1)}.
\end{align*}
For $m\geq 4,$ we can use the bound of Proposition \ref{entr_bd} with $\rho_d=1$ and $\delta=0.$ We have 
\begin{align*}
\int_{0}^1 (H_{d}^{1/2}({\mathcal F}, \eps)+1)d\eps \lesssim \sqrt{A},\int_{0}^1 (H_{d}({\mathcal F}, \eps)+1)d\eps \lesssim A,\ 
\int_{0}^1 (H_{d}^{2}({\mathcal F}, \eps)+1)d\eps \lesssim A^2,
\end{align*}
and the bound yields:
\begin{align}
\label{bd_m>3}
\nonumber
&
\Big\|\sup_{f\in {\mathcal F}}|R_f(\Sigma, \hat \Sigma_n-\Sigma)-{\mathbb E}R_f(\Sigma, \hat \Sigma_n-\Sigma)|\Bigr\|_{L_p}
\\
&
\lesssim 
\|\Sigma\|^2\Bigl(\sqrt{A}\frac{{\bf r}(\Sigma)}{\sqrt{n}}\sqrt{\frac{p}{n}} +A\Bigl(\frac{{\bf r}(\Sigma)}{\sqrt{n}}\vee 1\Bigr)\frac{p}{n} 
+ A^2\Bigl(\frac{p}{n}\Bigr)^{2}\Bigr). 
\end{align}
For $m=2$ or $m=3,$ some of the entropy integrals diverge and bound of Proposition \ref{entr_bd} with $\rho_d=1$ and $\delta\in (0,1/2)$ should be used.
In particular, for $m=3,$ we have 
\begin{align*}
\int_{\delta}^1 (H_{d}^{1/2}({\mathcal F}, \eps)+1)d\eps \lesssim \sqrt{A},\ \int_{\delta}^1 (H_{d}({\mathcal F}, \eps)+1)d\eps \lesssim A,\ 
\int_{\delta}^1 (H_{d}^{2}({\mathcal F}, \eps)+1)d\eps \lesssim A^2 \log\frac{1}{\delta}.
\end{align*}
Therefore,
\begin{align*}
&
\Big\|\sup_{f\in {\mathcal F}}|R_f(\Sigma, \hat \Sigma_n-\Sigma)-{\mathbb E}R_f(\Sigma, \hat \Sigma_n-\Sigma)|\Bigr\|_{L_p}
\\
&
\lesssim 
\|\Sigma\|^2\Bigl(\sqrt{A}\frac{{\bf r}(\Sigma)}{\sqrt{n}}\sqrt{\frac{p}{n}} +A\Bigl(\frac{{\bf r}(\Sigma)}{\sqrt{n}}\vee 1\Bigr)\frac{p}{n}
+ A^2 \log\frac{1}{\delta}\Bigl(\frac{p}{n}\Bigr)^{2}\Bigr) 
+ \delta \|\Sigma\|^2 \Bigl(\frac{{\bf r}(\Sigma)^2}{n} \vee \frac{p}{n}\vee \Bigl(\frac{p}{n}\Bigr)^2\Bigr).
\end{align*}
For $\delta=n^{-2},$ we get under the assumption that ${\bf r}(\Sigma)\lesssim n,$
\begin{align}
\label{bd_m_3}
\nonumber
&
\Big\|\sup_{f\in {\mathcal F}}|R_f(\Sigma, \hat \Sigma_n-\Sigma)-{\mathbb E}R_f(\Sigma, \hat \Sigma_n-\Sigma)|\Bigr\|_{L_p}
\\
&
\lesssim 
\|\Sigma\|^2\Bigl(\sqrt{A}\frac{{\bf r}(\Sigma)}{\sqrt{n}}\sqrt{\frac{p}{n}} +A\Bigl(\frac{{\bf r}(\Sigma)}{\sqrt{n}}\vee 1\Bigr)\frac{p}{n}
+ A^2 \log n\Bigl(\frac{p}{n}\Bigr)^{2}\Bigr).
\end{align}
Finally, for $m=2,$ we have 
\begin{align*}
\int_{\delta}^1(H_{d}^{1/2}({\mathcal F}, \eps)+1)d\eps \lesssim \sqrt{A},\ \int_{\delta}^1 (H_{d}({\mathcal F}, \eps)+1)d\eps \lesssim A \log\frac{1}{\delta},\ 
\int_{\delta}^1 (H_{d}^{2}({\mathcal F}, \eps)+1)d\eps \lesssim \frac{A^2}{\delta}. 
\end{align*}
Thus, bound of Proposition \ref{entr_bd} yields
\begin{align*}
&
\Big\|\sup_{f\in {\mathcal F}}|R_f(\Sigma, \hat \Sigma_n-\Sigma)-{\mathbb E}R_f(\Sigma, \hat \Sigma_n-\Sigma)|\Bigr\|_{L_p}
\\
&
\lesssim 
\|\Sigma\|^2\Bigl(\sqrt{A}\frac{{\bf r}(\Sigma)}{\sqrt{n}}\sqrt{\frac{p}{n}} +A\log\frac{1}{\delta}\Bigl(\frac{{\bf r}(\Sigma)}{\sqrt{n}}\vee 1\Bigr)\frac{p}{n}
+ \frac{A^2}{\delta} \Bigl(\frac{p}{n}\Bigr)^{2}\Bigr) 
+ \delta \|\Sigma\|^2 \Bigl(\frac{{\bf r}(\Sigma)^2}{n} \vee \frac{p}{n}\vee \Bigl(\frac{p}{n}\Bigr)^2\Bigr).
\end{align*}
Setting $\delta:= n^{-1},$ we get (assuming that ${\bf r}(\Sigma)\lesssim n$),
\begin{align}
\label{bd_m_2}
\nonumber
&
\Big\|\sup_{f\in {\mathcal F}}|R_f(\Sigma, \hat \Sigma_n-\Sigma)-{\mathbb E}R_f(\Sigma, \hat \Sigma_n-\Sigma)|\Bigr\|_{L_p}
\\
&
\lesssim 
\|\Sigma\|^2\Bigl(\sqrt{A}\frac{{\bf r}(\Sigma)}{\sqrt{n}}\sqrt{\frac{p}{n}} +A\log n \Bigl(\frac{{\bf r}(\Sigma)}{\sqrt{n}}\vee 1\Bigr)\frac{p}{n}
+ A^2\frac{p^2}{n}\Bigr).
\end{align}

Let now ${\mathcal F}\subset C^{m+1}({\mathbb R}_+)$ be an arbitrary class of functions $f$ satisfying the assumptions of the proposition 
(with no assumptions on ${\rm supp}(f)$). Let $\psi:{\mathbb R}\mapsto [0,1]$ be a $C^{\infty}$ function such that $\psi(x)=1$ for $x\leq 1/2$ 
and $\psi(x)=0$ for $x\geq 1.$ 
Let $A\geq 2c\|\Sigma\|+2,$ where $c>0$ is a sufficiently large numerical constant. Denote $\psi_A(x):= \psi(x/A), x\in {\mathbb R}.$ It is easy to check that, for all $f\in {\mathcal F},$ 
$(f\psi_A)(0)=0$ and 
\begin{align*}
 \max_{0\leq j\leq m-1}\|(f\psi_A)^{(2+j)}\|_{L_{\infty}[0,A]}\lesssim_m 1
\end{align*}
and ${\rm supp}(f\psi_A)\subset [0, A].$ Thus, bounds \eqref{bd_m>3}, \eqref{bd_m_3} and \eqref{bd_m_2} hold for the class $\{f\psi_A: f\in {\mathcal F}\}:$ 
for $m\geq 4,$
\begin{align}
\label{bd_m>3_A}
\nonumber
&
\Big\|\sup_{f\in {\mathcal F}}|R_{f\psi_A}(\Sigma, \hat \Sigma_n-\Sigma)-{\mathbb E}R_{f\psi_A}(\Sigma, \hat \Sigma_n-\Sigma)|\Bigr\|_{L_p}
\\
&
\lesssim_m 
\|\Sigma\|^2\Bigl(\sqrt{A}\frac{{\bf r}(\Sigma)}{\sqrt{n}}\sqrt{\frac{p}{n}} +A\Bigl(\frac{{\bf r}(\Sigma)}{\sqrt{n}}\vee 1\Bigr)\frac{p}{n} 
+ A^2\Bigl(\frac{p}{n}\Bigr)^{2}\Bigr);
\end{align}
for $m=3,$
\begin{align}
\label{bd_m_3_A}
\nonumber
&
\Big\|\sup_{f\in {\mathcal F}}|R_{f\psi_A}(\Sigma, \hat \Sigma_n-\Sigma)-{\mathbb E}R_{f\psi_A}(\Sigma, \hat \Sigma_n-\Sigma)|\Bigr\|_{L_p}
\\
&
\lesssim 
\|\Sigma\|^2\Bigl(\sqrt{A}\frac{{\bf r}(\Sigma)}{\sqrt{n}}\sqrt{\frac{p}{n}} +A\Bigl(\frac{{\bf r}(\Sigma)}{\sqrt{n}}\vee 1\Bigr)\frac{p}{n}
+ A^2 \log n\Bigl(\frac{p}{n}\Bigr)^{2}\Bigr);
\end{align}
and, for $m=2,$
\begin{align}
\label{bd_m_2_A}
\nonumber
&
\Big\|\sup_{f\in {\mathcal F}}|R_{f\psi_A}(\Sigma, \hat \Sigma_n-\Sigma)-{\mathbb E}R_{f\psi_A}(\Sigma, \hat \Sigma_n-\Sigma)|\Bigr\|_{L_p}
\\
&
\lesssim 
\|\Sigma\|^2\Bigl(\sqrt{A}\frac{{\bf r}(\Sigma)}{\sqrt{n}}\sqrt{\frac{p}{n}} +A\log n \Bigl(\frac{{\bf r}(\Sigma)}{\sqrt{n}}\vee 1\Bigr)\frac{p}{n}
+ A^2\frac{p^2}{n}\Bigr).
\end{align}

On the other hand, note that $\sigma(\Sigma)\subset (0,A/2)$ and, on the event $E:=\{\|\hat \Sigma_n\|<A/2\},$  we also have $\sigma(\hat \Sigma_n)\subset (0,A/2).$   
Since $1-\psi_A(x)=0$ for all $x<1/2,$ we can conclude that $\tau_{f(1-\psi_A)}(\Sigma)=0$ and, on the event $E,$ $\tau_{f(1-\psi_A)}(\hat \Sigma_n)=0$
and $(f(1-\psi_A))'(\Sigma)=0.$ Therefore, $R_{f(1-\psi_A)}(\Sigma, \hat \Sigma_n-\Sigma)=0$ on the event $E,$ which implies that 
\begin{align*} 
&
\Bigl\|\sup_{f\in {\mathcal F}}|R_{f(1-\psi_A)}(\Sigma, \hat \Sigma_n-\Sigma)|\Bigr\|_{L_p}= \Bigl\|\sup_{f\in {\mathcal F}}|R_{f(1-\psi_A)}(\Sigma, \hat \Sigma_n-\Sigma)| I_{E^c}\Bigr\|_{L_p}
\\
&
\leq \Bigl\|\sup_{f\in {\mathcal F}}|R_{f(1-\psi_A)}(\Sigma, \hat \Sigma_n-\Sigma)|\Bigr\|_{L_{2p}} {\mathbb P}^{1/2p}(E^c).
\end{align*}
Using bound \eqref{small_ball_R_f}, we get 
\begin{align*}
\Bigl\|\sup_{f\in {\mathcal F}}|R_{f(1-\psi_A)}(\Sigma, \hat \Sigma_n-\Sigma)|\Bigr\|_{L_{2p}} 
\lesssim \sup_{f\in {\mathcal F}}\|(f(1-\psi_A))''\|_{L_{\infty}} \|\Sigma\|^2 \Bigl(\frac{{\bf r}(\Sigma)^2}{n} \vee \frac{p}{n}\vee \Bigl(\frac{p}{n}\Bigr)^2\Bigr).
\end{align*} 
Since $\sup_{f\in {\mathcal F}}\|(f(1-\psi_A))''\|_{L_{\infty}}\lesssim 1,$ we can conclude that 
\begin{align*}
\Bigl\|\sup_{f\in {\mathcal F}}|R_{f(1-\psi_A)}(\Sigma, \hat \Sigma_n-\Sigma)|\Bigr\|_{L_{p}} 
\lesssim  \|\Sigma\|^2 \Bigl(\frac{{\bf r}(\Sigma)^2}{n} \vee \frac{p}{n}\vee \Bigl(\frac{p}{n}\Bigr)^2\Bigr) {\mathbb P}^{1/2p}(E^c).
\end{align*} 
Using the above bound for $p=1$ also yields 
\begin{align*}
\sup_{f\in {\mathcal F}}\Bigl|{\mathbb E}R_{f(1-\psi_A)}(\Sigma, \hat \Sigma_n-\Sigma)\Bigr| \leq \Bigl\|\sup_{f\in {\mathcal F}}|R_{f(1-\psi_A)}(\Sigma, \hat \Sigma_n-\Sigma)|\Bigr\|_{L_1}
\lesssim 
\|\Sigma\|^2 \frac{{\bf r}(\Sigma)^2}{n} {\mathbb P}^{1/2}(E^c).
\end{align*}
Therefore, 
\begin{align*}
\Bigl\|\sup_{f\in {\mathcal F}}|R_{f(1-\psi_A)}(\Sigma, \hat \Sigma_n-\Sigma)- {\mathbb E}R_{f(1-\psi_A)}(\Sigma, \hat \Sigma_n-\Sigma)|\Bigr\|_{L_{p}} 
\lesssim  \|\Sigma\|^2 \Bigl(\frac{{\bf r}(\Sigma)^2}{n} \vee \frac{p}{n}\vee \Bigl(\frac{p}{n}\Bigr)^2\Bigr) {\mathbb P}^{1/2p}(E^c).
\end{align*} 

Assuming that $A\geq 2c\|\Sigma\|+2$ with a sufficiently large constant $c$ and using bounds \eqref{KL_1} and \eqref{KL_2}, it is easy to check that, under the assumption ${\bf r}(\Sigma)\lesssim n,$ 
for some constant $c'>0$ and for $t=n,$
\begin{align*}
{\mathbb P}(E^c)&={\mathbb P}\{\|\hat \Sigma_n\|\geq A/2\}\leq  {\mathbb P}\{\|\hat \Sigma_n-\Sigma\|\geq (c-1)\|\Sigma\|\}
\\
&
\leq 
{\mathbb P}\Bigl\{|\|\hat \Sigma_n-\Sigma\|-{\mathbb E}\|\hat \Sigma_n-\Sigma\||\geq c'\|\Sigma\|\Bigl(\sqrt{\frac{t}{n}}\vee \frac{t}{n}\Bigr)\Bigr\}
\leq e^{-t}=e^{-n}.
\end{align*}
Therefore, under the assumption that ${\bf r}(\Sigma)\lesssim n,$
\begin{align*}
\Bigl\|\sup_{f\in {\mathcal F}}|R_{f(1-\psi_A)}(\Sigma, \hat \Sigma_n-\Sigma)- {\mathbb E}R_{f(1-\psi_A)}(\Sigma, \hat \Sigma_n-\Sigma)|\Bigr\|_{L_{p}} 
&\lesssim  \|\Sigma\|^2 \Bigl(\frac{{\bf r}(\Sigma)^2}{n} \vee \frac{p}{n}\vee \Bigl(\frac{p}{n}\Bigr)^2\Bigr) e^{-n/2p}
\\
&
\lesssim 
\|\Sigma\|^2 \frac{{\bf r}(\Sigma)^2}{n} e^{-n/2p}+ \|\Sigma\|^2 \Bigl(\frac{p}{n}\vee \Bigl(\frac{p}{n}\Bigr)^2\Bigr). 
\end{align*} 
Combining the last inequality with bounds \eqref{bd_m>3_A}, \eqref{bd_m_3_A} and \eqref{bd_m_2_A}, using the fact that $e^{-n/2p}\lesssim (\frac{p}{n})^2\wedge 1$ and taking $A= 2c\|\Sigma\|+2, $ 
it is easy to complete the proof.
\qed
\end{proof}

The proofs of propositions \ref{entr_bd_lin_part} and \ref{conc_F_smooth_lin} are very similar to the proofs of propositions \ref{entr_bd} and \ref{conc_F_smooth} provided above.
Concerning the proofs of propositions \ref{entr_Wasser} and \ref{conc_F_smooth_Wasser}, it is enough to construct on the same probability space versions $\bar G_n(f), f\in {\mathcal F}$
and $\bar G_{\Sigma}(f), f\in {\mathcal F}$ of stochastic processes $G_n(f), f\in {\mathcal F}$ and $G_{\Sigma}(f), f\in {\mathcal F}$ such that the $L_p$-norms of $\sup_{f\in {\mathcal F}}|\bar G_n(f)-\bar G_{\Sigma}(f)|$
could be upper bounded by the expressions in the right hand sides of inequalities of propositions \ref{entr_Wasser} and \ref{conc_F_smooth_Wasser}.
For this, we use the construction of Proposition \ref{norm_approx_lin} and the proof of the upper bounds are again quite similar to the proofs provided above (with some simplifications).

From some of the results stated in this section, it is easy to deduce the following simple proposition (that itself implies Proposition \ref{emp_pr_first} of Section \ref{Main_results}).
Recall that $\tilde G_n(f)= \sqrt{n/2}\int_{{\mathbb R}_+}f d(\mu_{\hat \Sigma_n}-\mu_{\Sigma}), f\in {\mathcal F}_1.$

\begin{proposition}
\label{emp_pr_first_detail}
Suppose that ${\bf r}(\Sigma)\lesssim n.$ 
Then, for all $p\geq 1,$ 
\begin{align*}
&
\Big\|\|\hat \mu_{\hat \Sigma_n} - \mu_{\Sigma}\|_{{\mathcal F}_1}\Bigr\|_{L_p}
\lesssim 
(\|\Sigma\|+1)^{1/2}\|\Sigma\|_2 \sqrt{\frac{p}{n}}  + (\|\Sigma\|+1)\|\Sigma\| \log n\frac{p}{n}
+\|\Sigma\|^2 \Bigl(\frac{{\bf r}(\Sigma)^2}{n} \vee 
\frac{p}{n} \vee \Bigl(\frac{p}{n}\Bigr)^2\Bigr)
\end{align*}
and 
\begin{align*}
&
\nonumber
{\mathcal W}_{{\mathcal F}_1, p} (\tilde G_n, G_{\Sigma})
\lesssim
(\|\Sigma\|+1)^{1/2}\|\Sigma\|_2 \sqrt{\frac{p}{n}}  + (\|\Sigma\|+1)\|\Sigma\| \log n \frac{p}{\sqrt{n}}
+\|\Sigma\|^2 \Bigl(\frac{{\bf r}(\Sigma)^2}{\sqrt{n}} \vee 
\frac{p}{\sqrt{n}} \vee \frac{p^2}{n^{3/2}}\Bigr).
\end{align*}
\end{proposition}

\begin{proof}
Note that 
\begin{align*}
\int_{{\mathbb R}_+} fd(\mu_{\hat \Sigma_n}-\mu_{\Sigma}) = \tau_f(\hat \Sigma_n) -\tau_f(\Sigma) = 
\langle f'(\Sigma), \hat \Sigma_n-\Sigma\rangle + R_f(\Sigma, \hat \Sigma_n-\Sigma).
\end{align*}
It follows from bound \eqref{bd_on_R_f_gen} that 
\begin{align*}
\sup_{f\in {\mathcal F}_1}|R_f(\Sigma, \hat \Sigma_n-\Sigma)| \leq \frac{1}{2} \|\hat \Sigma_n-\Sigma\|_2^2.
\end{align*}
Using the bound 
\begin{align*}
&
 \Bigl\|\|\hat \Sigma_n-\Sigma\|_2^2\Bigr\|_{L_p} \lesssim \|\Sigma\|^2 \Bigl(\frac{{\bf r}(\Sigma)^2}{n} \vee 
\frac{p}{n} \vee \Bigl(\frac{p}{n}\Bigr)^2\Bigr)
\end{align*}
that easily follows from propositions \ref{prop_lin} and \ref{norm_approx_lin}, we get that 
\begin{align*}
\Bigl\|\sup_{f\in {\mathcal F}_1}|R_f(\Sigma, \hat \Sigma_n-\Sigma)| \Bigr\|_{L_p} \lesssim \|\Sigma\|^2 \Bigl(\frac{{\bf r}(\Sigma)^2}{n} \vee 
\frac{p}{n} \vee \Bigl(\frac{p}{n}\Bigr)^2\Bigr).
\end{align*}
It remains to combine the last bound with the bounds of propositions \ref{conc_F_smooth_lin} and \ref{conc_F_smooth_Wasser} (both for $m=1$)
to complete the proof.

\qed
\end{proof}

\section{Bounding the bias of estimator $\hat T_f(X_1,\dots, X_n)$}
\label{bias_bounds}

Our goal in this section is to prove the following result concerning the bias of estimator $\hat T_f(X_1,\dots, X_n)$ of $\tau_f(\Sigma).$

\begin{theorem}
\label{bias_main}
Suppose $f\in C^{m+1}({\mathbb R}_+)$ for some $m\geq 2,$ $f(0)=0$ and $\|f^{(m+1)}\|_{L_{\infty}}<\infty.$ Then 
\begin{align*}
|{\mathbb E}\hat T_{f,m}(X_1,\dots,X_n)- \tau_f(\Sigma)|\lesssim_m
\|f^{(m+1)}\|_{L_{\infty}} \|\Sigma\|^{m+1} {\bf r}(\Sigma)\Bigl(\sqrt{\frac{{\bf r}(\Sigma)}{n}}\vee \sqrt{\frac{\log n}{n}}\Bigr)^{m+1}.
\end{align*}
\end{theorem}

For the proof, we need several deep results on higher-order Taylor expansions for functional $\tau_f,$ obtained in \cite{P_S_S}. 
Assume that function $f\in C^{m+1}({\mathbb R}_+)$ is extended to a function $f\in C^{m+1}({\mathbb R})$
(such an extension always exists).
For self-adjoint operators $A,H,$ denote by 
\begin{align*}
S_f^{(m)}(A,H) := f(A+H)-f(A)-\sum_{k=1}^{m}  \frac{1}{k!}\frac{d^k}{dt^k}f(A+tH)\vert_{t=0}
\end{align*}
the remainder of the $m$-th order Taylor expansion of operator function $f(A).$
Let ${\mathcal W}_k$ denote the class of functions $f\in C^{k}({\mathbb R})$  
such that the Fourier transform $\widetilde {f^{(k)}}$ of the $k$-th derivative of $f$
is integrable:
\begin{align*}
\|f\|_{{\mathcal W}_k} := \int_{{\mathbb R}} |\widetilde {f^{(k)}}(t)|dt<\infty.
\end{align*}
It was proved in \cite{P_S_S} (see Theorems 1.1, 2.1) that, for all $f\in \bigcap_{k=0}^{m+1}{\mathcal W}_k$ and all perturbations $H\in {\mathcal S}_{m+1},$ 
there exists 
a function $\eta_{m+1}=\eta_{m+1,f,A,H}\in L_1({\mathbb R})$ such that $S_f^{(m)}(A,H)\in {\mathcal S}_1$ and the following higher order Lifshits-Krein spectral 
shift formula holds 
\begin{align*}
{\rm tr}(S_f^{(m)}(A,H))= \int_{{\mathbb R}} f^{(m+1)}(\lambda)\eta_{m+1}(\lambda)d\lambda
\end{align*}
with
\begin{align*}
\|\eta_{m+1}\|_{L_1} \lesssim_m \|H\|_{m+1}^{m+1}.
\end{align*}
This result immediately provides a bound on the remainder of Taylor expansion 
of order $m$ for trace functional $\tau_f.$ Namely, defining 
\begin{align*}
R_f^{(m)}(A,H) := \tau_f(A+H)-\tau_f(A)-\sum_{k=1}^{m}  \frac{1}{k!}\frac{d^k}{dt^k}\tau_f(A+tH)\vert_{t=0},
\end{align*}
we get 
\begin{align}
\label{bd_on_rem}
|R_f^{(m)}(A,H)| \lesssim_m  \|f^{(m+1)}\|_{L_{\infty}} \|H\|_{m+1}^{m+1}.
\end{align}
It was also shown in \cite{P_S_S} that, for all $t\in [0,1],$ $\frac{d^{m+1}}{dt^{m+1}}f(A+tH)\in {\mathcal S}_1$ and 
\begin{align*}
\sup_{t\in [0,1]}\Bigl|{\rm tr}\Bigl(\frac{d^{m+1}}{dt^{m+1}}f(A+tH)\Bigr)\Bigr| \lesssim_m \|f^{(m+1)}\|_{L_{\infty}} \|H\|_{m+1}^{m+1}.
\end{align*}
The proof of these results relied on representations of higher order derivatives of $f(A)$ as multiple 
operator integrals (see \cite{Peller-06, Peller-16})
\begin{align*}
\frac{d^k}{dt^k}f(A+tH)\vert_{t=0}= (D^k f)(A)[H,\dots, H],
\end{align*}
where 
\begin{align}
\label{mult_int}
(D^k f)(A)[H_1,\dots, H_k]=\sum_{\pi \in S_k}\int_{\mathbb R} \dots \int_{\mathbb R} f^{[k]}(s_1,\dots, s_{k+1}) E_A(ds_1) H_{\pi(1)}dE_A(ds_2)H_{\pi(2)}\dots H_{\pi(k)} dE_A(ds_{k+1}),
\end{align}
$S_k$ being the set of all permutations of $1,\dots, k.$
Here $f^{[k]}$ is the $k$-th order divided difference of function $f,$ defined for $k=0$ as $f^{[0]}=f$ and for $k\geq 1$ as 
\begin{align*}
f^{[k]}(s_1,\dots, s_{k+1}):=
\begin{cases}
\frac{f^{[k-1]}(s_1,\dots, s_{k-1}, s_k)-f^{[k-1]}(s_1,\dots, s_{k-1}, s_{k+1})}{s_k-s_{k+1}} & s_k\neq s_{k+1}\\
\frac{\partial}{\partial t}f^{[k-1]}(s_1,\dots, s_{k-1}, t)\vert_{t=s_k} & s_k=s_{k+1}
\end{cases}
\end{align*}
and $E_A$ a projection valued measure called the resolution of identity of self-adjoint operator $A.$ 
Such integrals reduce to countable sums in the case when $A=\Sigma$ is a covariance operator (or another operator with discrete spectrum).
Using the technique of multiple operator integrals, it was possible to prove that 
\begin{align*}
\sup_{t\in [0,1]}\Bigl|{\rm tr}\Bigl(\frac{d^k}{dt^k}f(A+tH)\Bigr)\Bigr| \lesssim \|f^{(k)}\|_{L_{\infty}}\|H\|_k^{k}
\end{align*}
(under the assumption that $f\in \bigcap_{j=0}^k {\mathcal W}_j$).
Since $(D^k f)(A)[H_1,\dots, H_k]$ is a $k$-linear form and 
\begin{align*}
{\rm tr}\Bigl(\frac{d^k}{dt^k}f(A+tH)\vert_{t=0}\Bigr)={\rm tr}((D^k f)(A)[H,\dots, H]),
\end{align*}
it easily follows that ${\rm tr}((D^k f)(A)[H_1,\dots, H_k])$ is a bounded $k$-linear form on space ${\mathcal S}_k$ (and, for that matter, on all the spaces ${\mathcal S}_p$ for $p\leq k$).

These results allow us to obtain representation formulas for the bias of estimator $\tau_f(\hat \Sigma_n)$ of $\tau_f (\Sigma)$
slightly modifying (and specializing) Lemma 3.1 in \cite{Koltchinskii_2022}.

\begin{proposition}
\label{bias_decomp}
Let $f\in \bigcap_{k=0}^{m+1}{\mathcal W}_k$ for some $m\geq 2$ with $f(0)=0$ and $\|f^{(m+1)}\|_{L_{\infty}}<\infty.$ Suppose that ${\bf r}(\Sigma)\lesssim n.$ Then, the following representation holds 
\begin{align*}
{\mathbb E}\tau_f(\hat \Sigma_n) - \tau_f(\Sigma) = \sum_{l=1}^{m-1} \frac{\beta_{l,m}(\Sigma, f)}{n^l} +R_n,
\end{align*}
where coefficients $\beta_{l,m}(\Sigma, f)$ do not depend on $n$ and
\begin{align}
\label{remain} 
|R_n| \lesssim_m \|f^{(m+1)}\|_{L_{\infty}} 
\|\Sigma\|^{m+1} {\bf r}(\Sigma)\Bigl(\sqrt{\frac{{\bf r}(\Sigma)}{n}}\vee \sqrt{\frac{\log n}{n}}\Bigr)^{m+1}.
\end{align}
\end{proposition}

\begin{proof}
We use the $m$-th order Taylor expansion for the trace functional $\tau_f(\hat \Sigma_n)$ around $\Sigma:$
\begin{align*}
\tau_f (\hat \Sigma_n) = \tau_f(\Sigma) + \sum_{j=1}^{m}\frac{{\rm tr}((D^j f)(\Sigma)[\hat \Sigma_n-\Sigma,\dots, \hat \Sigma_n-\Sigma])}{j!} + R_f^{(m)}(\Sigma, \hat \Sigma_n-\Sigma),
\end{align*}
which implies that
\begin{align*}
{\mathbb E}\tau_f (\hat \Sigma_n) -\tau_f(\Sigma) =\sum_{j=1}^{m}\frac{{\mathbb E}{\rm tr}((D^j f)(\Sigma)[\hat \Sigma_n-\Sigma,\dots, \hat \Sigma_n-\Sigma])}{j!} + R_n,
\end{align*}
where $R_n={\mathbb E}R_f^{(m)}(\Sigma, \hat \Sigma_n-\Sigma).$ Since ${\rm tr}((D^j f)(\Sigma)[\hat \Sigma_n-\Sigma,\dots, \hat \Sigma_n-\Sigma])$ are bounded $j$-linear forms 
on the space ${\mathcal S}_1$ of nuclear operators and $\hat \Sigma_n-\Sigma\in {\mathcal S}_1,$ one can repeat the proof of Lemma 3.1 in \cite{Koltchinskii_2022}
to show that 
\begin{align*}
\sum_{j=1}^{m}\frac{{\mathbb E}{\rm tr}((D^j f)(\Sigma)[\hat \Sigma_n-\Sigma,\dots, \hat \Sigma_n-\Sigma])}{j!} = \sum_{l=1}^{m-1} \frac{\beta_{l,m}(\Sigma, f)}{n^l}
\end{align*}
with coefficients $\beta_{l,m}(\Sigma, f)$ not depending on $n.$ We can now use bound \eqref{bd_on_rem} for $R_f^{(m)}(\Sigma, \hat \Sigma_n-\Sigma)$
and the bound of Proposition \ref{prop_schat} for $p=m+1$ to show that, under the assumption ${\bf r}(\Sigma)\lesssim n,$ 
\begin{align*}
|R_n| &\leq {\mathbb E}|R_f^{(m)}(\Sigma, \hat \Sigma_n-\Sigma)|\lesssim \|f^{(m+1)}\|_{L_{\infty}} {\mathbb E}\|\hat \Sigma_n-\Sigma\|_{m+1}^{m+1}
\\
&
\lesssim_m \|f^{(m+1)}\|_{L_{\infty}} 
\|\Sigma\|^{m+1} {\bf r}(\Sigma)\Bigl(\sqrt{\frac{{\bf r}(\Sigma)}{n}}\vee \sqrt{\frac{\log n}{n}}\Bigr)^{m+1}.
\end{align*}

\qed
\end{proof}

We are ready to provide the proof of Theorem \ref{bias_main}.

\begin{proof}
We will first assume that $f\in \bigcap_{k=0}^{m+1}{\mathcal W}_k.$ 
We have 
\begin{align*}
{\mathbb E}\hat T_f(X_1,\dots,X_n)- \tau_f(\Sigma) = \sum_{j=1}^m C_j ({\mathbb E}\tau_f(\hat \Sigma_{n_j})- \tau_f(\Sigma)).
\end{align*}
Using Proposition \ref{bias_decomp}, we get 
\begin{align*}
{\mathbb E}\hat T_f(X_1,\dots,X_n)- \tau_f(\Sigma) &= \sum_{j=1}^m C_j \sum_{l=1}^{m-1} \frac{\beta_{l,m}(\Sigma, f)}{n_j^l} +\sum_{j=1}^m C_j R_{n_j}
\\
&
=
\sum_{l=1}^{m-1}\sum_{j=1}^m \frac{C_j}{n_j^l}\beta_{l,m}(\Sigma, f)+\sum_{j=1}^m C_j R_{n_j}
\\
&
=\sum_{j=1}^m C_j R_{n_j}
\end{align*}
since 
$
\sum_{j=1}^m \frac{C_j}{n_j^l}=0, l=1,\dots, m-1.
$
Hence 
\begin{align*}
|{\mathbb E}\hat T_f(X_1,\dots,X_n)- \tau_f(\Sigma)| \leq \sum_{j=1}^m |C_j||R_{n_j}| \lesssim_m \max_{1\leq j\leq m}|R_{n_j}|.  
\end{align*}
By bound \eqref{remain},
\begin{align*}
\max_{1\leq j\leq m}|R_{n_j}|\lesssim_m \|f^{(m+1)}\|_{L_{\infty}} 
\|\Sigma\|^{m+1} {\bf r}(\Sigma)\max_{1\leq j\leq m}\Bigl(\sqrt{\frac{{\bf r}(\Sigma)}{n_j}}\vee \sqrt{\frac{\log n_j}{n_j}}\Bigr)^{m+1}.
\end{align*}
Under the condition that $n_j\asymp_m n$ for all $j=1,\dots, m,$ this yields the bound 
\begin{align}
\label{bias_W}
|{\mathbb E}\hat T_f(X_1,\dots,X_n)- \tau_f(\Sigma)|\lesssim_m
\|f^{(m+1)}\|_{L_{\infty}} \|\Sigma\|^{m+1} {\bf r}(\Sigma)\Bigl(\sqrt{\frac{{\bf r}(\Sigma)}{n}}\vee \sqrt{\frac{\log n}{n}}\Bigr)^{m+1}.
\end{align}

Next, we extend the bound to all functions $f\in C^{m+1}({\mathbb R})\cap L_1({\mathbb R})$ with $f(0)=0$ and $\|f^{(m+1)}\|_{L_{\infty}}<\infty.$ 
To this end, we will construct a sequence $\{f_k\}$ of functions in ${\mathbb R}$
such that $f_k\in \bigcap_{j=0}^{m+1}{\mathcal W}_j$ and $f_k$ converges to $f$ pointwise as $k\to \infty.$
Let $p$ be the $(m+2)$-fold convolution of uniform densities in $[-1,1].$ 
It is supported in the interval $[-m-2, m+2],$ $p(0)>0$  
and its characteristic function is 
\begin{align*}
\widetilde{p}(t)=\Bigl(\frac{\sin (t)}{t}\Bigr)^{m+3}, t\in {\mathbb R}.
\end{align*}
It is easy to see that $p$ is $m+1$ times continuously differentiable and the functions
\begin{align*}
\widetilde{p^{(j)}}(t)= (-it)^j \widetilde{p}(t), t\in {\mathbb R}
\end{align*}
are integrable for all $j=0,\dots, m+1.$ Thus, $p\in \bigcap_{j=0}^{m+1}{\mathcal W}_j.$
Let 
\begin{align*}
p_k (x):= \frac{1}{\sigma_k}p\Bigl(\frac{x}{\sigma_k}\Bigr), x\in {\mathbb R},
\end{align*}
where $\sigma_k\in (0,1), \sigma_k\to 0$ as $k\to \infty.$ 
Clearly, $p_k \in \bigcap_{j=0}^{m+1}{\mathcal W}_j$ for all $k\geq 1.$

Define 
\begin{align*}
f_k(x):= (f\ast p_k)(x)-\frac{(f\ast p_k)(0)}{p(0)}p(x), x\in {\mathbb R}.
\end{align*}
Clearly, $f_k(0)=0.$ Also, by continuity of $f,$ we easily get that 
$(f\ast p_k)(x)\to f(x)$ as $k\to \infty$ for all $x\in {\mathbb R}.$
In particular, $(f\ast p_k)(0)\to f(0)=0$ as $k\to \infty.$ Therefore, 
$f_k(x)\to f(x)$ as $k\to \infty.$

By simple properties of convolution, we have 
\begin{align*}
f_k^{(j)}(x)= (f\ast p_k^{(j)})(x)- \frac{(f\ast p_k)(0)}{p(0)}p^{(j)}(x), j=0,\dots, m+1,
\end{align*}
which easily implies that 
\begin{align*}
\widetilde{f_k^{(j)}}(t) =\widetilde{f}(t) \widetilde{p_k^{(j)}}(t) - \frac{(f\ast p_k)(0)}{p(0)}\widetilde{p^{(j)}}(t), j=0,\dots, m+1.
\end{align*}
Since, $p, p_k \in\bigcap_{j=0}^{m+1}{\mathcal W}_j$ and $f\in L_1({\mathbb R})$ (implying that $\tilde f$ is uniformly bounded),
it follows that $f_k\in \bigcap_{j=0}^{m+1}{\mathcal W}_j$ for all $k\geq 1.$
Thus, by \eqref{bias_W}, 
\begin{align}
\label{bias_W_k}
|{\mathbb E}\hat T_{f_k}(X_1,\dots,X_n)- \tau_{f_k}(\Sigma)|\lesssim_m
\|f_k^{(m+1)}\|_{L_{\infty}} \|\Sigma\|^{m+1} {\bf r}(\Sigma)\Bigl(\sqrt{\frac{{\bf r}(\Sigma)}{n}}\vee \sqrt{\frac{\log n}{n}}\Bigr)^{m+1}.
\end{align}

Note also that 
\begin{align*}
f_k^{(j)}(x):= (f^{(j)}\ast p_k)(x)-\frac{(f\ast p_k)(0)}{p(0)}p^{(j)}(x), x\in {\mathbb R}, j=0,\dots, m+1.
\end{align*}
Using the Taylor expansion of order $m$ of function $f$ and the fact that $f(0)=0,$ it is easy to check that 
\begin{align*}
|f(-y)| \leq \frac{|f'(0)|}{1!}|y| + \dots + \frac{|f^{(m)}(0)|}{m!}|y|^m + \frac{\|f^{(m+1)}\|_{L_{\infty}}}{(m+1)!}|y|^{m+1}.
\end{align*}
Since 
\begin{align*}
(f\ast p_k)(0)= \int_{{\mathbb R}}f(-y) p_k(y)dy,
\end{align*}
it is straightforward to see that 
\begin{align}
\label{f_k_at_0}
|(f\ast p_k)(0)| \lesssim \sum_{l=1}^{m}  \frac{|f^{(l)}(0)|}{l!} \sigma_k^l + \frac{\|f^{(m+1)}\|_{L_{\infty}}}{(m+1)!}\sigma_k^{m+1}.
\end{align}
This yields the bound 
\begin{align}
\label{f_k''_inf}
\nonumber
\|f_k^{(m+1)}\|_{L_{\infty}}&\leq \|f^{(m+1)}\ast p_k\|_{L_{\infty}}+\frac{|(f\ast p_k)(0)|}{p(0)}\|p^{(m+1)}\|_{L_{\infty}}
\\
&
\lesssim_m \|f^{(m+1)}\|_{L_{\infty}}+ (|f'(0)|\vee \dots \vee |f^{(m)}(0)|)\sigma_k
\end{align}
and we can rewrite \eqref{bias_W_k} as follows:
\begin{align}
\label{bias_W_k_F}
&
\nonumber
|{\mathbb E}\hat T_{f_k}(X_1,\dots,X_n)- \tau_{f_k}(\Sigma)|
\\
&
\lesssim_m
(\|f^{(m+1)}\|_{L_{\infty}}+ (|f'(0)|\vee \dots \vee |f^{(m)}(0)|)\sigma_k)\|\Sigma\|^{m+1} {\bf r}(\Sigma)\Bigl(\sqrt{\frac{{\bf r}(\Sigma)}{n}}\vee \sqrt{\frac{\log n}{n}}\Bigr)^{m+1}.
\end{align}
It remains to justify passing to the limit as $k\to\infty$ in the last inequality. 
Note that, similarly to \eqref{f_k_at_0} and \eqref{f_k''_inf}, we can show that
\begin{align*}
&
|f_k^{(j)}(0)|\lesssim |f^{(j)}(0)|+ (|f^{(j+1)}(0)|\vee \dots \vee |f^{(m)}(0)|\vee \|f^{(m+1)}\|_{L_{\infty}})\sigma_k
\\
&
\lesssim |f^{(j)}(0)|\vee |f^{(j+1)}(0)|\vee \dots \vee |f^{(m)}(0)|\vee \|f^{(m+1)}\|_{L_{\infty}}.
\end{align*}
Therefore, by the $m$-th order Taylor expansion, for all $x\geq 0,$
\begin{align*}
|f_k(x)|\lesssim (|f'(0)|\vee \dots \vee |f^{(m)}(0)|)\vee \|f^{(m+1)}\|_{L_{\infty}})(x+x^2+\dots +x^{m+1}).
\end{align*}
Since $f_k$ converges to $f$ pointwise as $k\to\infty,$ 
\begin{align*}
\int_{{\mathbb R}_+}(x+x^2+\dots + x^{m+1})\mu_{\Sigma}(dx)= {\rm tr}(\Sigma)+ \|\Sigma\|_2^2 +\dots + \|\Sigma\|_{m+1}^{m+1}<\infty
\end{align*}
and 
\begin{align*}
{\mathbb E}\int_{{\mathbb R}_+}(x+x^2+\dots+ x^{m+1})\mu_{\hat \Sigma_n}(dx)= {\mathbb E}({\rm tr}(\hat \Sigma_n)+ \|\hat \Sigma_n\|_2^2+\dots +\|\hat \Sigma_n\|_{m+1}^{m+1})<\infty,
\end{align*}
we can use the dominated convergence theorem to show that 
\begin{align*}
\tau_{f_k}(\Sigma)=\int_{{\mathbb R}_+}f_k d\mu_{\Sigma}\to \int_{{\mathbb R}_+}f d\mu_{\Sigma}=\tau_f(\Sigma)
\end{align*}
and 
\begin{align*}
{\mathbb E}\tau_{f_k}(\hat \Sigma_n)={\mathbb E}\int_{{\mathbb R}_+}f_k d\mu_{\hat \Sigma_n}\to {\mathbb E}\int_{{\mathbb R}_+}f d\mu_{\hat \Sigma_n}={\mathbb E}\tau_f(\hat \Sigma_n)
\end{align*}
as $k\to \infty.$ The same convergence properties hold for $\tau_{f_k}(\hat \Sigma_{n_j}), j=1,\dots, m,$ which implies that 
\begin{align*}
{\mathbb E}\hat T_{f_k}(X_1,\dots,X_n)- \tau_{f_k}(\Sigma)\to {\mathbb E}\hat T_{f}(X_1,\dots,X_n)- \tau_{f}(\Sigma)\ {\rm as}\ k\to\infty.
\end{align*}
We can now pass to the limit in bound \eqref{bias_W_k_F} as $k\to\infty$ to complete the proof under the assumption that $f\in C^{m+1}({\mathbb R})\cap L_1({\mathbb R}).$

Finally, we need to get rid of the assumption that $f\in L_1({\mathbb R}).$ Let $\psi$ be a $C^{\infty}$ function in ${\mathbb R}$ with values in $[0,1],$ with ${\rm supp}(\psi)\subset [-1,1]$ and $\psi(x)=1, x\in [-1/2/1/2].$ Denote $\psi_k(x):= \psi(2^{-k}x), x\in {\mathbb R}.$ It is easy to see that $\{\psi_k\}$ is a nondecreasing sequence of functions 
and $\psi_k(x) \to 1$ as $k\to\infty$ for all $x\in {\mathbb R},$ which implies that $f\psi_k\to f$ pointwise. Note also that, for $f\in C^{m+1},$ we have $f \psi_k \in C^{m+1}$
and $\|(f\psi_k)^{(m+1)}\|_{L_{\infty}}\leq C\|f^{(m+1)}\|_{L_{\infty}}$ with a constant $C>0$ that depends only on $\psi$ and on $m.$ Since $f\psi_k$ are continuous functions with bounded 
supports, we have $f\psi_k\in L_1({\mathbb R}), k\geq 1.$ Thus, the bound of Theorem \ref{bias_main} holds for functions $f\psi_k$ (and $\|(f\psi_k)^{(m+1)}\|_{L_{\infty}}$ 
could be replaced by $\|f^{(m+1)}\|_{L_{\infty}}$ up to a constant that depends only on $m$).  To finish the proof, it is enough to show that 
\begin{align}
\label{conver_tau}
\tau_{f\psi_k}(\Sigma)\to \tau_f(\Sigma)\ {\rm and}\  {\mathbb E}\tau_{f\psi_k}(\hat \Sigma_n)\to {\mathbb E}\tau_f(\hat \Sigma_n)
\end{align}
as $k\to \infty$ for all $n\geq 1.$ This would imply that 
\begin{align*}
{\mathbb E}\hat T_{f\psi_k}(X_1,\dots,X_n)- \tau_{f\psi_k}(\Sigma)\to {\mathbb E}\hat T_{f}(X_1,\dots,X_n)- \tau_{f}(\Sigma)\ {\rm as}\ k\to\infty,
\end{align*}
and the bound of Theorem \ref{bias_main} would follow. 
It is easy to check that, for a nonnegative function $f,$ the convergence \eqref{conver_tau} follows by the monotone convergence theorem. 
In the general case, it is enough use the representation $f=f_{+}-f_{-},$ where $f_{+}:= f\vee0, f_{-}=-(f\wedge 0).$

\qed
\end{proof}

\section{Proofs of the main results}
\label{proof_of_main}


The proof of Theorem \ref{Th_M_YYY} easily follows from the following more detailed result.

\begin{theorem}
\label{Th_M}
Suppose $f\in C^{m+1}({\mathbb R}_{+})$ for some $m\geq 2,$ $f(0)=0,$ $\|f'\|_{L_{\infty}}<\infty,$ $\|f'\|_{{\rm Lip}}<\infty$ 
and $\|f^{(m+1)}\|_{L_{\infty}}<\infty.$
Suppose also 
that ${\bf r}(\Sigma)\lesssim n.$ Then, for all $p\geq 1,$
\begin{align*}
&
\|\hat T_{f,m}(X_1,\dots, X_n)-\tau_f(\Sigma)\|_{L_p} 
\\
&
\lesssim_m 
\Bigl(\|\Sigma f'(\Sigma)\|_2\sqrt{\frac{p}{n}}\vee \|\Sigma f'(\Sigma)\|\frac{p}{n}\Bigr)
+\|f'\|_{\rm Lip} \|\Sigma\|^2
\Bigl(\frac{{\bf r}(\Sigma)}{\sqrt{n}}\sqrt{\frac{p}{n}}\vee \Bigl(\frac{{\bf r}(\Sigma)}{\sqrt{n}}\vee 1\Bigr)\frac{p}{n} \vee\Bigl(\frac{p}{n}\Bigr)^{2}\Bigr)
\\
&
\ \ \ \ +\|f^{(m+1)}\|_{L_{\infty}} 
\|\Sigma\|^{m+1} {\bf r}(\Sigma)\Bigl(\sqrt{\frac{{\bf r}(\Sigma)}{n}}\vee \sqrt{\frac{\log n}{n}}\Bigr)^{m+1}
\\
&\lesssim_m 
\|f'\|_{L_{\infty}}\|\Sigma\|
\Bigl(\sqrt{{\bf r}(\Sigma^2)}\sqrt{\frac{p}{n}}\vee \frac{p}{n}\Bigr)
+\|f'\|_{\rm Lip} \|\Sigma\|^2
\Bigl(\frac{{\bf r}(\Sigma)}{\sqrt{n}}\sqrt{\frac{p}{n}}\vee \Bigl(\frac{{\bf r}(\Sigma)}{\sqrt{n}}\vee 1\Bigr)\frac{p}{n} \vee\Bigl(\frac{p}{n}\Bigr)^{2}\Bigr)
\\
&
\ \ \ \ +\|f^{(m+1)}\|_{L_{\infty}} 
\|\Sigma\|^{m+1} {\bf r}(\Sigma)\Bigl(\sqrt{\frac{{\bf r}(\Sigma)}{n}}\vee \sqrt{\frac{\log n}{n}}\Bigr)^{m+1}.
\end{align*}
\end{theorem}

\begin{proof}
Since 
\begin{align*}
\hat T_f(X_1,\dots, X_n)-{\mathbb E}\hat T_f(X_1,\dots, X_n)= \sum_{j=1}^m C_j (\tau_f(\hat \Sigma_{n_j}) - {\mathbb E}\tau_f(\hat \Sigma_{n_j})),
\end{align*}
we have 
\begin{align*}
\Bigl\|\hat T_f(X_1,\dots, X_n)-{\mathbb E}\hat T_f(X_1,\dots, X_n)\Bigr\|_{L_p}
&\lesssim \sum_{j=1}^m |C_j | \Bigl\|\tau_f(\hat \Sigma_{n_j}) - {\mathbb E}\tau_f(\hat \Sigma_{n_j})\Bigr\|_{L_p}
\\
&
\lesssim_m \max_{1\leq j\leq m}\Bigl\|\tau_f(\hat \Sigma_{n_j}) - {\mathbb E}\tau_f(\hat \Sigma_{n_j})\Bigr\|_{L_p}.
\end{align*}
Using the bound of Corollary \ref{conc_tau_f}, 
we get 
\begin{align*}
&
\max_{1\leq j\leq m}\Bigl\|\tau_f(\hat \Sigma_{n_j}) - {\mathbb E}\tau_f(\hat \Sigma_{n_j})\Bigr\|_{L_p}
\\
&
\lesssim_m 
\max_{1\leq j\leq m}\Bigl(\|\Sigma f'(\Sigma)\|_2 \sqrt{\frac{p}{n_j}}\vee \|\Sigma f'(\Sigma)\|\frac{p}{n_j}\Bigr)
+\|f'\|_{\rm Lip} \|\Sigma\|^2
\max_{1\leq j\leq m}\Bigl(\frac{{\bf r}(\Sigma)}{\sqrt{n_j}}\sqrt{\frac{p}{n_j}}\vee \Bigl(\frac{{\bf r}(\Sigma)}{\sqrt{n_j}}\vee 1\Bigr)\frac{p}{n_j} \vee\Bigl(\frac{p}{n_j}\Bigr)^{2}\Bigr).
\end{align*}
Under the condition that $n_j\asymp_m n$ for all $j=1,\dots, m,$ we conclude that 
\begin{align*}
&
\Bigl\|\hat T_f(X_1,\dots, X_n)-{\mathbb E}\hat T_f(X_1,\dots, X_n)\Bigr\|_{L_p}
\\
&
\lesssim_m
\Bigl(\|\Sigma f'(\Sigma)\|_2 \sqrt{\frac{p}{n}}\vee \|\Sigma f'(\Sigma)\|\frac{p}{n}\Bigr)
+\|f'\|_{\rm Lip} \|\Sigma\|^2
\Bigl(\frac{{\bf r}(\Sigma)}{\sqrt{n}}\sqrt{\frac{p}{n}}\vee \Bigl(\frac{{\bf r}(\Sigma)}{\sqrt{n}}\vee 1\Bigr)\frac{p}{n} \vee\Bigl(\frac{p}{n}\Bigr)^{2}\Bigr).
\end{align*}
Using bounds on the bias of Theorem \ref{bias_main}, it is easy to complete the proof of Theorem \ref{Th_M}.

\qed
\end{proof}

The proof of Theorem \ref{Th_M1_YYY} immediately follows from the next result.

\begin{theorem}
\label{Th_M1}
Suppose that ${\bf r}(\Sigma)\lesssim n.$ Let $p\geq 1.$
Then, for all $m\geq 4,$
\begin{align}
\label{bd_m>3_B}
&
\nonumber
\Big\|\|\hat \mu_{n,m} - \mu_{\Sigma}\|_{{\mathcal F}_m}\Bigr\|_{L_p}
\nonumber
\lesssim_m (\|\Sigma\|+1)^{1/2}\|\Sigma\|\sqrt{{\bf r}(\Sigma^2)} \sqrt{\frac{p}{n}}  + (\|\Sigma\|+1)\|\Sigma\| \frac{p}{n}
\\
&
\nonumber
+
\|\Sigma\|^2\Bigl((\|\Sigma\|+1)^{1/2}\frac{{\bf r}(\Sigma)}{\sqrt{n}}\sqrt{\frac{p}{n}} +(\|\Sigma\|+1)\Bigl(\frac{{\bf r}(\Sigma)}{\sqrt{n}}\vee 1\Bigr)\frac{p}{n} 
+ (\|\Sigma\|+1)^2\Bigl(\frac{{\bf r}(\Sigma)^2}{n}\vee 1\Bigr)\Bigl(\frac{p}{n}\Bigr)^{2}\Bigr)
\\
&
+ \|\Sigma\|^{m+1} {\bf r}(\Sigma)\Bigl(\sqrt{\frac{{\bf r}(\Sigma)}{n}}\vee \sqrt{\frac{\log n}{n}}\Bigr)^{m+1};
\end{align}
for $m=3,$
\begin{align}
\label{bd_m_3_B}
&
\nonumber
\Big\|\|\hat \mu_{n,m} - \mu_{\Sigma}\|_{{\mathcal F}_m}\Bigr\|_{L_p}
\lesssim 
(\|\Sigma\|+1)^{1/2}\|\Sigma\|\sqrt{{\bf r}(\Sigma^2)} \sqrt{\frac{p}{n}}  + (\|\Sigma\|+1)\|\Sigma\| \frac{p}{n}
\\
&
\nonumber
+\|\Sigma\|^2\Bigl((\|\Sigma\|+1)^{1/2}\frac{{\bf r}(\Sigma)}{\sqrt{n}}\sqrt{\frac{p}{n}} +(\|\Sigma\|+1)\Bigl(\frac{{\bf r}(\Sigma)}{\sqrt{n}}\vee 1\Bigr)\frac{p}{n} 
+ (\|\Sigma\|+1)^2 \Bigl(\frac{{\bf r}(\Sigma)^2}{n}\vee \log n\Bigr)\Bigl(\frac{p}{n}\Bigr)^{2}\Bigr)
\\
&
+ \|\Sigma\|^{m+1} {\bf r}(\Sigma)\Bigl(\sqrt{\frac{{\bf r}(\Sigma)}{n}}\vee \sqrt{\frac{\log n}{n}}\Bigr)^{m+1};
\end{align}
and, for $m=2,$
\begin{align}
\label{bd_m_2_B}
&
\nonumber
\Big\|\|\hat \mu_{n,m} - \mu_{\Sigma}\|_{{\mathcal F}_m}\Bigr\|_{L_p}
\lesssim 
(\|\Sigma\|+1)^{1/2}\|\Sigma\|\sqrt{{\bf r}(\Sigma^2)} \sqrt{\frac{p}{n}}  + (\|\Sigma\|+1)\|\Sigma\| \frac{p}{n}
\\
&
\nonumber
+
\|\Sigma\|^2\Bigl((\|\Sigma\|+1)^{1/2}\frac{{\bf r}(\Sigma)}{\sqrt{n}}\sqrt{\frac{p}{n}} +(\|\Sigma\|+1)\log n \Bigl(\frac{{\bf r}(\Sigma)}{\sqrt{n}}\vee 1\Bigr)\frac{p}{n}
+ (\|\Sigma\|+1)^2\frac{p^2}{n}\Bigr)
\\
&
+ \|\Sigma\|^{m+1} {\bf r}(\Sigma)\Bigl(\sqrt{\frac{{\bf r}(\Sigma)}{n}}\vee \sqrt{\frac{\log n}{n}}\Bigr)^{m+1}.
\end{align}
\end{theorem}

\begin{proof}
As in the proof of Theorem \ref{Th_M},
\begin{align*}
\int_{{\mathbb R}_+} f d(\hat \mu_n-\mu_{\Sigma}) &= \sum_{j=1}^m C_j (\tau_f(\hat \Sigma_{n_j})-\tau_f(\Sigma))
\\
&
= \sum_{j=1}^m C_j (\tau_f(\hat \Sigma_{n_j})-{\mathbb E}\tau_f(\hat \Sigma_{n_j})) + {\mathbb E}\sum_{j=1}^m C_j \tau_f(\hat \Sigma_{n_j})-\tau_f(\Sigma)
\\
&
=\sum_{j=1}^m C_j (\tau_f(\hat \Sigma_{n_j})-{\mathbb E}\tau_f(\hat \Sigma_{n_j})) + {\mathbb E}\hat T_f(X_1,\dots, X_n)-\tau_f(\Sigma).
\end{align*}
Hence,
\begin{align*}
&
\Bigl\|\sup_{f\in {\mathcal F}_m}\Bigl|\int_{{\mathbb R}_+} f d(\hat \mu_n-\mu_{\Sigma})\Bigr|\Bigr\|_{L_p} 
\\
&
\leq \sum_{j=1}^m |C_j| \Bigl\|\sup_{f\in {\mathcal F}_m}\Bigl|\tau_f(\hat \Sigma_{n_j})-{\mathbb E}\tau_f(\hat \Sigma_{n_j})\Bigr|\Bigr\|_{L_p} +
\sup_{f\in {\mathcal F}_m} |{\mathbb E}\hat T_f(X_1,\dots, X_n)-\tau_f(\Sigma)|
\\
&
\lesssim 
\max_{1\leq j\leq m}\Bigl\|\sup_{f\in {\mathcal F}_m}\Bigl|\tau_f(\hat \Sigma_{n_j})-{\mathbb E}\tau_f(\hat \Sigma_{n_j})\Bigr|\Bigr\|_{L_p} +
\sup_{f\in {\mathcal F}_m} |{\mathbb E}\hat T_f(X_1,\dots, X_n)-\tau_f(\Sigma)|,
\end{align*}
and it is enough to use the bounds of Proposition \ref{conc_F_smooth} and Theorem \ref{bias_main} to complete the proof of Theorem \ref{Th_M1}.

\qed
\end{proof}

Next we provide the proof of  Theorem \ref{Th_M_A}.

\begin{proof}
The claims of Theorem \ref{Th_M_A} are obvious in view of theorems \ref{Th_M}, \ref{Th_M1} and the following simple inequalities:
\begin{align*}
&
{\mathbb E}\Bigl|\check T_{f,m}(X_1,\dots, X_n)- \tau_f(\Sigma)\Bigr|^p = {\mathbb E}\Bigl|{\mathbb E}(\hat T_{f,m}(X_1,\dots, X_n)|{\mathcal A}_{\rm sym})- \tau_f(\Sigma)\Bigr|^p
\\
&
={\mathbb E}\Bigl|{\mathbb E}(\hat T_{f,m}(X_1,\dots, X_n)- \tau_f(\Sigma))|{\mathcal A}_{\rm sym})\Bigr|^p
\leq {\mathbb E} {\mathbb E}(|\hat T_{f,m}(X_1,\dots, X_n)- \tau_f(\Sigma)|^p|{\mathcal A}_{\rm sym})
\\
&
=
{\mathbb E} |\hat T_{f,m}(X_1,\dots, X_n)- \tau_f(\Sigma)|^p
\end{align*}
and 
\begin{align*}
&
{\mathbb E} \sup_{f\in {\mathcal F}_m} \Bigl| \int_{\mathbb R_{+}} fd\check \mu_{n,m}- \int_{{\mathbb R}_+}fd\mu_{\Sigma}\Bigr|^p
= 
{\mathbb E} \sup_{f\in {\mathcal F}_m} \Bigl| {\mathbb E}\Bigl(\int_{\mathbb R_{+}} fd\hat \mu_{n,m}- \int_{{\mathbb R}_+}fd\mu_{\Sigma}|{\mathcal A}_{\rm sym}\Bigr)\Bigr|^p
\\
&
\leq 
{\mathbb E} {\mathbb E}\Bigl(\sup_{f\in {\mathcal F}_m} \Bigl|\int_{\mathbb R_{+}} fd\hat \mu_{n,m}- \int_{{\mathbb R}_+}fd\mu_{\Sigma}\Bigr|^p| {\mathcal A}_{\rm sym}\Bigr)=
{\mathbb E} \sup_{f\in {\mathcal F}_m} \Bigl|\int_{\mathbb R_{+}} fd\hat \mu_{n,m}- \int_{{\mathbb R}_+}fd\mu_{\Sigma}\Bigr|^p.
\end{align*}

\qed
\end{proof}

The proof of Theorem \ref{Th_M_B_YYY} is an immediate corollary of the following result.

\begin{theorem}
\label{Th_M_B}
Suppose $f\in C^{m+1}({\mathbb R}_{+})$ for some $m\geq 2,$ $f(0)=0,$ $\|f'\|_{\rm Lip}<\infty$ and $\|f^{(m+1)}\|_{L_{\infty}}<\infty.$
Suppose also 
that ${\bf r}(\Sigma)\lesssim n.$ Then, for all $p\geq 1,$
\begin{align*}
&
\|\check T_{f,m}(X_1,\dots, X_n)-\tau_f(\Sigma)- \langle f'(\Sigma), \hat \Sigma_n-\Sigma\rangle\|_{L_p} 
\\
&\lesssim_m 
\|f'\|_{\rm Lip} \|\Sigma\|^2\Bigl(\frac{{\bf r}(\Sigma)}{\sqrt{n}}\sqrt{\frac{p}{n}} \vee \Bigl(\frac{{\bf r}(\Sigma)}{\sqrt{n}}\vee 1\Bigr)\frac{p}{n}\vee \Bigl(\frac{p}{n}\Bigr)^{2}\Bigr)
+ 
\|f^{(m+1)}\|_{L_{\infty}} 
\|\Sigma\|^{m+1} {\bf r}(\Sigma)\Bigl(\sqrt{\frac{{\bf r}(\Sigma)}{n}}\vee \sqrt{\frac{\log n}{n}}\Bigr)^{m+1}.
\end{align*}
As a consequence, for all $p\geq 1,$
\begin{align*}
&
W_p\Bigl(\sqrt{\frac{n}{2}}\Bigl(\check T_{f,m}(X_1,\dots, X_n)-\tau_f(\Sigma)\Bigr), G_{\Sigma}(f)\Bigr) 
\\
&
\lesssim_m 
 \frac{\|\Sigma f'(\Sigma)\|_2\sqrt{p}\vee \|\Sigma f'(\Sigma)\| p}{\sqrt{n}} 
+\|f'\|_{\rm Lip} \|\Sigma\|^2{\bf r}(\Sigma)\Bigl(\sqrt{\frac{p}{n}}\vee \frac{p}{n}\Bigr) + \|f'\|_{\rm Lip} \|\Sigma\|^2\sqrt{n}\Bigl(\frac{p}{n}\vee \Bigl(\frac{p}{n}\Bigr)^{2}\Bigr)
\\
&
\ \ \ \ \ \ + 
\|f^{(m+1)}\|_{L_{\infty}} 
\|\Sigma\|^{m+1} {\bf r}(\Sigma)\sqrt{n}\Bigl(\sqrt{\frac{{\bf r}(\Sigma)}{n}}\vee \sqrt{\frac{\log n}{n}}\Bigr)^{m+1}.
\end{align*}
\end{theorem}

\begin{proof}
Note that 
\begin{align*}
&
\check T_{f,m}(X_1,\dots, X_n)-{\mathbb E}\check T_{f,m}(X_1,\dots, X_n)= \sum_{j=1}^m C_j {\mathbb E}(\tau_f(\hat \Sigma_{n_j})|{\mathcal A}_{\rm sym})-
\sum_{j=1}^m C_j {\mathbb E}\tau_f(\hat \Sigma_{n_j})
\\
&
= \sum_{j=1}^m C_j {\mathbb E}(\tau_f(\hat \Sigma_{n_j})-{\mathbb E} \tau_f(\hat \Sigma_{n_j})|{\mathcal A}_{\rm sym}).
\end{align*} 
Since 
\begin{align*}
\tau_f(\hat \Sigma_{n_j})-{\mathbb E} \tau_f(\hat \Sigma_{n_j})= \langle f'(\Sigma), \hat \Sigma_{n_j}-\Sigma\rangle + R_f(\Sigma, \hat \Sigma_{n_j}-\Sigma)-
{\mathbb E}R_f(\Sigma, \hat \Sigma_{n_j}-\Sigma),
\end{align*}
\begin{align*}
{\mathbb E}\Bigl(\langle f'(\Sigma), \hat \Sigma_{n_j}-\Sigma\rangle|{\mathcal A}_{\rm sym}\Bigr) = 
\langle f'(\Sigma), {\mathbb E}(\hat \Sigma_{n_j}|{\mathcal A}_{\rm sym})-\Sigma\rangle = \langle f'(\Sigma), \hat \Sigma_n-\Sigma\rangle
\end{align*}
and $\sum_{j=1}^m C_j=1,$ we easily get that 
\begin{align*}
&
\check T_{f,m}(X_1,\dots, X_n)-{\mathbb E}\check T_{f,m}(X_1,\dots, X_n)- \langle f'(\Sigma), \hat \Sigma_n-\Sigma\rangle
\\
&
= \sum_{j=1}^m C_j {\mathbb E}(R_f(\Sigma, \hat \Sigma_{n_j}-\Sigma)-
{\mathbb E}R_f(\Sigma, \hat \Sigma_{n_j}-\Sigma)|{\mathcal A}_{\rm sym}),
\end{align*}
which implies 
\begin{align}
\label{represent_lin_T_f,m}
&
\nonumber
\check T_{f,m}(X_1,\dots, X_n)-\tau_f(\Sigma)- \langle f'(\Sigma), \hat \Sigma_n-\Sigma\rangle
\\
&
= {\mathbb E} \check T_{f,m}(X_1,\dots, X_n)-\tau_f(\Sigma)+ 
\sum_{j=1}^m C_j {\mathbb E}(R_f(\Sigma, \hat \Sigma_{n_j}-\Sigma)-{\mathbb E}R_f(\Sigma, \hat \Sigma_{n_j}-\Sigma)|{\mathcal A}_{\rm sym}).
\end{align}
Note that ${\mathbb E}\check T_{f,m}(X_1,\dots, X_n)={\mathbb E}\hat T_{f,m}(X_1,\dots, X_n),$ so, the bias of estimator 
$\check T_{f,m}(X_1,\dots, X_n)$ can be controlled by the bound of Theorem \ref{bias_main}.
It remains to observe that 
\begin{align*}
&
\Bigl\|\sum_{j=1}^m C_j {\mathbb E}(R_f(\Sigma, \hat \Sigma_{n_j}-\Sigma)-{\mathbb E}R_f(\Sigma, \hat \Sigma_{n_j}-\Sigma)|{\mathcal A}_{\rm sym})\|_{L_p}
\\
&
\leq \sum_{j=1}^m |C_j| \max_{1\leq j\leq m} 
\Bigl\|{\mathbb E}(R_f(\Sigma, \hat \Sigma_{n_j}-\Sigma)-{\mathbb E}R_f(\Sigma, \hat \Sigma_{n_j}-\Sigma)|{\mathcal A}_{\rm sym})\Bigr\|_{L_p}
\\
&
\lesssim_m \max_{1\leq j\leq m} \Bigl\|{\mathbb E}(R_f(\Sigma, \hat \Sigma_{n_j}-\Sigma)-{\mathbb E}R_f(\Sigma, \hat \Sigma_{n_j}-\Sigma)|{\mathcal A}_{\rm sym})\Bigr\|_{L_p}
\\
&
\lesssim_m \max_{1\leq j\leq m} \Bigl\|R_f(\Sigma, \hat \Sigma_{n_j}-\Sigma)-{\mathbb E}R_f(\Sigma, \hat \Sigma_{n_j}-\Sigma))\Bigr\|_{L_p},
\end{align*}
to use the bound of Theorem \ref{prop_rem} and to recall that $n_j\asymp_m n, j=1,\dots, m,$ to complete the proof of the 
first claim of Theorem \ref{Th_M_B}.

The second claim easily follows from the first claim and Proposition \ref{norm_approx_lin}.

\qed
\end{proof}

Finally, Theorem \ref{Th_M_C_YYY} is an immediate consequence of the following result.

\begin{theorem}
\label{Th_M_C}
Suppose that ${\bf r}(\Sigma)\lesssim n.$
Then,
for $m \geq 4,$  
\begin{align}
\label{bd_m>3_A'_X}
&
\nonumber
{\mathcal W}_{{\mathcal F}_m, p} (\check G_n, G_{\Sigma})
\lesssim_m (\|\Sigma\|+1)^{1/2}\|\Sigma\| \sqrt{{\bf r}(\Sigma^2)}\sqrt{\frac{p}{n}}  + (\|\Sigma\|+1)\|\Sigma\| \frac{p}{\sqrt{n}}
\\
&
\nonumber
+
\|\Sigma\|^2\Bigl((\|\Sigma\|+1)^{1/2}{\bf r}(\Sigma)\sqrt{\frac{p}{n}} +(\|\Sigma\|+1)\Bigl({\bf r}(\Sigma)\frac{p}{n}\vee \frac{p}{\sqrt{n}}\Bigr)
+ (\|\Sigma\|+1)^2\Bigl(\frac{{\bf r}(\Sigma)^2}{n}\vee 1\Bigr)\frac{p^2}{n^{3/2}}\Bigr)
\\
&
+ 
\|\Sigma\|^{m+1} {\bf r}(\Sigma)\sqrt{n}\Bigl(\sqrt{\frac{{\bf r}(\Sigma)}{n}}\vee \sqrt{\frac{\log n}{n}}\Bigr)^{m+1};
\end{align}
for $m=3,$
\begin{align}
\label{bd_m_3_A'_X}
&
\nonumber
{\mathcal W}_{{\mathcal F}_m, p} (\check G_n, G_{\Sigma})
\lesssim (\|\Sigma\|+1)^{1/2}\|\Sigma\| \sqrt{{\bf r}(\Sigma^2)}\sqrt{\frac{p}{n}}  + (\|\Sigma\|+1)\|\Sigma\| \frac{p}{\sqrt{n}}
\\
&
\nonumber
+\|\Sigma\|^2\Bigl((\|\Sigma\|+1)^{1/2}{\bf r}(\Sigma)\sqrt{\frac{p}{n}} +(\|\Sigma\|+1)\Bigl({\bf r}(\Sigma)\frac{p}{n} \vee \frac{p}{\sqrt{n}}\Bigr)
+ (\|\Sigma\|+1)^2 \Bigl(\frac{{\bf r}(\Sigma)^2}{n}\vee \log n\Bigr)\frac{p^2}{n^{3/2}}\Bigr)
\\
&
+
\|\Sigma\|^{m+1} {\bf r}(\Sigma)\sqrt{n}\Bigl(\sqrt{\frac{{\bf r}(\Sigma)}{n}}\vee \sqrt{\frac{\log n}{n}}\Bigr)^{m+1};
\end{align}
and, for $m=2,$
\begin{align}
\label{bd_m_2_A'_X}
&
\nonumber
{\mathcal W}_{{\mathcal F}_m, p} (\check G_n, G_{\Sigma})
\lesssim (\|\Sigma\|+1)^{1/2}\|\Sigma\| \sqrt{{\bf r}(\Sigma^2)} \sqrt{\frac{p}{n}}  + (\|\Sigma\|+1)\|\Sigma\| \frac{p}{\sqrt{n}}
\\
&
\nonumber
+
\|\Sigma\|^2\Bigl((\|\Sigma\|+1)^{1/2}{\bf r}(\Sigma)\sqrt{\frac{p}{n}} +(\|\Sigma\|+1)\log n \Bigl({\bf r}(\Sigma)\frac{p}{n}\vee \frac{p}{\sqrt{n}}\Bigr)
+ (\|\Sigma\|+1)^2\frac{p^2}{\sqrt{n}}\Bigr)
\\
&
+
\|\Sigma\|^{m+1} {\bf r}(\Sigma)\sqrt{n}\Bigl(\sqrt{\frac{{\bf r}(\Sigma)}{n}}\vee \sqrt{\frac{\log n}{n}}\Bigr)^{m+1}.
\end{align}
\end{theorem}

\begin{proof}
Note that, by representation \eqref{represent_lin_T_f,m},
\begin{align*}
&
\nonumber
\check G_n(f)= G_n(f)
+
\sqrt{\frac{n}{2}}({\mathbb E} \check T_{f,m}(X_1,\dots, X_n)-\tau_f(\Sigma))
\\
&
+ 
\sqrt{\frac{n}{2}}\sum_{j=1}^m C_j {\mathbb E}(R_f(\Sigma, \hat \Sigma_{n_j}-\Sigma)-{\mathbb E}R_f(\Sigma, \hat \Sigma_{n_j}-\Sigma)|{\mathcal A}_{\rm sym}).
\end{align*}
Therefore, 
\begin{align*}
{\mathcal W}_{{\mathcal F}_m, p}(\check G_n, G_{\Sigma}) &\leq {\mathcal W}_{{\mathcal F}_m, p}(G_n, G_{\Sigma}) + \Bigl\|\sup_{f\in {\mathcal F}_m}|\check G_n(f)-G_n(f)|\Bigr\|_{L_p}
\\
&
\leq {\mathcal W}_{{\mathcal F}_m, p}(G_n, G_{\Sigma})+ \sqrt{\frac{n}{2}}\sup_{f\in {\mathcal F}_m}|{\mathbb E} \check T_{f,m}(X_1,\dots, X_n)-\tau_f(\Sigma)|
\\
&
+
\Bigl\|\sup_{f\in {\mathcal F}_m}\Bigl|\sum_{j=1}^m C_j 
{\mathbb E}(R_f(\Sigma, \hat \Sigma_{n_j}-\Sigma)-{\mathbb E}R_f(\Sigma, \hat \Sigma_{n_j}-\Sigma)|{\mathcal A}_{\rm sym})\Bigr|\Bigr\|_{L_p}.
\end{align*}
Note also that
\begin{align*}
&
\Bigl\|\sup_{f\in {\mathcal F}_m}\Bigl|\sum_{j=1}^m C_j 
{\mathbb E}(R_f(\Sigma, \hat \Sigma_{n_j}-\Sigma)-{\mathbb E}R_f(\Sigma, \hat \Sigma_{n_j}-\Sigma)|{\mathcal A}_{\rm sym})\Bigr|\Bigr\|_{L_p}
\\
&
\leq \sum_{j=1}^m |C_j| \max_{1\leq j\leq m} \Bigl\|\sup_{f\in {\mathcal F}_m}\Bigl|{\mathbb E}(R_f(\Sigma, \hat \Sigma_{n_j}-\Sigma)-{\mathbb E}R_f(\Sigma, \hat \Sigma_{n_j}-\Sigma)|{\mathcal A}_{\rm sym})\Bigr|\Bigr\|_{L_p}
\\
&
\lesssim_m \max_{1\leq j\leq m} \Bigl\|{\mathbb E}\Bigl(\sup_{f\in {\mathcal F}_m}\Bigl|R_f(\Sigma, \hat \Sigma_{n_j}-\Sigma)-{\mathbb E}R_f(\Sigma, \hat \Sigma_{n_j}-\Sigma))\Bigr|\Bigl|{\mathcal A}_{\rm sym}\Bigr)\Bigr\|_{L_p}
\\
&
\lesssim_m \max_{1\leq j\leq m} \Bigl\|\sup_{f\in {\mathcal F}_m}\Bigl|R_f(\Sigma, \hat \Sigma_{n_j}-\Sigma)-{\mathbb E}R_f(\Sigma, \hat \Sigma_{n_j}-\Sigma))\Bigr|\Bigr\|_{L_p}
\end{align*}
and, under the assumption that $n_j\asymp n, j=1,\dots, m,$ the right hand side of the last inequality could be upper bounded using Proposition \ref{conc_F_smooth}.
It is enough to use this proposition along with Theorem \ref{bias_main} and Proposition \ref{conc_F_smooth_Wasser} to complete the proof.

\qed
\end{proof}

\section{Proofs of lower bounds and related results}
\label{proof_lower}

We start with the proof of Proposition \ref{lower_first_term}.

\begin{proof}
In what follows, we will deal only with random variables $X\sim N(0,\Sigma)$ taking values in a fixed subspace $L\subset {\mathbb H}$ with 
${\rm dim}(L)=d=[r].$ In this case, ${\rm Im}(\Sigma)\subset L$ and, assuming that $\|\Sigma\|\leq a,$ we have $\Sigma\in {\mathcal S}(a,r).$ 
Moreover, we can view $\Sigma$ as a covariance operator in space $L$ and, choosing an orthonormal basis $e_1,\dots, e_d$ of $L,$ we can 
identify $L$ with ${\mathbb R}^d$ and view $\Sigma$ as a $d\times d$ covariance matrix. 

We will use a version of the two hypotheses method (see \cite{Tsybakov}, Section 2.3).
Consider two covariance operators in $L:$
$\Sigma_0 := \gamma_1 a I_d$ and $\Sigma_1:= (1+ \frac{c_1}{\sqrt{dn}})\gamma_1 a I_d,$
where $c_1= \frac{\gamma_2-\gamma_1}{\gamma_1}.$
It is immediate that $\Sigma_0, \Sigma_1\in {\mathcal S}(a,r).$

By a simple computation, the following bound 
on the Kullback-Leibler distance between the distributions of i.i.d. observations 
$X_1,\dots, X_n$ sampled from $N(0,\Sigma_0)$ and from $N(0, \Sigma_1)$ holds:
\begin{align*}
K(N(0,\Sigma_0)^{\otimes n}\| N(0, \Sigma_1)^{\otimes n}) &=nK(N(0,\Sigma_0)\| N(0, \Sigma_1)) = 
n\Bigl({\rm tr}(\Sigma_1^{-1}\Sigma_0)-d +\log \frac{{\rm det}(\Sigma_1)}{{\rm det}(\Sigma_0)}\Bigr) 
\\
&
= n d \Bigl(\frac{1}{1+ c_1/\sqrt{nd}} -1 +\log(1+c_1/\sqrt{nd})\Bigr)
\lesssim c_1^2.
\end{align*}
This bound easily implies that minimax error of testing the hypotheses $H_0: \Sigma=\Sigma_0$ against 
the alternative $H_1: \Sigma=\Sigma_1$ is bounded away from zero by a constant depending only on $c_1$
(see \cite{Tsybakov}, Theorem 2.2). 
Note also that $\tau_f(\Sigma_0)= f(\gamma_1 a) d$ and $\tau_f(\Sigma_1)= f(\gamma_1 a (1+c_1/\sqrt{nd})) d.$
Without loss of generality, assume that $f'(x)\geq \lambda, x\in [\gamma_1 a,\gamma_2 a].$
Then,
\begin{align*}
\tau_f(\Sigma_1)-\tau_f(\Sigma_0) = d \int_{\gamma_1 a}^{\gamma_1 a (1+c_1/\sqrt{nd})} f'(x)dx \geq d \lambda \gamma_1 a c_1/\sqrt{nd}
= \lambda \gamma_1 c_1 a \sqrt{\frac{d}{n}}.
\end{align*}
By a standard argument used in the two hypotheses method, this easily implies
\begin{align*}
&
\inf_{T_{n,f}}\sup_{\Sigma\in {\mathcal S}(a,r)}{\mathbb E}_{\Sigma}^{1/2}(T_{n,f}(X_1,\dots, X_n)-\tau_f(\Sigma))^2
\\
&
\geq 
\inf_{T_{n,f}}\max_{\Sigma\in \{\Sigma_0, \Sigma_1\}}{\mathbb E}_{\Sigma}^{1/2}(T_{n,f}(X_1,\dots, X_n)-\tau_f(\Sigma))^2
\gtrsim \lambda \gamma_1 c_1 a \sqrt{\frac{d}{n}},
\end{align*}
completing the proof.

\qed
\end{proof}

Next, we provide the proof of Proposition \ref{upper_piecewise}.

\begin{proof}
First note that
\begin{align} 
\label{start_x}
&
\nonumber
{\mathbb E}_{\Sigma}^{1/2}(\tilde T_g(X_1,\dots, X_n)-g(\Sigma))^2 
\\
&
\nonumber
\leq 
{\mathbb E}_{\Sigma}^{1/2}(\tilde T_g(X_1,\dots, X_n)-g(\Sigma))^2 I(\|\hat \Sigma_n-\Sigma\|< \delta)
+ 
{\mathbb E}_{\Sigma}^{1/2}(\tilde T_g(X_1,\dots, X_n)-g(\Sigma))^2 I(\|\hat \Sigma_n-\Sigma\|\geq \delta)
\\
&
\nonumber
\leq 
{\mathbb E}_{\Sigma}^{1/2}(\tilde T_g(X_1,\dots, X_n)-g(\Sigma))^2 I(\|\hat \Sigma_n-\Sigma\|< \delta)
+ 
{\mathbb E}_{\Sigma}^{1/2}\tilde T_g(X_1,\dots, X_n)^2 I(\|\hat \Sigma_n-\Sigma\|\geq \delta) 
\\
&
\ \ \ \ \ \ \ + |g(\Sigma)| {\mathbb P}_{\Sigma}^{1/2}\{\|\hat \Sigma_n-\Sigma\|< \delta\}.
\end{align} 
Next, observe that, for $\Sigma \in D_g,$ $\|\Sigma-\Sigma_j\|<\delta$ for some (in fact, unique) value of $j.$
If $\|\hat \Sigma_n-\Sigma\|<\delta,$ we have $\|\hat \Sigma_n-\Sigma_j\|<2\delta,$ 
which also holds for the unique 
value of $j.$ Thus, we have that $g(\Sigma)=\tau_{f_j}(\Sigma)$ and $\tilde T_{g}(X_1,\dots, X_n)= \hat T_{f_j}(X_1,\dots, X_n),$ 
implying that, for $\Sigma$ satisfying $\|\Sigma-\Sigma_j\|<\delta,$
\begin{align*}
{\mathbb E}_{\Sigma}^{1/2}(\tilde T_g(X_1,\dots, X_n)-g(\Sigma))^2 I(\|\hat \Sigma_n-\Sigma\|< \delta) 
&
={\mathbb E}_{\Sigma}^{1/2}(\hat T_f(X_1,\dots, X_n)-\tau_{f_j}(\Sigma))^2 I(\|\hat \Sigma_n-\Sigma\|< \delta) 
\\
&
\leq {\mathbb E}_{\Sigma}^{1/2}(\hat T_f(X_1,\dots, X_n)-\tau_{f_j}(\Sigma))^2.
\end{align*}
Therefore, for all $\Sigma\in {\mathcal S}(a,r)\cap D_g,$
\begin{align}
\label{<delta_bd}
{\mathbb E}_{\Sigma}^{1/2}(\tilde T_g(X_1,\dots, X_n)-g(\Sigma))^2 I(\|\hat \Sigma_n-\Sigma\|< \delta) 
&
\nonumber
\leq \max_{1\leq j\leq N}{\mathbb E}_{\Sigma}^{1/2}(\hat T_f(X_1,\dots, X_n)-\tau_{f_j}(\Sigma))^2
\\
&
\lesssim_{m,a} 
\sqrt{\frac{r}{n}} + r \Bigl(\sqrt{\frac{r}{n}}\Bigr)^{m+1},
\end{align}
where we used the bound of Theorem \ref{Th_M_YYY}.

On the other hand note that, for all $\Sigma\in {\mathcal S}(a,r),$ 
\begin{align}
\label{bd_on_g}
|g(\Sigma)| \leq \max_{1\leq j\leq N}|\tau_{f_j}(\Sigma)|\leq \max_{1\leq j\leq N}\|f_j'\|_{L_{\infty}}{\rm tr}(\Sigma) \leq 
 \max_{1\leq j\leq N}\|f_j'\|_{L_{\infty}} a r \lesssim_a r
\end{align}
and 
\begin{align*}
&
|\tilde T_g(X_1,\dots, X_n)|\leq \max_{1\leq j\leq N} |\hat T_{f_j}(X_1,\dots, X_n)|= \max_{1\leq j\leq N} \Bigl|\sum_{i=1}^m C_i\tau_{f_j}(\hat \Sigma_{n_i})\Bigr|
\\
&
\leq \sum_{i=1}^m |C_i| \max_{1\leq j\leq N} \max_{1\leq i\leq m}|\tau_{f_j}(\hat \Sigma_{n_i})| \lesssim_m 
\max_{1\leq j\leq N} \|f_j'\|_{L_{\infty}} \max_{1\leq i\leq m}{\rm tr}(\hat \Sigma_{n_i})
\\
&
\lesssim_m  \max_{1\leq i\leq m}{\rm tr}(\hat \Sigma_{n_i}).
\end{align*}
Therefore,
\begin{align}
\label{bd_on_tilde_T}
&
\nonumber
{\mathbb E}_{\Sigma}^{1/2}\tilde T_g(X_1,\dots, X_n)^2 I(\|\hat \Sigma_n-\Sigma\|\geq \delta) \leq 
{\mathbb E}_{\Sigma}^{1/4}\tilde T_g(X_1,\dots, X_n)^4 {\mathbb P}_{\Sigma}^{1/2}\{\|\hat \Sigma_n-\Sigma\|\geq \delta\}
\\
&
\nonumber
\lesssim_m \|\max_{1\leq i\leq m}{\rm tr}(\hat \Sigma_{n_i})\|_{L_4} {\mathbb P}_{\Sigma}^{1/2}\{\|\hat \Sigma_n-\Sigma\|\geq \delta\}
\lesssim_m \|\max_{1\leq i\leq m}{\rm tr}(\hat \Sigma_{n_i})\|_{L_4} {\mathbb P}_{\Sigma}^{1/2}\{\|\hat \Sigma_n-\Sigma\|\geq \delta\}
\\
&
\nonumber
\lesssim_m \Bigl({\rm tr}(\Sigma)+  \max_{1\leq i\leq m}\|{\rm tr}(\hat \Sigma_{n_i})-{\rm tr}(\Sigma)\|_{L_4}\Bigr) {\mathbb P}_{\Sigma}^{1/2}\{\|\hat \Sigma_n-\Sigma\|\geq \delta\} 
\\
&
\lesssim_m \|\Sigma\|{\bf r}(\Sigma){\mathbb P}_{\Sigma}^{1/2}\{\|\hat \Sigma_n-\Sigma\|\geq \delta\},
\end{align}
where we used the bound of Proposition \ref{trace_conc} and the conditions that  $n_i \asymp n, i=1,\dots, m$ and ${\bf r}(\Sigma)\lesssim n.$

It follows from bounds \eqref{start_x}, \eqref{<delta_bd}, \eqref{bd_on_g} and \eqref{bd_on_tilde_T} that 
\begin{align*} 
&
\nonumber
\sup_{g\in {\mathcal G}_{\delta}}\sup_{\Sigma\in {\mathcal S}(a,r)\cap D_g}{\mathbb E}_{\Sigma}^{1/2}(\tilde T_g(X_1,\dots, X_n)-g(\Sigma))^2 
\\
&
\lesssim_{m,a} 
\sqrt{\frac{r}{n}} + r \Bigl(\sqrt{\frac{r}{n}}\Bigr)^{m+1} + r \sup_{\Sigma\in {\mathcal S}(a,r)}{\mathbb P}_{\Sigma}^{1/2}\{\|\hat \Sigma_n-\Sigma\|\geq \delta\}.
\end{align*}
It remains to use bounds \eqref{KL_1} and \eqref{KL_2} to get, under the conditions $\delta \geq C\sqrt{\frac{r}{n}}$ 
for a sufficiently large numerical constant $C>0,$ $a\lesssim 1$ and $r\leq n,$ that 
\begin{align*}
\sup_{\Sigma\in {\mathcal S}(a,r)}{\mathbb P}_{\Sigma}^{1/2}\{\|\hat \Sigma_n-\Sigma\|\geq \delta\}\leq \exp\{-n(\delta\wedge \delta^2)/2\},
\end{align*}
which implies the claim.

\qed 
\end{proof} 

We now turn to the proof of Proposition \ref{min_max_lower_eff_rank}.

\begin{proof}
As in the proof of Proposition \ref{lower_first_term}, we assume that random variable $X\sim N(0,\Sigma)$ takes values in a fixed subspace 
$L\subset {\mathbb H}$ with ${\rm dim}(L)=d=[r]$ and use a representation of $\Sigma$ as a $d\times d$ covariance matrix. 


To prove the first bound of Proposition \ref{min_max_lower_eff_rank}, one has to construct functionals from the class ${\mathcal G}_{\delta}$
for which a minimax lower bound of order $\sqrt{\frac{d}{n}}$ holds. 
This is done for the functional $g(\Sigma):= {\rm tr} (\Sigma) I_{B(\Sigma_0,\delta)}(\Sigma)$ 
for any $\delta\geq \frac{2c_1}{\sqrt{nd}}$ and the proof of the minimax lower bound is based 
on the two hypotheses method (exactly as in the proof of the lower bound of Proposition \ref{lower_first_term}).


A more difficult argument is needed to construct functionals $g\in {\mathcal G}_{\delta}$ for which a minimax lower bound of the order $d \Bigl(\sqrt{\frac{d}{n}}\Bigr)^{m+1}$ holds. Note that it is enough to prove this bound only for a sufficiently large $d,$ say, $d\geq d_0$ (otherwise, the 
first term $\sqrt{\frac{d}{n}}$ will be dominant up to a constant that depends only on $d_0$). 

First, we construct a sufficiently 
large set $\Sigma_{\omega}, \omega\in B\subset \{-1,1\}^{d(d+1)/2}$ of ``well separated" $d\times d$ covariance matrices. Namely, 
let $A_{\omega}:=(\omega_{ij})_{i,j=1}^d$ be a symmetric matrix with entries $\omega_{ij}=\pm 1,$ and define 
\begin{align*} 
\Sigma_{\omega}:= 
I_d + \eps \Bigl(2I_d+ \frac{2A_{\omega}}{\sqrt{d}}\Bigr), 
\omega=(\omega_{ij}: i\leq j) \in\{-1,1\}^{d(d+1)/2},
\end{align*}   
with $\eps\in (0,1)$ to be chosen later. Assuming that $\omega_{ij}, i\leq j$ are independent Rademacher random variables, $A_{\omega}$
is a symmetric $d\times d$ Bernoulli random matrix. It is well known that, with a high probability, $\|A_{\omega}\|\leq c' \sqrt{d}$ for some 
constant $c'\geq 2.$ Moreover, by Wigner theorem,
\begin{align*}
\frac{1}{d} {\rm tr}\Bigl(\Bigl(2I_d+ \frac{2A_{\omega}}{\sqrt{d}}\Bigr)^{m+1}\Bigr) \overset{{\mathbb P}}{\longrightarrow} 
\int_{-2}^2 (2+x)^{m+1}\mu_{\rm sc}(dx)=:c_m>0\ {\rm as}\ d\to\infty,
\end{align*}
where $\mu_{\rm sc}(dx)=\frac{1}{2\pi}\sqrt{4-x^2}dx$ is Wigner's semicircle law in $[-2,2].$ It follows that there exists a sequence $\delta_d\to 0$ as $d\to\infty$ and a number $\bar d\geq 1$ such that
with probability at least $1/2$ for all $d\geq \bar d$  
\begin{align}
\label{cond_A_1}
\|A_{\omega}\|\leq c'\sqrt{d}
\end{align}
and
\begin{align}
\label{cond_A_2}
\Bigl|\frac{1}{d} {\rm tr}\Bigl(\Bigl(2I_d+ \frac{2A_{\omega}}{\sqrt{d}}\Bigr)^{m+1}\Bigr)-c_m\Bigr|\leq \delta_d.
\end{align}
In other words, for all $d\geq \bar d,$ there exists a subset $B_d'\subset \{-1,1\}^{d(d+1)/2}$ such that ${\rm card}(B_d')\geq 2^{d(d+1)/2-1}$ and, for all $\omega\in B_d',$
bounds \eqref{cond_A_1} and \eqref{cond_A_2} hold. In what follows, we assume that $(c'-1)\eps <1/4.$ This implies that, for all $\omega\in B_d',$ $\Sigma_{\omega}$ is a 
covariance operator with $\|\Sigma_{\omega}\|\lesssim 1$ and $\|\Sigma_{\omega}^{-1}\|\lesssim 1.$   

We can now use standard arguments from the proof of Varshamov-Gilbert lemma to show that there exists $B\subset B_d'$
such that ${\rm card}(B)\geq 2^{d^2/8}$ and 
\begin{align*}
h(\omega,\omega') := \sum_{i\leq j}I(\omega_{ij}\neq \omega_{ij}') \geq d^2/8, \omega,\omega'\in B, \omega\neq \omega'
\end{align*}
(see the proof of Proposition 5.1 in \cite{Koltchinskii_2022} for a similar argument with more details).

Note that, for all $\omega, \omega'\in B, \omega\neq \omega',$
\begin{align*}
\|\Sigma_{\omega}-\Sigma_{\omega'}\|_2^2 &=\Bigl(\frac{2\eps}{\sqrt{d}}\Bigr)^2 \|A_{\omega}-A_{\omega'}\|_2^2= 
\frac{4\eps^2}{d} \sum_{i,j}(\omega_{ij}-\omega_{ij}')^2 \geq \frac{4\eps^2}{d} \sum_{i\leq j}(\omega_{ij}-\omega_{ij}')^2 
\\
&
= \frac{16\eps^2}{d} h(\omega, \omega')
\geq \frac{16\eps^2}{d} \frac{d^2}{8} = 2\eps^2 d
\end{align*}
and 
\begin{align*}
\|\Sigma_{\omega}-\Sigma_{\omega'}\|\geq \frac{1}{\sqrt{d}}\|\Sigma_{\omega}-\Sigma_{\omega'}\|_2 \geq \sqrt{2}\eps.
\end{align*}

We will need the following lemma. 

\begin{lemma}
\label{min_max_omega}
Let $X_1,\dots, X_n$ be i.i.d. random variables sampled from $N(0,\Sigma_{\omega}), \omega\in B.$
Suppose that $\eps \leq c_1\sqrt{\frac{d}{n}}$ for a small enough constant $c_1>0.$
Then 
\begin{align*}
\inf_{\hat \omega}\max_{\omega\in B}{\mathbb E}_{\Sigma_{\omega}}\|\Sigma_{\hat \omega}- \Sigma_{\omega}\|_2^2 \gtrsim \eps^2 d,
\end{align*}
where the infimum is taken over all the estimators $\hat \omega(X_1,\dots, X_n)$ of parameter $\omega$ based on $X_1,\dots, X_n.$
\end{lemma}

The proof relies on a well known approach to minimax lower bounds based on many hypotheses
(see \cite{Tsybakov}, Section 2.6). 

In what follows, $\eps := c_1\sqrt{\frac{d}{n}}$ for a small enough constant $c_1>0$
(as in Lemma \ref{min_max_omega}) and we also assume that $(c'-1)\eps <1/4.$ 

\begin{lemma}
\label{tr_f_repr}
Let $f\in C^{m+2}({\mathbb R})$ be a function such that $f(0)=0,$ $f(1)=\dots=f^{(m)}(1)=0$ and $f^{(m+1)}(1)\neq 0.$
Suppose also that $\|f'\|_{L_{\infty}}\leq 1,$ $\|f''\|_{L_{\infty}}\leq 1$ and $\|f^{(m+2)}\|_{L_{\infty}}\leq 1.$ Then, for all $d\geq \bar d$
and all $\omega\in B,$ 
\begin{align*}
|\tau_f(\Sigma_{\omega}) - c_m(f)\eps^{m+1} d| \lesssim_m \eps^{m+1}d (\delta_d+ \eps),
\end{align*}
where 
\begin{align*}
c_m(f):=  c_m \frac{f^{(m+1)}(1)}{(m+1)!}.
\end{align*}
\end{lemma}

\begin{proof}
Clearly, 
\begin{align*}
\tau_f (\Sigma_{\omega})= {\rm tr}(f(\Sigma_{\omega})) = {\rm tr}\Bigl(f\Bigl(I_d + \eps\Bigl(2I_d+ \frac{2A_{\omega}}{\sqrt{d}}\Bigr)\Bigr)\Bigr).
\end{align*}
Note that 
\begin{align*}
f(1+ \eps(2+x))= \sum_{j=0}^{m+1} \frac{f^{(j)}(1)}{j!}  \eps^j (2+x)^j + R(x)=  \frac{f^{(m+1)}(1)}{(m+1)!}  \eps^{m+1} (2+x)^{m+1}+ R(x),
\end{align*}
where the following bound holds for the remainder $R(x):$
\begin{align*}
|R(x)|\leq \frac{\|f^{(m+2)}\|_{L_{\infty}}}{(m+2)!} \eps^{m+2} |2+x|^{m+2}.
\end{align*}
Therefore,
\begin{align*}
\tau_f (\Sigma_{\omega})&= \sum_{k=1}^d f\Bigl(1+ \eps\Bigl(2+ \lambda_k\Bigl(\frac{2A_{\omega}}{\sqrt{d}}\Bigr)\Bigr)\Bigr)
= \frac{f^{(m+1)}(1)}{(m+1)!} \eps^{m+1} d \frac{1}{d}{\rm tr}\Bigl(2I_d+ \frac{2A_{\omega}}{\sqrt{d}}\Bigr)^{m+1}
\\
&
= c_m \frac{f^{(m+1)}(1)}{(m+1)!} \eps^{m+1} d  + \rho = c_m(f) \eps^{m+1} d  + \rho
\end{align*}
where 
\begin{align*}
|\rho| &\leq \frac{|f^{(m+1)}(1)|}{(m+1)!} \eps^{m+1} d \delta_d + \frac{\|f^{(m+2)}\|_{L_{\infty}}}{(m+2)!} \eps^{m+2} d \max_{1\leq k\leq d}
\Bigl|2+\lambda_k\Bigl(\frac{2A_{\omega}}{\sqrt{d}}\Bigr)\Bigr|^{m+2} 
\\
&
\leq \frac{|f^{(m+1)}(1)|}{(m+1)!} \eps^{m+1} d \delta_d + 2^{m+2}(1+c')^{m+2}\frac{\|f^{(m+2)}\|_{L_{\infty}}}{(m+2)!} \eps^{m+2} d 
\\
&
\lesssim_m \eps^{m+1} d (\delta_d+\eps).
\end{align*}

\qed
\end{proof}

Assume now that $f\in C^{m+2}({\mathbb R})$ satisfies the conditions of Lemma \ref{tr_f_repr} and $f^{(m+1)}(1)=b>0,$ $b\asymp 1.$
Let $\delta:=\eps/4$ and define 
\begin{align*}
g_{ij}(\Sigma):= \sum_{\omega\in B} I_{B(\Sigma_{\omega},\delta)}(\Sigma) \tau_{\omega_{ij}f}(\Sigma), \Sigma\in D_{g_{ij}}, 1\leq i\leq j\leq d.
\end{align*}
Clearly, $g_{ij}\in {\mathcal G_{\delta}}, i\leq j$ and 
\begin{align}
\label{g_ij_rep}
g_{ij}(\Sigma_{\omega}) = \tau_{\omega_{ij}f}(\Sigma_{\omega})= \omega_{ij}\bar c_m \eps^{m+1} d + \rho_{ij}, i\leq j,
\end{align}
where $\bar c_m =\frac{c_m b}{(m+1)!}$ and 
\begin{align}
\label{up_g_ij}
\max_{i\leq j} 
|\rho_{ij}| 
\lesssim_m \eps^{m+1} d (\delta_d+\eps).
\end{align}

Suppose now that 
\begin{align*}
\sup_{g\in {\mathcal G}_{\delta}}\inf_{T_{n,g}}\sup_{\Sigma\in {\mathcal S}(a,r)\cap D_g}
{\mathbb E}_{\Sigma}^{1/2}(T_{n,g}(X_1,\dots, X_n)-g(\Sigma))^2 < \Delta.
\end{align*}
Then, for all $i\leq j,$ there exist estimators $\tilde T_{ij}(X_1,\dots, X_n)$ of functionals $g_{ij}(\Sigma)$ such that 
\begin{align*}
\max_{i\leq j}\sup_{\Sigma\in {\mathcal S}(a,r)\cap D_{g_{ij}}}
{\mathbb E}_{\Sigma}^{1/2}(\tilde T_{ij}(X_1,\dots, X_n)-g_{ij}(\Sigma))^2 < \Delta,
\end{align*}
which implies that 
\begin{align}
\label{main_tilde_ij}
\max_{i\leq j}\max_{\omega\in B}
{\mathbb E}_{\Sigma_{\omega}}^{1/2}(\tilde T_{ij}(X_1,\dots, X_n)-g_{ij}(\Sigma_{\omega}))^2 < \Delta.
\end{align}
Denote $\tilde \omega_{ij}:= {\rm sign}(\tilde T_{ij}(X_1,\dots, X_n))$
and let
\begin{align*}
\hat \omega \in {\rm Argmin}_{\omega \in B} \sum_{i,j}(\tilde \omega_{ij}-\omega_{ij})^2.
\end{align*}
Then, for all $\omega\in B,$ 
\begin{align*}
\Bigl(\sum_{i,j}(\hat \omega_{ij}-\omega_{ij})^2\Bigr)^{1/2} \leq \Bigl(\sum_{i,j}(\hat \omega_{ij}-\tilde \omega_{ij})^2\Bigr)^{1/2}
+\Bigl(\sum_{i,j}(\tilde \omega_{ij}-\omega_{ij})^2\Bigr)^{1/2} \leq 2 \Bigl(\sum_{i,j}(\tilde \omega_{ij}-\omega_{ij})^2\Bigr)^{1/2}.
\end{align*}
Therefore,
\begin{align}
\label{HS_hat_om}
&
\nonumber
\|\Sigma_{\hat \omega}-\Sigma_{\omega}\|_2^2 = 
\frac{4\eps^2}{d} \sum_{i,j}(\hat \omega_{ij}-\omega_{ij})^2
\leq \frac{16\eps^2}{d} \sum_{i,j}(\tilde \omega_{ij}-\omega_{ij})^2
\\
&
= \frac{16\eps^2}{d}\frac{1}{\bar c_m^2 \eps^{2(m+1)}} \frac{1}{d^2} 
\sum_{i,j}(\bar c_m \eps^{m+1} d\ \tilde \omega_{ij}-\bar c_m \eps^{m+1} d\ \omega_{ij})^2
\end{align}
Note that 
\begin{align}
\label{tildeT_omega}
|\bar c_m \eps^{m+1} d\ \tilde \omega_{ij}-\bar c_m \eps^{m+1} d\ \omega_{ij}| \leq 2|\tilde T_{ij}(X_1,\dots, X_n)-\bar c_m \eps^{m+1} d\ \omega_{ij}|.
\end{align}
Indeed, suppose $\tilde \omega_{ij} =+1$ and $\omega_{ij}=-1.$ Then, 
\begin{align*}
|\bar c_m \eps^{m+1} d\ \tilde \omega_{ij}-\bar c_m \eps^{m+1} d\ \omega_{ij}| = 2\bar c_m \eps^{m+1} d.
\end{align*}
On the other hand, $\tilde \omega_{ij}=+1$ implies that $\tilde T_{ij}(X_1,\dots, X_n)\geq 0$ and 
\begin{align*}
|\tilde T_{ij}(X_1,\dots, X_n)-\bar c_m \eps^{m+1} d\ \omega_{ij}| = \tilde T_{ij}(X_1,\dots, X_n)+ \bar c_m \eps^{m+1} d 
\geq \bar c_m \eps^{m+1} d = \frac{1}{2}|\bar c_m \eps^{m+1} d\ \tilde \omega_{ij}-\bar c_m \eps^{m+1} d\ \omega_{ij}|,
\end{align*}
which implies \eqref{tildeT_omega}. Similar argument holds when $\tilde \omega_{ij}=-1, \omega_{ij}=+1.$

Note also that, by \eqref{g_ij_rep} and \eqref{up_g_ij}, 
\begin{align*}
&
|\tilde T_{ij}(X_1,\dots, X_n)-\bar c_m \eps^{m+1} d\ \omega_{ij}| \leq 
|\tilde T_{ij}(X_1,\dots, X_n)-g_{ij}(\Sigma_{\omega})| + |g_{ij}(\Sigma_{\omega})-\bar c_m \eps^{m+1} d\ \omega_{ij}|
\\
&
\leq |\tilde T_{ij}(X_1,\dots, X_n)-g_{ij}(\Sigma_{\omega})|  + c_m' \eps^{m+1}d (\delta_d+\eps).
\end{align*}
Using the last bound along with \eqref{HS_hat_om} and \eqref{tildeT_omega}, we get 
\begin{align*}
\|\Sigma_{\hat \omega}-\Sigma_{\omega}\|_2^2 &\leq 
\frac{16\eps^2}{d}\frac{1}{\bar c_m^2 \eps^{2(m+1)}} \frac{1}{d^2} 
\sum_{i,j}(\bar c_m \eps^{m+1} d\ \tilde \omega_{ij}-\bar c_m \eps^{m+1} d\ \omega_{ij})^2
\\
&
\leq 
\frac{2^6\eps^2}{d}\frac{1}{\bar c_m^2 \eps^{2(m+1)}} \frac{1}{d^2} 
\sum_{i,j} (\tilde T_{ij}(X_1,\dots, X_n)-\bar c_m \eps^{m+1} d\ \omega_{ij})^2
\\
&
\leq 
\frac{2^7\eps^2}{d}\frac{1}{\bar c_m^2 \eps^{2(m+1)}} \frac{1}{d^2} 
\sum_{i,j} (\tilde T_{ij}(X_1,\dots, X_n)-g_{ij}(\Sigma_{\omega}))^2
\\
&
+\frac{2^7\eps^2}{d}\frac{1}{\bar c_m^2 \eps^{2(m+1)}}(c_m')^2 \eps^{2(m+1)}d^2 (\delta_d+\eps)^2
\\
&
\leq 
\frac{2^7\eps^2}{d}\frac{1}{\bar c_m^2 \eps^{2(m+1)}} \frac{1}{d^2} 
\sum_{i,j} (\tilde T_{ij}(X_1,\dots, X_n)-g_{ij}(\Sigma_{\omega}))^2 + c_m'' \eps^2 d (\delta_d+\eps)^2,
\end{align*}
where $c_m'':= \frac{2^7 (c_m')^2}{\bar c_m^2}.$
Therefore,
\begin{align*}
&
{\mathbb E}_{\Sigma_{\omega}}\|\Sigma_{\hat \omega}-\Sigma_{\omega}\|_2^2 
\\
&
\leq 
\frac{2^7\eps^2}{d}\frac{1}{\bar c_m^2 \eps^{2(m+1)}} \frac{1}{d^2} 
\sum_{i,j} {\mathbb E}_{\Sigma_{\omega}}(\tilde T_{ij}(X_1,\dots, X_n)-g_{ij}(\Sigma_{\omega}))^2 + c_m'' \eps^2 d (\delta_d+\eps)^2,
\end{align*}
and, by bound \eqref{main_tilde_ij},
\begin{align*}
\max_{\omega\in B}{\mathbb E}_{\Sigma_{\omega}}\|\Sigma_{\hat \omega}-\Sigma_{\omega}\|_2^2 
\leq \frac{2^7\eps^2}{d}\frac{1}{\bar c_m^2 \eps^{2(m+1)}}\Delta^2 +c_m'' \eps^2 d (\delta_d+\eps)^2.
\end{align*}
By Lemma \ref{min_max_omega}, we get that 
\begin{align*}
\eps^2 d  \leq C'\Bigl(\frac{\eps^2}{d}\frac{1}{\bar c_m^2 \eps^{2(m+1)}}\Delta^2 +c_m'' \eps^2 d (\delta_d+\eps)^2\Bigr)
\end{align*}
with a sufficiently large constant $C'>0.$ Since also $\delta_d\to 0$ as $d\to \infty,$ we can choose $d_0\geq \bar d$ so that, for all 
$d\geq d_0,$ $\delta_d^2\leq \frac{1}{8C'c_{m}''}.$ If, in addition,  $\eps^2= c_1^2\frac{d}{n} \leq \frac{1}{8C'c_{m}''}$ (which holds 
if constant $c_1$ is small enough), we can conclude 
that $\Delta \gtrsim \eps^{m+1} d$ for all $d\geq d_0,$
which implies that 
\begin{align*}
\sup_{g\in {\mathcal G}_{\delta}}\inf_{T_{n,g}}\sup_{\Sigma\in {\mathcal S}(a,r)\cap D_g}
{\mathbb E}_{\Sigma}^{1/2}(T_{n,g}(X_1,\dots, X_n)-g(\Sigma))^2 \gtrsim d\Bigl(\sqrt{\frac{d}{n}}\Bigr)^{m+1},
\end{align*}
completing the proof.

\qed
\end{proof}

\end{document}